\documentclass[leqno,a4paper,11pt]{article}

\usepackage{hyperref}
\usepackage{amsthm}
\usepackage{amsmath}
\usepackage{amssymb}

\usepackage{mathtools}
\mathtoolsset{showonlyrefs}
\usepackage{enumitem}
\usepackage{multirow}
\usepackage[normalem]{ulem}

\usepackage{tikz}
\usetikzlibrary{arrows,angles}

\newcommand{\Con}{\ensuremath{\mathcal{C}}}

\newcommand{\D}{\ensuremath{{\mathcal D}}}
\newcommand{\mb}[1]{\ensuremath{\mathbb{#1}}}
\newcommand{\N}{\mb{N}}

\newcommand{\R}{\mb{R}}

\newcommand{\sgn}{\mathop{\mathrm{sgn}}}

\newcommand{\lara}[1]{\langle #1 \rangle}

\newfont{\bl}{msbm10 scaled \magstep2}

\newcommand{\beq}{\begin{equation}}
\newcommand{\eeq}{\end{equation}}
\newcommand{\notmid}{\mid\kern-0.5em\not\kern0.5em}

\newcommand{\eps}{\varepsilon}

\newenvironment{pr}{\begin{proof}[\textbf{Proof:}] \ }{\end{proof}}
\newtheorem{thm}{Theorem}[section]
\newtheorem{lem}[thm]{Lemma}
\newtheorem{prop}[thm]{Proposition}
\newtheorem{ex}[thm]{Example}
\newtheorem{cor}[thm]{Corollary}
\newtheorem{rem}[thm]{Remark}
\newtheorem{defi}[thm]{Definition}

\newcommand{\LLS}{Lorentzian length space }
\newcommand{\LLSn}{Lorentzian length space}
\newcommand{\LpLS}{Lorentzian pre-length space }
\newcommand{\LpLSn}{Lorentzian pre-length space}
\newcommand{\Xll}{\ensuremath{(X,d,\ll,\leq,\tau)} }

\newtheorem{pop}[thm]{Proposition}

\renewcommand{\labelenumi}{(\roman{enumi})}

\newcommand{\mi}{\R^2_1}
\newcommand{\lm}[1]{\mathbb{L}^2(#1)}
\newcommand{\ma}{\measuredangle}
\newcommand{\ct}{\Delta \bar x \bar y \bar z}
\newcommand{\arcosh}{\mathrm{arcosh}}

\newcommand{\mc}[1]{\mathcal{#1}}

\newcommand{\widebar}[1]{\overline{#1}}
\usepackage{accents}
\newcommand{\ubar}[1]{\underaccent{\bar}{#1}}

\renewcommand{\labelenumi}{(\roman{enumi})}
\renewcommand\theenumi\labelenumi

\newtheorem{ithm}{Theorem}[section]

\title{Hyperbolic angles in Lorentzian length spaces and timelike curvature bounds}
\author{Tobias Beran\thanks{{\tt tobias.beran@univie.ac.at}, Faculty of Mathematics, University of Vienna, Austria.}\ \ and Clemens S\"amann\thanks{{\tt clemens.saemann@univie.ac.at}, Faculty of Mathematics, University of Vienna, Austria. Current address: {\tt clemens.saemann@maths.ox.ac.uk} Mathematical Institute, University of Oxford, UK.}}

\begin{document}
 \maketitle
 
 \begin{abstract}
Within the synthetic-geometric framework of Lorentzian (pre-)len\-gth spaces developed in Kunzinger and S\"amann (Ann.\ Glob.\ Anal.\ Geom.\ 54(3):399--447, 2018) we introduce a notion of a hyperbolic angle, an angle between timelike curves and related concepts like timelike tangent cone and exponential map. This provides valuable technical tools for the further development of the theory and paves the way for the main result of the article, which is the characterization of timelike curvature bounds (defined via triangle comparison) with an angle monotonicity condition. Further, we improve on a geodesic non-branching result for spaces with timelike curvature bounded below.
\bigskip

\medskip

\noindent
\emph{MSC2020:}
28A75, %Length, area, volume, other geometric measure theory 
51K10, %Synthetic differential geometry
53C23, %Global geometric and topological methods (à la Gromov); differential geometric analysis on metric spaces
53C50, %Lorentz manifolds, manifolds with indefinite metrics 
53B30, %Lorentz metrics, indefinite metrics
53C80, %Applications to physics 
83C99 %None of the above, but in this section (83C GR)

\end{abstract}

 \tableofcontents

\section{Introduction}

In the theory of (metric) length spaces, Alexandrov spaces and CAT$(K)$-spaces angles play an important role. They are an indispensable tool and in particular allow one to characterize curvature bounds with a monotonicity condition on angles.
Thus a fruitful way forward in the theory of \emph{Lorentzian (pre-)length spaces} \cite{KS:18} is to introduce the notion of an hyperbolic angle into the Lorentzian synthetic setting. As compared to the metric case, several difficulties have to be overcome, foremost the issue of keeping track of the time orientation and causal relations of the points and curves considered. Thus several proofs become more involved than one would expect from the metric case.
\medskip

Lorentzian (pre-)length spaces have been developed to allow one to define timelike or causal curvature bounds for spacetimes of low regularity or for situations where one does not even have a Lorentzian metric, like in some approaches to Quantum Gravity, e.g.\ the causal set approach \cite{BLMS:87} or causal Fermion systems \cite{Fin:18} (see also \cite[Subsec.\ 5.3]{KS:18}). Triangle comparison theorems in Lorentzian geometry were pioneered by Harris in \cite{Har:82} and triangle comparison (to model spaces) is the key concept to define timelike and causal curvature bounds in \cite{KS:18}. The study of spacetimes (and cone structures) of low regularity is a very active field with connections to General Relativity and in particular to singularity theorems, the cosmic censorship conjecture and (in-)extendibility  (cf.\ \cite{CG:12, FS:12, Sae:16, GL:17, Sbi:18, GGKS:18, GLS:18, Min:19a, GKSS:20, Gra:20,LLS:21,KOSS:22}). Questions of (in-)extendibility of spacetimes can be formulated in the setting of \LLSn s and this approach enabled \cite{GKS:19}, for the first time, to relate inextendibility to a blow-up of (synthetic) curvature.

Furthermore, as pioneered by Kronheimer and Penrose \cite{KP:67} and Busemann \cite{Bus:67}, causality theory and (parts of) Lorentzian geometry can be studied decoupled from the manifold structure \cite[Subsec.\ 3.5]{KS:18}, \cite{ACS:20, BGH:21}. A large class of \LLSn s have been investigated in \cite{AGKS:21}, where warped products of a one-dimensional base with metric length space fibers are considered. These so-called \emph{generalized cones} have particularly nice causal properties. Moreover, one can relate Alexandrov curvature bounds of the fibers to timelike curvature bounds of the entire space and vice versa, and a first synthetic singularity theorem was established in this setting. Recently, the gluing of \LpLSn s was developed in \cite{BR:22} and applied to gluing of spacetimes, thereby establishing an analogue of a result by Reshetnyak for CAT$(K)$-spaces, which roughly states that gluing is compatible with upper curvature bounds.

Cavalletti and Mondino \cite{CM:20} took the field even further by introducing a synthetic notion of Ricci curvature bounds by using techniques from optimal transport analogous to the Lott-Villani-Sturm theory of metric measure spaces with Ricci curvature bounded below \cite{LV:09,Stu:06a,Stu:06b}. They built on work by McCann \cite{McC:20} and Mondino-Suhr \cite{MS:22}, who characterize timelike Ricci curvature bounds in terms of synthetic notions \`a la Lott-Villani-Sturm for smooth spacetimes. A related topic was then to find canonical measures in the Lorentzian setting akin to Hausdorff measures in the metric one. This has been achieved recently in \cite{McCS:22}, where also a synthetic dimension for \LpLSn s is given, the compatibility of the Lorentzian measures with the volume measure of continuous spacetimes is shown and its relation to \cite{CM:20} is studied.

Finally, there is the related approach of Sormani and Vega \cite{SV:16} that introduces a metric on Lorentzian manifolds with a suitable time function, the so-called \emph{null distance}. Its relation to spacetime convergence has been explored in \cite{AB:22}, while in \cite{KS:22} the null distance is introduced on \LLSn s and a kind of compatibility between the two approaches is established.
\medskip

In the final stages of preparing the article we have been made aware of the preprint \cite{BMS:22} by Barrera, Montes de Oca and Solis, where the authors also introduce angles in \LLSn s. Compared to our work they are investigating only curvature bounds from below but discuss an angle comparison condition and first variation formula, which we do not. On the other hand in our more comprehensive study we consider bounds from below and above, exponential and logarithmic map, triangle inequality of angles and equivalence of monotonicity comparison and timelike curvature bounds. Thus our two works nicely complement each other and their simultaneous appearance is a clear indication that this direction of research will prove fruitful.
\bigskip

In the following subsection we present the main results of our article and outline the structure of the paper.

\subsection{Main results and outline of the article}

The plan of the paper is as follows. In Subsection \ref{subsec-not} we introduce the relevant background on Lorentzian 
geometry and fix some notations and conventions. Then, to conclude the introduction we give a 
brief review of the theory of \LLSn s in Subsection \ref{subsec-lls} and of curvature comparison in Subsection \ref{subsec-cur-com}.

In Section \ref{sec-ang} we introduce angles in \LpLSn s in timelike triangles and between timelike curves. A main tool is the Lorentzian version of the law of cosines (Lemma \ref{lorLawOfCosines}) and corollaries thereof. As a compatibility check we establish that in a smooth and strongly causal spacetime the angle between continuously differentiable timelike curves agrees with the synthetic angle in the sense of \LpLSn s. We also show that one can alternatively use $K$-comparison angles for the definition of angles (for any $K\in\R$).

After these technical preparations we study in detail angles between timelike curves that have the same time orientation in Section \ref{sec-ang-same-to}. A central result is the triangle inequality.

\setcounter{section}{3}
\setcounter{ithm}{0}
\begin{ithm}[Triangle inequality for (upper) angles]
 Let \Xll be a strongly causal and locally causally closed \LpLS with $\tau$ locally finite-valued and locally continuous. Let $\alpha,\beta,\gamma\colon[0,B)\rightarrow X$ be timelike curves with coinciding time orientation starting at $x:=\alpha(0)=\beta(0)=\gamma(0)$. Then
 \begin{equation}
  \ma_x(\alpha,\gamma)\leq \ma_x(\alpha,\beta) + \ma_x(\beta,\gamma)\,.
 \end{equation}
\end{ithm}

Two timelike curves of the same time orientation and starting at the same point are defined to be equivalent if they have zero angle. Together with the triangle inequality this makes the space of (future/past) directed timelike directions at a point into a metric space. This allows one to define the \emph{timelike tangent cone} as the Minkowski cone over the metric space of (future/past) directed timelike directions. Also, one can define an exponential and a logarithmic map and we establish foundational properties and relations of these objects.

Moreover, we introduce a monotonicity condition for angles at the past/ future endpoint of a timelike triangle, called \emph{future/past $K$-monotonicity comparison}. This gives then a bound on the angle between timelike geodesics of the same time orientation (Corollary \ref{cor-K-mon-ang-bou}), and that angles exist, if finite (Lemma \ref{lem-K-ang-com-ang-ex-fudi}). Another central result is that timelike curvature bounds imply future/past $K$-monotonicity comparison as follows.

\setcounter{ithm}{9}
\begin{ithm}[Triangle comparison implies future $K$-monotonicity comparison]
 Let $\Xll$ be a locally strictly timelike geodesically connected Lor\-entzian pre-length space and let $K\in\R$. If $X$ has timelike curvature bounded below (above) by $K$, it also satisfies future $K$-monotonicity comparison from below (above).
\end{ithm}

In Section \ref{sec-ang-arb-to} we establish the triangle inequality between timelike curves of arbitrary time orientation in Theorem \ref{thm-ang-tri-equ-oth-cas}. Moreover, we improve upon a geodesic non-branching result of \cite{KS:18}, where we do not need any additional hypotheses, as follows.

\setcounter{section}{4}
\setcounter{ithm}{7}
\begin{ithm}[Timelike non-branching]
Let $X$ be a strongly causal \LpLS with timelike curvature bounded below by some $K\in\mb{R}$. Then timelike distance realizers cannot branch.
\end{ithm}

Finally, in Subsection \ref{subsec-ang-com}, we define a general $K$-monotonicity comparison condition (Definition \ref{def-ang-com}), show that it implies the existence of angles in Lemma \ref{lem-K-ang-com-ang-ex} and, this being the main result of the article, that it characterizes timelike curvature bounds (Definition \ref{def-tri-com}) as follows.

\setcounter{ithm}{12}
\begin{ithm}[Equivalence of triangle and monotonicity comparison]
 Let $\Xll$ be a locally strictly timelike geodesically connected \LpLS and let $K\in\R$. Then $X$ has timelike curvature bounded below (above) by $K$ if and only if it satisfies $K$-monotonicity comparison from below (above).
\end{ithm}

The proof of the Lorentzian law of cosines Lemma \ref{lorLawOfCosines} and of a related lemma is outsourced to the Appendix \ref{PrLorLOC}.

\setcounter{section}{1}

\subsection{Notation and conventions}\label{subsec-not}

We fix some notation and conventions as follows. One of our main examples are \emph{spacetimes} $(M,g)$, where the Lorentzian metric $g$ is of different regularity classes. Here $M$ denotes a smooth, connected, second countable Hausdorff manifold and the Lorentzian metric $g$ on $M$ is continuous, $\mathcal{C}^2$ or smooth. Note that our convention is that $g$ is of signature $(-+++\ldots)$. A vector $v\in TM$ is called
\begin{align}
%\left.
\begin{cases}
 \text{\emph{timelike}}\\
 \text{\emph{null}}\\
 \text{\emph{causal}}\\
 \text{\emph{spacelike}}
\end{cases}
%\right\}
 \text{\quad if \qquad}
 g(v,v)\qquad
 \begin{cases}
  < 0\,,\\
  = 0 \text{ and } v\neq 0\,,\\
  \leq 0\text{ and } v\neq 0\,,\\
  >0\text{ or } v=0\,.
 \end{cases}
\end{align} 
Moreover, we assume that $(M,g)$ is time-oriented (i.e., there exists a continuous timelike vector field $\xi$, that is, $g(\xi,\xi)<0$ everywhere).
We call $(M,g)$ a \emph{continuous/$\mathcal{C}^2$-/smooth spacetime}. Furthermore, a {causal} vector $v\in TM$ is \emph{future/past directed} if $g(v,\xi)<0$ (or $g(v,\xi)>0$, respectively), where $\xi$ is the global timelike vector field giving the time orientation of the spacetime $(M,g)$. Analogously, one defines the causal character and time orientation of sufficiently smooth curves into $M$. The \emph{(Lorentzian) length} $L^g(\gamma)$ of a causal curve $\gamma\colon[a,b]\rightarrow M$ is defined as $L^g(\gamma):=\int_a^b \sqrt{-g(\dot\gamma,\dot\gamma)}$ and the \emph{time separation function} is defined as follows. For $x,y\in M$ set $\tau(x,y):=\sup\, \{ L^g(\gamma): \gamma \text{ f.d.\ causal from } x \text{ to } y\}\cup\{0\}$.

Two events $x,y\in M$ are \emph{timelike} related if there is a future directed timelike curve from $x$ to $y$, denoted by 
$x\ll y$. Analogously, $x\leq y$ if there is a future directed causal curve from $x$ to $y$ or $x=y$. The spacetime $(M,g)$ 
is \emph{strongly causal} if for every point $p\in M$ and for every neighborhood $U$ of $p$ there is a neighborhood $V$ of 
$p$ such that $V\subseteq U$ and all causal curves with endpoints in $V$ are contained in $U$, (in which case $V$ is 
called \emph{causally convex} in $U$).
\medskip

Finally, let us recall the notion of sectional curvature on a smooth semi-Riemannian manifold.
\begin{defi}[Sectional curvature]
Let $(M,g)$ be a semi-Riemannian manifold, $p\in M$ a point. A non-degenerate subspace $P\subseteq T_pM$ is \emph{non-degenerate} if $(g|_p)|_P:P\times P\to\R$ is a non-degenerate symmetric bilinear map.
Let $P\subseteq T_pM$ be a non-degenerate $2$-dimensional subspace (or plane). Let $v,w\in P$ be linearly independent. Let $R$ be the Riemann curvature tensor of $g$, then the \emph{sectional curvature} of $M$ at $p$ on the plane $P$ is $\frac{g(R(v,w)w,v)}{g(v,v)g(w,w)-g(v,w)^2}$.
\end{defi}

\subsection{A brief introduction to Lorentzian (pre-)length spaces}\label{subsec-lls}

Kronheimer and Penrose \cite{KP:67} studied causality theory from an abstract point-of-view and their basic object is a so-called \emph{causal space}. Here we use a slightly more general version of this.

\begin{defi}[Causal space]
A set $X$ with two binary relations $\ll,\leq$, where both are transitive, $\leq$ is reflexive, and any timelike related pairs $p\ll q$ are also causally related $p\leq q$ (sometimes denoted as ${\ll}\subseteq{\leq}$), is called a \emph{causal space}.
\end{defi}

We define the chronological and causal futures and pasts and the chronological and causal diamonds as follows.
\begin{itemize}
\item $I^+(p):=\{q\in X:p\ll q\}$, $I^-(q):=\{p\in X:p\ll q\}$,
\item $J^+(p):=\{q\in X:p\leq q\}$, $J^-(q):=\{p\in X:p\leq q\}$,
\item $I(p,q):=I^+(p)\cap I^-(q)$, $J(p,q):=J^+(p)\cap J^-(q)$.
\end{itemize}

A \LpLS generalizes spacetimes by taking the causal relations and the time separation function as its fundamental objects, thereby foregoing the smooth structure of the manifold completely.

\begin{defi}[\LpLSn]
A \emph{\LpLSn} $\Xll$ is a causal space $(X,\ll,\leq)$ together with a metric $d$ on $X$ and a map $\tau:X\times X\to[0,\infty]$ satisfying
\begin{itemize}
\item $\tau$ is lower semicontinuous (with respect to $d$),
\item $\tau(p,r)\geq\tau(p,q)+\tau(q,r)$ for $p\leq q\leq r$ (reverse triangle inequality) and
\item $\tau(p,q)>0\Leftrightarrow p\ll q$.
\end{itemize}
\end{defi}

To introduce the notion of \emph{intrinsic} time separation functions, we define the length of curves as follows.
\begin{defi}[Causal, timelike and null curves]
Let $(X,\ll,\leq,d,\tau)$ be a \LpLSn.
% A curve $\gamma$ in $X$ is just a curve in the metric space $(X,d)$.
A non-constant locally Lipschitz continuous curve $\gamma\colon I\rightarrow X$ is \emph{future directed causal} or \emph{future directed timelike} if $\forall t_1,t_2\in I, t_1<t_2,\; \gamma(t_1)\ll\gamma(t_2)$ or $\gamma(t_1)\leq\gamma(t_2)$, respectively. The \emph{causal character} of a future directed causal curve is \emph{timelike} if it is timelike and \emph{null} if $\forall t_1,t_2\in I, t_1<t_2,\; \gamma(t_1)\not\ll\gamma(t_2)$.
% Otherwise, it is of no causal character.

For past directed causal and past directed timelike, we reverse these relations. Upon parameter reversal, they are future directed causal / timelike.
\end{defi}

\begin{defi}[Length of curves]
Let $(X,\ll,\leq,d,\tau)$ be a \LpLSn. 
We can define the length of future directed causal curves as in metric spaces, replacing the supremum with an infimum in the definition of the variational length.  We define the \emph{length} $L_\tau(\gamma)=\inf\{\sum_i\tau(\gamma(t_i),\gamma(t_{i+1})):(t_i)\text{ a partition of } I\}$. %We also define $L_\tau(\gamma,s,t)=L_\tau(\gamma|_{[s,t]})$ for $s\leq t$ and $L_\tau(\gamma,s,t)=-L_\tau(\gamma,t,s)$ for $s\geq t$ (which agrees for $t=s$ if $\tau(\gamma(t),\gamma(t))=0$).
%For past causal curves, we introduce a minus into the length.
\end{defi}

One of the main examples of an \LpLS is a spacetime $(M,g)$ together with its timelike and causal relations $\ll,\leq$, its time separation function $\tau$ and a metric induced by a complete Riemannian background metric, cf.\ \cite[Ex.\ 2.11]{KS:18}. Moreover, spacetimes of low regularity (i.e., with continuous metric and sufficiently well-behaved causality) and Lorentz-Finsler spaces are \LpLSn s as well, cf.\ \cite[Prop.\ 5.8, Prop.\ 5.14]{KS:18}.

\begin{defi}[Intrinsic space]
A \LpLS is \emph{strictly intrinsic} or \emph{geodesic} if for all $p< q$ there exists a future directed causal curve $\gamma$ from $p$ to $q$ of length $L_\tau(\gamma)=\tau(p,q)$. Such a curve is called a \emph{distance realizer}.

A  Lorentzian  pre-length  space  is \emph{intrinsic} if for all $p< q$ and all $\varepsilon>0$ there exists a future directed causal curve $\gamma$ from $p$ to $q$ of length $L_\tau(\gamma)>\tau(p,q)-\varepsilon$.  Such a curve is called an \emph{$\varepsilon$-distance realizer}.
\end{defi}

A subset $A\subseteq X$ in a \LpLS is called \emph{(causally) convex} if for any two points $p,q\in A$ their causal diamond is contained in $A$, i.e., $J(p,q)\subseteq A$.

A key technical tool in smooth semi-Riemannian geometry is the existence of geodesically convex neighborhoods, 
in which the causality is particularly simple and where one has a complete description of length-maximizing
curves. The analogue of this notion in the present context is the following. 
A \LpLS $X$ is called \emph{localizable\/} if every $x\in X$ has an open, so-called \emph{localizing\/} %
neighborhood $\Omega_x$ such that
\begin{enumerate}[label=(\roman*)]
	\item The $d$-length of all causal curves contained in $\Omega_x$
	is uniformly bounded.
	\item \label{def-loc-LpLS} The localizing neighborhood $\Omega_x$ can be turned into a \LpLS by using the restrictions $d\rvert_{\Omega_x\times\Omega_x}$, $\ll\rvert_{\Omega_x\times \Omega_x}$ and $\leq\rvert_{\Omega_x\times\Omega_x}$, and define the local time separation function $\omega_x\colon \Omega_x \times \Omega_x\rightarrow [0,\infty)$ as follows. For $p,q\in\Omega_x$ we set $\omega_x(p,q):=\sup\{L_\tau(\gamma):\, \gamma:[a,b]\to\Omega_x$ is $\leq$-causal from $p$ to $q\}\cup\{0\}$. For $p<q$ we require that there is an $\omega_x$-realizer $\gamma_{p,q}$ in $\Omega_x$ from $p$ to $q$, and that this makes $\Omega_x$ into a \LpLSn. See e.g.\ \cite[chapter 1.7.2]{Ber:20} or \cite[Def.\ 3.16]{KS:18}.
	\item For every $y\in\Omega_x$ we have $I^\pm(y)\cap\Omega_x\neq\emptyset$.
\end{enumerate}
If, in addition, the neighborhoods $\Omega_x$ can be chosen such that
 \begin{enumerate}
 	\item[(iv)]\label{def-loc-LpLS-4} Whenever $p,q\in\Omega_x$ satisfy $p\ll q$ then $\gamma_{p,q}$ is timelike 
 	and strictly longer than any future-directed 
 	causal curve in $\Omega_x$ from $p$ to $q$ that contains a null segment,
 \end{enumerate}
 then \Xll is called {\em regularly localizable}.
 \bigskip

Locally distance realizing curves can be thought of as an analog of geodesics.

\begin{defi}[Geodesic]
    Let $X$ be a \LpLSn. A (say) future directed causal curve $\gamma\colon I\rightarrow X$ is a \emph{geodesic} if for each $t\in I$ there is a neighborhood $[a,b]$ of $t$ (i.e., $a<t<b$, but allowing for equality at the endpoints of $I$) such that $\gamma\rvert_{[a,b]}$ is a \emph{distance realizer}, i.e.\ $\tau(\gamma(a),\gamma(b))=L_\tau(\gamma\rvert_{[a,b]})$.
\end{defi}

In localizable spaces, the notion of a geodesic was previously defined using localizable neighborhoods as follows.
\begin{defi}[Geodesic \`a la {\cite[Def.\ 4.1]{GKS:19}} ]
     Let $X$ be a localizable \LpLSn. A (say) future directed causal curve $\gamma\colon I\rightarrow X$ is a \emph{geodesic} if it is \emph{locally maximal}, i.e., for every $t_0\in I$ there is a localizing neighborhood $\Omega_{\gamma(t_0)}$ and a neighborhood $[a,b]$ of $t_0$ in $I$ such that $\gamma\rvert_{[a,b]}$ is maximal from $\gamma(a)$ to $\gamma(b)$ in $\Omega_{\gamma(t_0)}$.
    \end{defi}
In this article we will use the first definition and the latter notion actually satisfies the first one when everything is considered in the \LpLS $\Omega_x$. Thus, when developing the theory here, it is enough to use the first version, and not employ localizability unless explicitly stated.
\medskip
    
\begin{defi}
We call $\tau$ \emph{locally finite-valued} if every point has a neighborhood $U$ such that $\tau<\infty$ on $U\times U$. Similarly, we call $\tau$ \emph{locally continuous} if every point has a neighborhood $U$ such that $\tau$ is continuous on $U\times U$.

Moreover, $X$ is \emph{locally causally closed} if every point has a neighborhood $U$ such that for all $x_n, y_n\in U$ with $x_n \to x\in \bar U$, $y_n\to y\in \bar U$ and $x_n\leq y_n$ for all $n\in\N$, then $x\leq y$. 

Finally, $X$ is \emph{causally path-connected} if for all $x,y\in X$ with $x<y$ (or $x\ll y$) there is a future directed causal (or timelike) curve from $x$ to $y$. 
\end{defi}
\begin{rem}
Note that by \cite[Lem.\ 4.3]{GKS:19} an intrinsic and strongly causal \LpLS is locally finite-valued. On the other hand, if $\tau$ is locally finite-valued, then $X$ is chronological, i.e, $\ll$ is irreflexive: By \cite[Prop.\ 2.14]{KS:18} one has that for all $x\in X$ either $\tau(x,x)=0$ or $\tau(x,x)=\infty$. The latter is excluded by assumption, so $\tau(x,x)=0$, which is equivalent to $x\not\ll x$.

The definition of local causal closedness is the original definition given in \cite[Def.\ 3.4]{KS:18}. However, in \cite{ACS:20} an issue with non-strongly-causal \LpLSn s was pointed out and an alternative definition has been proposed in \cite[Def.\ 2.19]{ACS:20}. We opted for the former one as we are only concerned with strongly causal spaces and for compatibility with previous results in \cite{KS:18,AGKS:21}. 

Any timelike distance realizer can be parametrized with respect to $\tau$-arclength as long as $\tau$ is locally finite-valued and locally continuous, cf.\ \cite[Subsec.\ 3.7]{KS:18}.

\end{rem}

\begin{defi}[\LLSn]
A \emph{Lorentzian length space} is a Lorentzian pre-length space which is locally causally closed, causally path connected, intrinsic and localizable.
\end{defi}
\begin{defi}\label{def:cbb-max-cc-fu} 
An open set $U\subseteq X$ is called \emph{timelike geodesically connected} if whenever $x,y \in U$ with $x\ll y$, there exists a future-directed maximal geodesic in $U$ from $x$ to $y$. 
An open set $U\subseteq X$ is called \emph{strictly timelike geodesically connected} if whenever $x$, $y \in U$ with $x\ll y$, there exists a future-directed maximal geodesic in $U$ from $x$ to $y$, and that any future-directed maximal geodesic in $U$ from $x$ to $y$ is timelike. 
$X$ is called \emph{locally strictly timelike geodesically connected} if it is covered by strictly timelike geodesically connected neighborhoods.
\end{defi}
\subsection{Curvature comparison for \LpLSn s}\label{subsec-cur-com}
Curvature bounds for metric spaces, generalizing sectional curvature bounds of Riemannian manifolds, are defined by comparing distances in triangles to distances in comparison triangles in two-dimensional (Riemannian) manifolds of constant curvature, see e.g.\ \cite{BH:99, BBI:01, AKP:22}. Timelike and causal curvature bounds were introduced analogously for \LpLSn s in \cite{KS:18}. In this setting we measure distances with the time separation, so we restrict to causal triangles.

\begin{defi}[Geodesic triangles]
A \emph{timelike (geodesic) triangle} $\Delta = (p_1,p_2,p_3)$ in a \LpLS $X$ consists of three points $p_1\ll p_2\ll p_3\in X$ (with $\tau(p_i,p_j)<\infty$ for $i<j$) and three future directed causal distance realizing curves $\alpha_{ij}$ connecting $p_i$ to $p_j$ (for $i<j$). Analogously, we define a \emph{causal triangle}.

An \emph{admissible causal (geodesic) triangle} $(p_1,p_2,p_3)$ in a \LpLS $X$ consists of three points $p_1\ll p_2\leq p_3$ or $p_1\leq p_2\ll p_3\in X$ (with $\tau(p_i,p_j)<\infty$ for $i<j$) and three possibly constant\footnote{Of course, if e.g.\ $p_1\leq p_2\ll p_3$, only $\alpha_{1,2}$ can be constant, and $\alpha_{2,3}$ in the other case.} future directed causal distance realizing curves $\alpha_{ij}$ connecting $p_i$ to $p_j$ (for $i<j$). We call the sides between two vertices $p_i\ll p_j$ a \emph{timelike side} (although it need not be realized via a timelike curve). 

We call $p_1$ the \emph{past endpoint} and $p_3$ the \emph{future endpoint} of the triangle. A causal or timelike triangle is called \emph{non-degenerate} if the reverse triangle inequality $\tau(p,r)\geq\tau(p,q)+\tau(q,r)$ is strict, and it is called \emph{degenerate in the strict sense} if the sides $\alpha_{12}$ and $\alpha_{23}$ are (reparametrized) parts of the longest side $\alpha_{13}$.
\end{defi}

\subsubsection{Comparison (model) spaces}

Here we recall the two-dimensional Lorentzian manifolds of constant curvature $K\in\R$, i.e., (scaled) (anti-)de Sitter spacetime and Minkowski spacetime.

\begin{defi}[Model spaces]\label{dfnConstCurvLor}
For integers $0\leq m\leq n$ denote by $\R^{n}_m$ the vector space $\R^{n}$ together with the inner product $b(v,w)=-\sum_{i=1}^m v_iw_i+\sum_{i=m+1}^{n}v_iw_i$, where $v=(v_1,\ldots,v_{n})$, $w=(w_1,\ldots,w_{n})$. Let $K\in\R$.

The Lorentzian \emph{$K$-planes} or the \emph{comparison spaces} of constant curvature $K$ are the following.
\begin{itemize}
\item Positive curvature $K>0$: We define $\lm{K}$ to be the universal cover of $\{v\in\R^{3}_1:b(v,v)=\frac{1}{K^2}\}$. %For $K=1$, this is deSitter space, the universal cover of a "standing" one-sheeted hyperboloid.
\item Negative curvature $K<0$: We define $\lm{K}$ to be the universal cover of $\{v\in\R^{3}_2:b(v,v)=-\frac{1}{K^2}\}$. %For $K=-1$, this is anti-deSitter space, the universal cover of a "lying" one-sheeted hyperboloid where the time direction wraps around the hyperboloid.
\item Curvature $K=0$ (flat): $\lm{0}:=\R^2_1$, the two-dimensional Minkowski spacetime.
\end{itemize}
%These comparison spaces are universal covers of smooth submanifolds of $\R^{3}_1$ and $\R^{3}_2$, respectively. The semi-Riemannian inner product $b$ of $\R^{n}_m$ restricts to a Lorentzian inner product on these spaces. 
The (finite timelike) diameter of $\lm{K}$ is $D_K=\frac{\pi}{\sqrt{-K}}$ for $K<0$ and $D_K=\infty$ for $K\geq0$\footnote{This is the maximum time separation which is finite.}.
\end{defi}

\begin{defi}
Three numbers $a_{12},a_{23},a_{13}\geq0$ satisfying the reverse triangle inequality $a_{12}+a_{23}\leq a_{13}$ (making $a_{13}$ the largest) are said to satisfy the \emph{timelike size bounds for $K$} if $a_{13}<D_K$.\footnote{$a_{13}<D_K$ is trivial if $K\geq0$.} %This is useful for constructing causal triangles with prescribed side-lengths in the Lorentzian $K$-plane, see Proposition \ref{lorTrianglesExist}.
\end{defi}

Andersson-Howard \cite{AH:98} introduced semi-Riemannian curvature bou\-nds and Alexander-Bishop developed the theory further in \cite{AB:08}, by characterizing smooth semi-Riemannian sectional curvature bounds by triangle comparison. In the semi-Riemannian setting it is necessary to distinguish the causal character of the tangent planes considered. Otherwise only spaces of constant curvature would fulfill the sectional curvature bounds, see \cite[Prop.\ 8.28]{ONe:83}.

Let $M$ be a semi-Rie\-mann\-ian manifold. Let $p\in M$ be a point and $P\subseteq T_pM$ be a plane. Then $P$ is called \emph{spacelike} if $g|_p$ is positive or negative definite on $P$, \emph{timelike} if $g|_p$ is non-degenerate and indefinite on $P$.

\begin{defi}[Sectional curvature comparison]
$M$ satisfies \emph{sectional curvature comparison from below (or above) by $K\in\R$} if for all points $p\in M$ and all spacelike planes $P\subseteq T_pM$, the sectional curvature is $\geq K$ (or $\leq K$) and for all timelike planes $P\subseteq T_pM$, the sectional curvature is $\leq K$ (or $\geq K$).
\end{defi}
\begin{rem}
Equivalently, sectional curvature comparison from below can be written in compact form as $g(R(v,w)w,v)\geq K(g(v,v)g(w,w)-g(v,w)^2)$ (and with inequality reversed for sectional curvature comparison from above).
\end{rem}

\begin{rem}
Sectional curvature comparison is not transitive in dimension larger than two: if $M$ satisfies sectional curvature comparison from below / above by $K\in\R$, it does not automatically satisfy it from below / above for any $\tilde{K}\neq K$. But if $M$ satisfies timelike sectional curvature comparison from below (or above) by $K$ (i.e., the sectional curvature comparison just for timelike planes), it automatically satisfies timelike curvature comparison from below (or above) by any $\tilde{K}\geq K$ (or $\tilde{K}\leq K$) (but note the reversal of the inequality as compared to the Riemannian case). For two dimensional Lorentzian manifolds there are only timelike tangent planes, hence sectional curvature comparison is timelike sectional curvature comparison. We give a detailed account for this in the case of \LpLSn s (and the Lorentzian $K$-planes) and triangle comparison in Lemma \ref{lem-rel-dif-K}.
\end{rem}

Note that comparison triangles exist, when the side-lengths satisfy size bounds for $K\in\R$. To be precise \cite[Lem.\ 2.1]{AB:08} gives the following.
\begin{pop}[Comparison triangles exist]\label{lorTrianglesExist}
Given three non-negative reals $a_{12},a_{23},a_{13}$ satisfying the reverse triangle inequality $a_{12}+a_{23}\leq a_{13}$ and timelike size bounds for $K\in\R$, there exists a causal triangle $\Delta p_1 p_2 p_3$ in a normal neighborhood in the $K$-plane such that $\tau(p_i,p_j)=a_{ij}$ (for $i<j$).

Any two such triangles $\Delta p_1 p_2 p_3$, $\Delta q_1 q_2 q_3$ (for the same side-lengths $a_{ij}$) in the $K$-plane are related by an isometry $\varphi$ mapping one to the other. The isometry $\varphi$ is unique unless $a_{13}=0$ (making all $a_{ij}=0$) or the reverse triangle inequality is actually an equality.
\end{pop}

\subsubsection{Triangle comparison}

In this subsection we introduce timelike curvature bounds as in \cite[Subsec.\ 4.3]{KS:18}.

Let $\Xll$ be a \LpLSn , $\lm{K}$ be the $K$-plane with time separation function $\bar{\tau}$ and $p_1\ll p_2\ll p_3$ be three timelike related points in $X$. A \emph{comparison triangle} of $\Delta p_1 p_2 p_3$ in the Lorentzian $K$-plane is a timelike triangle with vertices $\widebar{p_1}\ll\widebar{p_2}\ll\widebar{p_3}$ in the $K$-plane with agreeing side-lengths, i.e., $\tau(p_i,p_j)=\bar{\tau}(\widebar{p_i},\widebar{p_j})$ for $i<j$. We say the vertex $p_i$ corresponds to $\widebar{p_i}$. Moreover, we say that the side $\alpha_{ij}$ connecting $p_i$ to $p_j$ ($i<j$) corresponds to the side $\widebar{\alpha_{ij}}$ connecting $\widebar{p_i}$ to $\widebar{p_j}$. 

For a point $q$ on some side $\alpha_{ij}$ of the triangle we define the corresponding point $\widebar{q}$ on $\widebar{\alpha_{ij}}$ by requiring equal distances to the endpoints of the curve it is on, i.e., we require $\tau(p_i,q)=\bar{\tau}(\widebar{p_i},\widebar{q})$ (and then automatically, $\tau(q,p_j)=\bar{\tau}(\widebar{q},\widebar{p_j})$ as $\alpha_{ij}$ and $\widebar{\alpha_{ij}}$ are distance realizing). Note that it might be necessary to specify which side $q$ should be considered to be on (as two sides can partially overlap in $X$, but not in $\lm{K}$, unless the triangle is degenerate). If we take two such points $q_1,q_2$ (usually on different sides), we can compare the time separation of $q_1$ to $q_2$ with the time separation of the corresponding points in the comparison situation. This sets the stage for timelike and causal curvature comparison.

\begin{defi}[Timelike curvature bounds]\label{def-tri-com}
Let $X$ be a \LpLS and $K\in\R$. An open subset $U$ is called a \emph{timelike $\geq K$-comparison neighborhood} (or \emph{timelike $\leq K$-comparison neighborhood}) or just \emph{comparison neighborhood} if
\begin{itemize}
\item $\tau$ is finite and continuous on $U\times U$,
\item $U$ is strict timelike geodesically connected and
\item for all timelike triangles $\Delta p_1p_2p_3$ in $U$ satisfying timelike size bounds for $K$, $q_1,q_2$ two points on some sides $\alpha$ and $\beta$ and all (any) comparison situations $\Delta \widebar{p_1}\widebar{p_2}\widebar{p_3},\widebar{q_1},\widebar{q_2}$ in the $K$-plane, the time separation satisfies 
\[\tau(q_1,q_2)\leq\bar{\tau}(\widebar{q_1},\widebar{q_2})\quad\text{ (or }\tau(q_1,q_2)\geq\bar{\tau}(\widebar{q_1},\widebar{q_2})\text{)}\,.\]
\end{itemize}
We say $X$ has \emph{timelike curvature bounded below by $K$} if it is covered by timelike $\geq K$-comparison neighborhoods. Likewise, it has \emph{timelike curvature is bounded above by $K$} if it is covered by timelike $\leq K$-comparison neighborhoods.
\end{defi}
\begin{rem}
Note that in the case of timelike curvature bounded below, $q_1\ll q_2$ implies $\widebar{q_1}\ll\widebar{q_2}$, and in the case of timelike curvature bounded above, $\widebar{q_1}\ll\widebar{q_2}$ implies $q_1\ll q_2$. Furthermore, if $U$ is also a locally causally closed neighborhood we also have for curvature bounded below that $q_1\leq q_2$ implies $\widebar{q_1}\leq\widebar{q_2}$ and in curvature bounded above the other implication holds.
\end{rem}

In an analogous way we can define causal curvature bounds for admissible causal triangles where two vertices are not necessarily timelike related, cf.\ \cite[Subsec. 4.5]{KS:18}.

\begin{rem}[Automatic size bounds]\label{lorAutomaticSizeBounds}
If $X$ is strongly causal and has timelike curvature bounded above / below, we can restrict ourselves to comparison neighborhoods $U$ where timelike size bounds are automatically satisfied and of the form of timelike diamonds, and there exists a basis of the topology of such neighborhoods, see \cite[Rem.\ 2.1.8]{Ber:20}.
\end{rem}

% \begin{rem}[One-sided comparison]\label{oneSidedComparison}
% If we restrict ourselves to the case where one of the $q_i$ is a vertex $p_k$ of the triangle (which we call \emph{one-sided comparison}), we introduce the following comparison denominations.
% 
% We distinguish \emph{future comparison} if $q_i=p_3$, \emph{past comparison} if $q_i=p_1$ and \emph{across comparison} if $q_i=p_2$. 
% 
% The definition of timelike and causal curvature bounds with triangle comparisons involving two points $q_1$ and $q_2$ on two sides is equivalent to defining it with only one-sided comparisons, see \cite[Remark 2.2.28]{Ber:20}.
% \end{rem}

\section{Angles}\label{sec-ang} 
In this section we introduce comparison angles, establish foundational properties and define the notion of an angle between timelike curves.

\subsection{Definition and basic properties}

\begin{defi}[$K$-comparison angles]
 Let $K\in\R$, let \Xll be a \LpLS and let $x,y,z\in X$ be \emph{causally} related, i.e.\ $x\leq y\leq z$ or a permutation of this, and all time separations between these points finite. If $K<0$ we demand that the side-lengths of $\Delta x y z$ satisfy the timelike size bounds for $K$. Let $x\ll y$ and $x\ll z$ and let $\ct$ be a comparison triangle of $\Delta x y z$ in $\lm{K}$.
The \emph{$K$-comparison (hyperbolic)} angle at $x$ is defined as
 \begin{equation}
  \tilde\ma_x^K(y,z):=\ma^{\lm{K}}_{\bar x}(\bar y,\bar z)\,.
 \end{equation}
%We leave $\tilde\ma_x(y,z)$ undefined if $x,y,z$ are not causally related, $x\not\ll y$, $x\not\ll z$ or the three points do not satisfy timelike size bounds for $K$.

 Additionally, for notational simplicity we set $\tilde\ma_x(y,z)$ $:=\tilde\ma_x^0(y,z)$. 

To make the notation more concise we define the signed angle: If the angle is at a future or past endpoint, we set $\sigma=-1$, if it is not, $\sigma=1$.
We define the \emph{signed angle $\tilde\ma_x^{K,\mathrm{S}}(y,z):=\sigma\tilde\ma_x^K(y,z)$}. 
\end{defi}
\begin{rem}
Note that in case $x\ll y\ll z$ we have $\tilde\ma_x^{K,\mathrm{S}}(y,z)=-\tilde\ma_x^K(y,z)$, so minus the ``usual'' angle, see Figure \ref{fig-ang-tri}. Moreover, we will sometimes refer to the value of $\sigma$ when we want to distinguish cases depending on the time orientation of the vertices. Note also that $\tilde\ma_x^K(y,z)=\tilde\ma_x^K(z,y)$.
\end{rem}

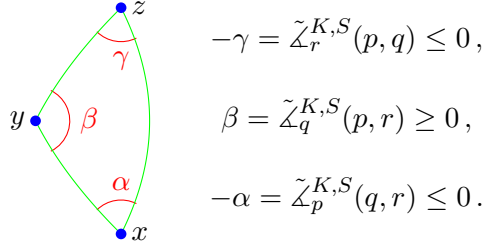
\begin{figure}[h!]
\begin{center}
\bgroup
\def\arraystretch{2.2}%  1 is the default, change whatever you need
\begin{tabular}{cc}
\multirow{3}{*}{%\includegraphics{tikz/triangle_all_angles.pdf}
\begin{tikzpicture}[x=1.5cm,y=1.5cm]
\draw [shift={(-0.25,-1.)},color=red]  (63.43494882292201:0.3) arc (63.43494882292201:135.:0.3) ; %\alpha
\draw [shift={(-1.,0.)},color=red]  (-63.43494882292201:0.3) arc (-63.43494882292201:63.43494882292201:0.3) ;%\beta:inner version
\draw [shift={(-0.25,1.)},color=red]  (-135.:0.3) arc (-135.:-63.43494882292201:0.3) ;%\gamma
\draw[color=green,smooth,samples=100,domain=-1:1] plot({-\x*\x/4},{\x});
\draw[color=green,smooth,samples=100,domain=-1:0] plot({\x*\x/4-\x/2-1},{\x});
\draw[color=green,smooth,samples=100,domain=0:1] plot({\x*\x/4+\x/2-1},{\x});
\fill[color=blue] (-0.25,-1.) circle (2pt);%x
\fill[color=blue] (-1.,0.) circle (2pt);%y
\fill[color=blue] (-0.25,1.) circle (2pt);%z
\draw (-0.25,-1.) node[anchor=west]{$x$} (-1.,0.) node[anchor=east]{$y$} (-0.25,1.) node[anchor=west]{$z$};
\draw[color=red] (-0.25,-1.) +(0,0.3) node[anchor=south]{$\alpha$} (-1.,0.) +(0.3,0) node[anchor=west]{$\beta$} (-0.25,1.) +(0,-0.3) node[anchor=north]{$\gamma$};
\end{tikzpicture}
} 
& $-\gamma=\tilde\ma_r^{K,S}(p,q)\leq 0\,,$\\
& $\beta=\tilde\ma_q^{K,S}(p,r)\geq 0\,,$\\
& $-\alpha=\tilde\ma_p^{K,S}(q,r)\leq 0\,.$ 
\end{tabular} 
\egroup
\end{center}
\caption{The three signed angles of a triangle.}\label{fig-ang-tri}
\end{figure}

\begin{rem}[Angles in comparison spaces]
Let $\bar x,\bar y,\bar z\in \lm{K}$, such that $\bar x\ll \bar y$, $\bar x\ll \bar z$ then $\ma^{\lm{K}}$ is defined as follows:
 \begin{equation}
  \ma^{\lm{K}}_{\bar x}(\bar y,\bar z)= \arcosh\Bigl(\bigl|\lara{\gamma_{\bar x,\bar y}'(0),\gamma_{\bar x,\bar z}'(0)}\bigr|\Bigr)\,,
 \end{equation}
where $\gamma_{r,s}$ is the unit speed timelike geodesic from $r$ to $s$ and $\lara{\cdot,\cdot}$ is the metric on $\lm{K}$. We have an analogous equation at the other vertices. Note that $\arcosh\Bigl(\bigl|\lara{\gamma_{\bar x,\bar y}'(0),\gamma_{\bar x,\bar z}'(0)}\bigr|\Bigr)$ is defined as the argument is greater equal than one by the reverse Cauchy Schwarz inequality for timelike vectors: $|\lara{v,w}|\geq |v||w|$.

The sign $\sigma$ of the signed angle $\tilde\ma_x^{K,\mathrm{S}}(y,z)$ is just the sign of $\lara{\gamma_{\bar x,\bar y}'(0),\gamma_{\bar x,\bar z}'(0)}$ (which is the reason to define it this way).
\end{rem}
Physically, the angle is a monotonously increasing function in the change of velocity of a particle switching from $\alpha$ to $\beta$.
Moreover, note that in case $\sigma=1$, the angles might seem to go against our (spatial) intuition: The angle $\omega$ between two parts of a straight line in Minkowski space is $0$ here, in contrast to the Euclidean angle taking the maximum value $\pi$, as indicated in Figure \ref{fig-zeroAngle}.
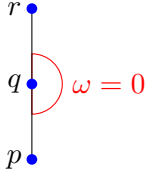
\begin{figure}[h!]
\begin{center}
\begin{tikzpicture}
\draw [color=red] (0,0) -- (-90.:0.4) arc (-90.:90.:0.4) -- cycle;%\gamma
\draw (0,-1) -- (0,1);
\fill[color=blue] (0,-1) circle(2pt) (0,0) circle(2pt) (0,1) circle(2pt);
\draw (0,-1) node[anchor=east]{$p$} (0,0) node[anchor=east]{$q$} (0,1) node[anchor=east]{$r$};
\draw[color=red] (0,0) +(0.4,0) node[anchor=west]{$\omega=0$};
\end{tikzpicture}
\end{center}
\caption{The $\omega=0$ angle with $\sigma=1$ appears large to our Euclidean intuition.}\label{fig-zeroAngle}
\end{figure}

%  Take two particles with worldlines $\alpha,\beta$ which meet at $x=\alpha(0)=\beta(0)$. Let $v$ be the relative speed [in units of c], and let $\omega=\ma_x(\alpha,\beta)$ (for any choice of future / past half). Then $\tanh(\omega)=v$. See wikipedia: Rapidity

A fundamental tool in (semi-)Riemannian and metric geometry is the law of cosines. For the convenience of the reader we include a proof in our setting in the appendix \ref{PrLorLOC}. Note that it also could be derived from  \cite[Lem.\ 2.2. and Lem.\ 2.3]{AB:08}. The definition of the comparison angle leads to the \emph{hyperbolic law of cosines} as follows.

\begin{lem}[Law of cosines]\label{lorLawOfCosines}
For $p,q,r$ in the Lorentzian $K$-plane forming a finite causal triangle (not necessarily in this order), let $a=\max(\tau(p,q),$ $\tau(q,p))$, $b=\max(\tau(q,r),\tau(r,q))$, $c=\max(\tau(p,r),\tau(r,p))$ be the side-lengths (note that in each maximum, one of the two arguments is $0$) with $a,b>0$, i.e., $q$ is timelike related to $p$ and $r$ but $p$ and $r$ need only be causally related. Let $\omega=\tilde{\ma}^K_q(p,r)$ be the hyperbolic angle at $q$, $\sigma$ be the sign of the signed angle $\tilde{\ma}^{K,\mathrm{S}}_q(p,r)$ and the scaling factor $s=\sqrt{|K|}$. Then we have:
\smallskip

\begin{tabular}{cl}
$a^2+b^2=c^2-2ab\sigma\cosh(\omega)$&for $K=0$\,,\\
$\cos(sc) = \cos(sa) \cos(sb)-\sigma\cosh(\omega) \sin(sa) \sin(sb)$&for $K<0$\,,\\
$\cosh(sc) = \cosh(sa) \cosh(sb)+\sigma\cosh(\omega) \sinh(sa) \sinh(sb)$&for $K>0$\,.
\end{tabular}
\end{lem}
\begin{rem}[Analyticity and monotonicity in the law of cosines]\label{rem-loc-mon}¸
Note these three formulae join up to an \emph{analytic formula}: For the first equation, use $c^2=a^2+b^2+2ab\sigma\cosh(\omega)$. For the second equation, use $\frac{1-\cos(sc)}{|K|} = \frac{1-\cos(sa) \cos(sb)}{|K|}+\sigma\cosh(\omega) \frac{\sin(sa) \sin(sb)}{|K|}$. For the third equation, use $\frac{\cosh(sc)-1}{|K|} = \frac{\cosh(sa) \cosh(sb)-1}{|K|}+\sigma\cosh(\omega) \frac{\sinh(sa) \sinh(sb)}{|K|}$. Then the left- resp.\ right-hand-sides join up to an analytic formula in $c$ and $K$ resp.\ $a$, $b$, $\omega$ and $K$. 

In particular, fixing two sides and varying the third, $\omega$ is a strictly increasing function in the longest side and a strictly decreasing function in the other two sides.
\end{rem}

In the case $\sigma=-1$, we can even make the side opposite the angle spacelike.

\begin{cor}[Extended law of cosines]\label{cor-ext-loc}
Let $p,q,r$ in the Lorentzian $K$-plane with pairwise finite $\tau$-distances. We assume $p\ll q$ and $p\ll r$, but $q$ and $r$ not causally related and not all three on a single geodesic. Let $a=\tau(p,q)$, $b=\tau(p,r)$ be the side-lengths. Let $\omega=\ma^{\lm{K}}_p(q,r)$. \footnote{Note that this angle is still defined as an angle in the model space $\lm{K}$ even though the side opposite the angle is spacelike.} Then we have\\
\begin{tabular}{cl}
$a^2+b^2< 2ab\cosh(\omega)$&for $K=0$\,,\\
$1 < \cos(sa) \cos(sb)+\cosh(\omega) \sin(sa) \sin(sb)$&for $K<0$\,,\\
$1 > \cosh(sa) \cosh(sb)-\cosh(\omega) \sinh(sa) \sinh(sb)$&for $K>0$\,.
\end{tabular}
\end{cor}
\begin{proof}
Let $q,r$ be not causally related. Without loss of generality we assume $\tau(p,r)\geq \tau(p,q)$ and consider the timelike geodesic $\alpha$ from $p$ to $q$.  We consider the parameter $t'=\sup \alpha^{-1}(I^-(r))$, then the point $q':=\alpha(t')$ is the point along $\alpha$ being null related to $r$, i.e., $q'\leq r$, $q'\not\ll r$. Note $\omega:=\ma^{\lm{K}}_p(q,r)=\ma^{\lm{K}}_p(q',r)$. We define the side-lengths: $a'=\tau(p,q')<a=\tau(p,q)$, $b=\tau(p,r)$. We apply the law of cosines (Lemma \ref{lorLawOfCosines} with $\sigma=-1$ and a null side) to this situation. We now form a new triangle $\Delta \bar p \bar q \bar r$ with $\tau(\bar p,\bar q)=a$, $\tau(\bar p,\bar r)=b$ and $\bar q\leq \bar r$ null related. Let $\bar\omega:=\ma_{\bar p}^{\lm{K}}(\bar q,\bar r)$. The triangles $\Delta \bar p \bar q \bar r$ and $\Delta p q' r$ only differ by the value of $\omega$ and the side-length $a>a'$. Now we apply the monotonicity statement of the law of cosines to $a$. Thus $\omega$ is a decreasing function of $a$, thus $\omega<\bar \omega$, and the formula in the law of cosines holds with equality for $a,b,c=0,\bar \omega$. Replacing $\bar \omega$ by $\omega$, one easily sees whether the right-hand-side increases or decreases.
\end{proof}
\begin{rem}
The extended law of cosines fits nicely with the monotonicity of the law of cosines, i.e., taking a timelike triangle and making the side opposite the angle null makes the angle larger, and similarly going from a null side to a spacelike side opposite the angle increases the angle. 
\end{rem}

The following is helpful for calculations and is an example how one can utilize the law of cosines. The proof is outsourced to the appendix \ref{PrLorLOC}.

\begin{lem}[Calculating one-sided comparison situations] \label{lorOneSidedCalcs}
(1) Let $K\in\mb{R}$ and let $a+b+c\leq d$, and $d$ satisfy timelike size bounds for $K$. Then construct a timelike triangle $\Delta p_1p_2p_3$ in $\lm{K}$ with sidelengths $a,b+c,d$. We get a point $q$ on the side $[p_2p_3]$ with $\tau(p_2,q)=b$ (see Figure \ref{fig-one_sided_comparisons}). We denote $x=\tau(p_1,q)$ and $s=\sqrt{|K|}$. Then we have
\begin{align*}
\text{for }K=0\\
x^2&=\frac{bd^2+a^2c}{b+c}-bc\,,\\
\text{for }K<0\\
\cos(sx)&=(\cos(sd)-\cos(sa)\cos(s(b+c)))\ \times\\
&\qquad\frac{\sin(sb)}{\sin(s(b+c))}+\cos(sa)\cos(sb)\,,\\
\text{for }K>0\\
\cosh(sx)&=(\cosh(sd)-\cosh(sa)\cosh(s(b+c)))\ \times\\
&\qquad\frac{\sinh(sb)}{\sinh(s(b+c))}+\cosh(sa)\cosh(sb)\,, \text{ and}
\end{align*}
\[
x^2=\frac{bd^2+a^2c}{b+c}-bc+O(s^2d^4)
%+O(s^2x^4)+O(s^2a^4)+O(s^2b^4)+O(s^2a^2b^2)+O(s^2b^2c^2)+O(s^2a^2c^2)+O(s^2a^2bc)+O(s^2d^4)
\text{ for $K\to0$ or $d\to0$}\,.
\]
(2) Under the same conditions, we can construct a timelike triangle $\Delta p_1p_2p_3$ in $\lm{K}$ with sidelengths 
$a+b,c,d$. We get a point $q$ on the side $[p_1p_2]$ with $\tau(p_1,q)=a$ (see Figure \ref{fig-one_sided_comparisons}). We 
denote $x=\tau(q,p_3)$ and $s=\sqrt{|K|}$. Then we have

\begin{align*}
\text{for }K=0\\
x^2&=\frac{bd^2+c^2a}{b+a}-ba\,,\\
\text{for }K<0\\
\cos(sx)&=(\cos(sd)-\cos(sc)\cos(s(b+a)))\ \times\\
&\qquad \frac{\sin(sb)}{\sin(s(b+a))}+\cos(sc)\cos(sb)\,,\\
\text{for }K>0\\
\cosh(sx)&=(\cosh(sd)-\cosh(sc)\cosh(s(b+a)))\ \times\\
\qquad&\frac{\sinh(sb)}{\sinh(s(b+a))}+\cosh(sc)\cosh(sb)\,, \text{ and}
\end{align*}
\[
x^2=\frac{bd^2+c^2a}{b+a}-ba+O(s^2d^4)\text{ for $K\to0$ or $d\to 0$}\,.
\]

(3) Now let $K\in\mb{R}$ and let $a+b\leq c+d$, and $c+d$ satisfy timelike size bounds for $K$. Then construct a timelike triangle 
$\Delta p_1p_2p_3$ in $\lm{K}$ with sidelengths $a,b,c+d$. We get a point $q$ on the side $[p_1p_3]$ with $\tau(p_1,q)=c$ (see 
Figure \ref{fig-one_sided_comparisons}). We denote $x=\max(\tau(p_2,q),\tau(q,p_2))$ and $s=\sqrt{|K|}$. Then we have
\begin{align*}
\text{for }K=0\\
x^2&=\frac{cb^2+a^2d}{c+d}-cd\,,\\
\text{for }K<0\\
\cos(sx)&=(\cos(sb)-\cos(sa)\cos(s(c+d)))\ \times\\
&\frac{\sin(sc)}{\sin(s(c+d))}+\cos(sa)\cos(sc)\,,\\
\text{for }K>0\\
\cosh(sx)&=(\cosh(sb)-\cosh(sa)\cosh(s(c+d)))\ \times\\
&\frac{\sinh(sc)}{\sinh(s(c+d))}+\cosh(sa)\cosh(sc)\,, \text{ and}
\end{align*}
\[
x^2=\frac{cb^2+a^2d}{c+d}-cd+O(s^2(c+d)^4)\text{ for $K\to0$ or $c+d\to0$}\,.
\]
Note that in each case, the equations can be transformed to match the left-hand-sides of the equations in the analyticity statement of the law of cosines \ref{lorLawOfCosines}, and then the left- and right-hand-sides join up to an analytic formula.
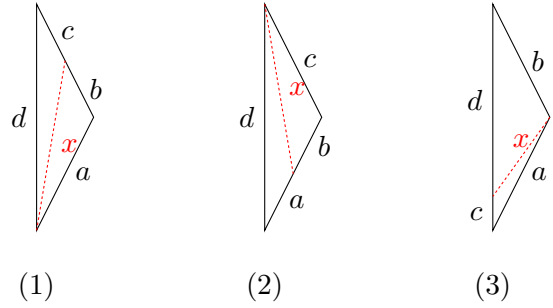
\begin{figure}[h!]
\begin{center}
\begin{tikzpicture}[x=1.5cm,y=1.5cm]
\begin{scope}
\draw (0,0) -- (0.5,1) -- (0,2) -- cycle;
\draw[color=red,dash pattern=on 1pt off 1pt] (0,0) -- ({(0.5+0)/2},{(1+2)/2});
\draw 
({(0+0.5)/2},{(0+1)/2}) node[anchor=west]{$a$}
({(2*0.5+0.5+0)/4},{(2*1+1+2)/4}) node[anchor=west]{$b$}
({(0.5+0+2*0)/4},{(1+2+2*2)/4}) node[anchor=west]{$c$}
({(0+0)/2},{(0+2)/2}) node[anchor=east]{$d$};
\draw[color=red]({(0+(0.5+0)/2)/2},{(0+(1+2)/2)/2}) node[anchor=west]{$x$};
\draw (0,-0.5) node{(1)};
\end{scope}

\begin{scope}[shift={(2,0)}]
\draw (0,0) -- (0.5,1) -- (0,2) -- cycle;
\draw[color=red,dash pattern=on 1pt off 1pt] (0,2) -- ({(0+0.5)/2},{(0+1)/2});
\draw 
({(2*0+0+0.5)/4},{(2*0+0+1)/4}) node[anchor=west]{$a$}
({(0+0.5+2*0.5)/4},{(0+1+2*1)/4}) node[anchor=west]{$b$}
({(0.5+0)/2},{(1+2)/2}) node[anchor=west]{$c$}
({(0+0)/2},{(0+2)/2}) node[anchor=east]{$d$};
\draw[color=red]({(0+(0+0.5)/2)/2},{(2+(0+1)/2)/2}) node[anchor=west]{$x$};
\draw (0,-0.5) node{(2)};
\end{scope}

\begin{scope}[shift={(4,0)}]
\draw (0,0) -- (0.5,1) -- (0,2) -- cycle;
\draw[color=red,dash pattern=on 1pt off 1pt] (0.5,1) -- (0,0.3);
\draw 
({(0+0.5)/2},{(0+1)/2}) node[anchor=west]{$a$}
({(0.5+0)/2},{(1+2)/2}) node[anchor=west]{$b$}
({(0+0)/2},{(0+0.3)/2}) node[anchor=east]{$c$}
({(0+0)/2},{(0.3+2)/2}) node[anchor=east]{$d$};
\draw[color=red]({(0.5+0)/2},{(1+0.3)/2}) node[anchor=south]{$x$};
\draw (0,-0.5) node{(3)};
\end{scope}

\end{tikzpicture}
\end{center}
\caption{The three one-sided comparison situations (case (3) consists of two cases: one can distinguish whether $x$ points to 
the future or the past).}\label{fig-one_sided_comparisons}
\end{figure}
\end{lem}

At this point we introduce the angle between timelike curves.

\begin{defi}[Angles between curves]
 Let \Xll be a \LpLSn, where $\tau$ is locally finite-valued. Let $\alpha,\beta\colon[0,\eps)\rightarrow X$ be two future or past directed timelike curves with $\alpha(0)=\beta(0)=:x$. We define
\begin{align}
  A_K&:=\{(s,t)\in(0,\eps)^2\colon\\
  &\alpha(s)\leq\beta(t)\text{ and } \Delta x \alpha(s)\beta(t)\text{satisfies size bounds for }K \text{ or }\\
  &\beta(t)\leq\alpha(s)\text{ and } \Delta x\beta(t)\alpha(s)\text{ satisfies size bounds for }K\}\,.
 \end{align}
 The \emph{upper angle} between $\alpha$ and $\beta$ at $x$ is
 \begin{equation}
  \ma_x(\alpha,\beta):=\limsup_{(s,t)\in A_0;\,s,t\searrow 0} \tilde\ma_x(\alpha(s),\beta(t))\,.
 \end{equation}
Moreover, if $\lim_{(s,t)\in A_0;\,s,t\searrow 0} \tilde\ma_x(\alpha(s),\beta(t))$ exists and is finite, then we call $\ma_x(\alpha,\beta)$ the \emph{angle} between $\alpha$ and $\beta$ at $x$, and say \emph{the angle between $\alpha$ and $\beta$ at $x$ exists}. 
Similarly to the signed angle for three points, we define the signed angle between curves: If $x$ is the future or past endpoint, i.e., $\alpha$ and $\beta$ have the same time orientation, we set $\sigma=-1$ and if it is not, i.e., $\alpha$ and $\beta$ have different time orientation, we set $\sigma=1$. We define the signed angle $\ma_x^{\mathrm{S}}(\alpha,\beta)=\sigma\ma_x(\alpha,\beta)$. 
\end{defi}
\begin{rem}
The set $A_K$ is precisely the set of $(s,t)$ for which $\tilde\ma_x^K(\alpha(s),\beta(t))$ is defined. Also, note that the signed angle between $\alpha$ and $\beta$ has the same sign as the signed angle of $\Delta\alpha(s)x\beta(t)$ at $x$ for any $s,t>0$.
\end{rem}

Note that the definition of the upper angle of two future or past directed timelike curves $\alpha,\beta$ makes sense, i.e., there are $s,t\to 0$ with $(s,t)\in A_K$: In case $\alpha$ and $\beta$ have a different time orientation, i.e., $\sigma=1$, the condition $\alpha(s)\leq\beta(t)$ or conversely is automatically satisfied. In the other case, i.e., $\sigma=-1$, say without loss of generality $\alpha,\beta$ future directed. Then for fixed $t\in[0,\eps)$, there is a $\delta>0$ such that $\alpha([0,\delta))\subseteq I^-(\beta(t))$, i.e., $x\ll \alpha(s)\ll\beta(t)$ for all $s\in[0,\delta)$.

\bigskip

Another useful concept is the one of a hinge:
\begin{defi}[$K$-comparison hinges]
Let \Xll be a \LpLSn, where $\tau$ is locally finite-valued. We call two future or past directed timelike geodesics $\alpha:[0,a]\to X$, $\beta:[0,b]\to X$ with $\alpha(0)=\beta(0)=:x$ such that the angle between them exists a \emph{hinge}.

Let $(\alpha,\beta)$ be a hinge, and let $K\in\mb{R}$. Then a \emph{comparison hinge} is a hinge $(\bar{\alpha},\bar{\beta})$ in $\lm{K}$, where the sides have the same time orientation and with $\ma_x(\alpha,\beta)=\ma^{\lm{K}}_{\bar{x}}(\bar{\alpha},\bar{\beta})$, $L_\tau(\alpha)=L(\bar{\alpha})$ and $L_\tau(\beta)=L(\bar{\beta})$. 
\end{defi}
\begin{rem}
Comparison hinges always exist and are unique up to (unique) isometry of $\lm{K}$, if $L_\tau(\alpha)$ and $L_\tau(\beta)$ satisfy timelike size bounds for $K$, where the isometry is unique if none of $L_\tau(\alpha),L_\tau(\beta)$ and $\ma_x(\alpha,\beta)$ is zero.
\end{rem}

As a first compatibility check we establish below that in a ($\Con^2$, strongly causal) spacetime the angle between two timelike curves with respect to the Lorentzian metric and the angle as defined above agree.
\begin{pop}[Angles agree]\label{hypAnglesAgreeSpt}
Let $\alpha,\beta:[0,\varepsilon)\to M$ be future or past directed $\Con^1$-regular $g$-timelike curves in a strongly causal $\Con^2$-spacetime $(M,g)$ with $\alpha(0)=\beta(0)=x$. Then the angle $\ma_x(\alpha,\beta)$ in the \LpLS is the same as the Lorentzian angle in $(M,g)$, i.e.,
\begin{equation}
 \ma_x(\alpha,\beta)=\arcosh\Bigl(\frac{\bigl|g(\alpha'(0),\beta'(0))\bigr|}{|\alpha'(0)|\,|\beta'(0)|}\Bigr)\,.
\end{equation}
\end{pop}
\begin{proof}
First, note that by \cite[Ex.\ 3.24(i)]{KS:18} $(M,g)$ gives rise to a \LLS and in particular the time separation function is locally finite. Second, by \cite[Lem.\ 2.21(i)]{KS:18} any $g$-timelike curve is timelike in the relation sense.

Let $U$ be a causally convex normal neighbourhood of $x$. We now consider $\varphi_\varepsilon:\frac{1}{\varepsilon}\exp_x^{-1}(U)\to U$ given by $\varphi_\varepsilon(v)=\exp_x(\varepsilon v)$. We pullback the curves and the metric via $\varphi_\varepsilon$, i.e., $\tilde{\alpha}_\varepsilon(t):=\varphi^{{-1}}_\varepsilon(\alpha(\varepsilon t))$ and $\tilde{\beta}_\varepsilon(t):=\varphi^{{-1}}_\varepsilon(\beta(\varepsilon t))$ for $t\in [0,1)$ and note that they converge uniformly in $\mathcal{C}^1$ to $\tilde{\alpha}(t):=t\alpha'(0)$ and $\tilde{\beta}(t):=t\beta'(0)$ as $\varepsilon\searrow 0$. Similarly, the metric $\tilde{g}_\varepsilon=\frac{1}{\varepsilon^2} (\varphi_\varepsilon)_*g$ converges locally uniformly in $\mathcal{C}^1$ to $g_x$, which we can assume to be the Minkowski metric. In particular, we obtain a time separation function $\tilde{\tau}_\varepsilon$ on $\frac{1}{\varepsilon}\exp_x^{-1}(U)$ from the metric $\tilde{g}_\varepsilon$. Note that as $U$ is causally convex, we have $\varepsilon\tilde{\tau}_\varepsilon(p,q)=\tau(\varphi_\varepsilon(p),\varphi_\varepsilon(q))$.

As the (Minkowski-)comparison angle is scale-invariant, we have that $\tilde{\ma}_x(\alpha(\varepsilon s),\beta(\varepsilon t))=\tilde{\ma}_0(\tilde{\alpha}_\varepsilon(s),\tilde{\beta}_\varepsilon(t))$, where the latter uses $\tilde{\tau}_\varepsilon$ to define the comparison angle. Finally, since $\alpha'(0)$ and $\beta'(0)$ are bounded away from the null cone, we have $\frac{\tilde{\tau}(x,\tilde{\alpha}(s))}{\tilde{\tau}_\varepsilon(x,\tilde{\alpha}_\varepsilon(s))} \to 1$ as well as $\frac{\tilde{\tau}(x,\tilde{\beta}(s))}{\tilde{\tau}_\varepsilon(x,\tilde{\beta}_\varepsilon(s))} \to 1$, and $\tilde{\tau}(\tilde{\alpha}(s),\tilde{\beta}(t)) - \tilde{\tau}_\varepsilon(\tilde{\alpha}_\varepsilon(s),\tilde{\beta}_\varepsilon(t)) \to 0$, where $\tilde\tau$ denotes the Minkowski time separation.

By the law of cosines formula, we see that the comparison angle between $\tilde{\alpha}_\varepsilon(s)$ and $\tilde{\beta}_\varepsilon(t)$ converges to the comparison angle between $\tilde{\alpha}(s)$ and $\tilde{\beta}(t)$ which is $\arcosh\left(\frac{|g_x(\alpha'(0),\beta'(0))|}{|\alpha'(0)||\beta'(0)|}\right)$.

% Finally, since $\tilde{\tau}_\varepsilon$ converges to the Minkowski metric $g_x$, we conclude that
% $\ma_x(\alpha,\beta)=\lim_{\varepsilon\searrow 0} \tilde{\ma}_0(\tilde{\alpha}_\varepsilon(s),\tilde{\beta}_\varepsilon(t)) = \arcosh\Bigl(\frac{\bigl|g_x(\alpha'(0),\beta'(0))\bigr|}{|\alpha'(0)|\,|\beta'(0)|}\Bigr)$.
\end{proof}

\medskip

As a consequence we obtain that the angle between timelike curves does not depend on the comparison angle used. This is an analog of \cite[Prop.\ II.3.1, p.\ 184]{BH:99} in the metric case.
\begin{pop}[All $\tilde{\ma}^K$ converge to $\ma$]\label{lorSphericalangleNoDepK}
Let $X$ be a strongly causal \LpLS with $\tau$ locally finite-valued, $\alpha,\beta$ timelike curves (each can be future or past directed) with $\alpha(0)=\beta(0)=x$. Then for all $K\in\R$, we have that 
\begin{align}
 \limsup_{(s,t)\in A_K;\, s,t\searrow 0}\tilde{\ma}^K_x(\alpha(s),\beta(t))=\ma_x(\alpha,\beta)\,,
\end{align}
and the limit superior on the left-hand-side is a limit and finite if and only if the angle between $\alpha$ and $\beta$ exists. 
\end{pop}
\begin{pr}
We restrict to $t,s$ small enough to have the timelike size bounds automatically satisfied (cf.\ Remark \ref{lorAutomaticSizeBounds}). Setting $l=\sqrt{|K|}$, $\sigma=\pm1$ the appropriate sign and the side-lengths $a=\max(\tau(x,\alpha(s)),\tau(\alpha(s),x))$, $b=\max(\tau(x,\beta(t)),\tau(\beta(t),x))$, $c=\max(\tau(\alpha(s),\beta(t)),\tau(\beta(t),\alpha(s)))$ we apply the law of cosines (Lemma \ref{lorLawOfCosines}): If $K>0$, we obtain
\[\cosh(\tilde{\ma}^K_x(\alpha(s),\beta(t)))=\sigma\frac{\cosh(la) \cosh(lb)-\cosh(lc)}{\sinh(la) \sinh(lb)}=:(\star)\,.\]

We first look at the denominator of this fraction: Note $\sinh(x)=x+o(x)$ as $x\to 0$, so we have 
\begin{equation*}
\sinh(la)\sinh(lb)=l^2ab + o(la)\cdot lb + la\cdot o(lb)=l^2ab\cdot (1+o(1))\,.
\end{equation*}
Now we have $\frac{1}{\sinh(la)\sinh(lb)}=\frac{1+o(1)}{l^2ab}$ and are done with the denominator.

Now we look at the enumerator of this fraction: We use the Taylor series of $\cosh$, i.e., $\cosh(x)=1+\sum_{n\geq2, \text{ even}} \frac{x^n}{n!}$. For the whole enumerator, we thus have
\begin{align}
&\cosh(la)\cosh(lb)-\cosh(lc)\\
&=\sum_{n\geq2, \text{ even}} \frac{(la)^n}{n!} + \sum_{n\geq2, \text{ even}} \frac{(lb)^n}{n!} + o(l^2ab) -\sum_{n\geq2, \text{ even}} \frac{(lc)^n}{n!}\,,
\end{align}
where we abbreviated the sum of terms where $a^2b^2$ appeared by $o(l^2ab)$. We now look at each $n$ separately, together with the denominator: for $n=2$, we get the desired $\frac{l^2}{l^2}\cdot\frac{a^2+b^2-c^2}{2ab}$. For $n>2$, we get $\frac{l^n}{n!}$ times
\begin{align*}
&a^n+b^n-c^n\\
&= (a^2+b^2)(a^{n-2}+b^{n-2}-c^{n-2})+(a^2+b^2-c^2)c^{n-2} -a^2 b^{n-2} - a^{n-2}b^2\\
&=o(1)(a^{n-2}+b^{n-2}-c^{n-2})-(a^2+b^2-c^2)o(1) + o(ab)\,,
\end{align*}
now we use induction on $n$ to get $a^n+b^n-c^n=o(a^2+b^2-c^2)$. 

In total, we have
\begin{align}
(\star)&=\sigma \frac{(a^2+b^2-c^2)\cdot\Bigl(1+\sum_{n\geq4, \text{ even}}\frac{l^{n-2}}{n!}o(a^2+b^2-c^2)\Bigr)+o(l^2ab)}{2ab}\\
&\times\,(1+o(1))\,.
\end{align}
Note that the sum in the enumerator is
\begin{align}
(a^2+b^2-c^2)\,o(1)\sum_{n\geq4, \text{ even}}\frac{l^{n-2}}{n!}\sum_{k=0}^{n-2}c^k\,, 
\end{align}
which is $o(a^2+b^2-c^2)$. So we have
\begin{equation*}
(\star)=\sigma \frac{(a^2+b^2-c^2)\cdot(1+o(1)) + o(ab)}{2ab}\cdot (1+o(1))\,.
\end{equation*}

For $a,b,c\to0$, this converges if and only if 
\[\cosh(\tilde{\ma}^0_x(\alpha(s),\beta(t)))=\sigma\frac{a^2+b^2-c^2}{2ab}\,,\]
converges, and they converge to the same value, and similarly for $K<0$.
\end{pr}

%\begin{cor}\label{cor-ang-ind-K}
%  Let \Xll be a localizing \LpLSn, where $\tau$ is finite-valued and let $K\in\R$. Let $\alpha,\beta\colon[0,\eps)\rightarrow X$ be two future directed timelike geodesics with $\alpha(0)=\beta(0)=:x$. Then
% \begin{equation}
%  \ma_x(\alpha,\beta)=\limsup_{(s,t)\in A_K;\,s,t\searrow 0} \tilde\ma_x^K(\alpha(s),\beta(t))\,,
% \end{equation}
% i.e., the notion of upper angle does not depend on which comparison angle is used.
%\end{cor}

\section{Angles between timelike curves of the same time orientation}\label{sec-ang-same-to}

In this section we study the angle between curves and, in particular, geodesics that have the same time orientation, i.e., are all future or past directed. Some proofs of the results here are given later in Section \ref{sec-ang-arb-to}, where these results are proven in full generality, i.e., where the curves can have a different time orientation. However, no logical issues arise from this.
\medskip

We start by establishing that (upper) angles satisfy the (usual) triangle inequality if all the curves have the same time orientation.

\begin{thm}[Triangle inequality for (upper) angles]\label{thm-ang-tri-equ}
 Let \Xll be a strongly causal and locally causally closed \LpLS with $\tau$ locally finite-valued and locally continuous. Let $\alpha,\beta,\gamma\colon[0,B)\rightarrow X$ be  timelike curves with coinciding time orientation starting at $x:=\alpha(0)=\beta(0)=\gamma(0)$. Then
 \begin{equation}
  \ma_x(\alpha,\gamma)\leq \ma_x(\alpha,\beta) + \ma_x(\beta,\gamma)\,.
 \end{equation}
\end{thm}
\begin{pr}
 This proof is a direct adaptation of the proof of \cite[Prop.\ I.1.14]{BH:99} or of \cite[Thm.\ 3.6.34]{BBI:01} in the metric setting, with several adaptions needed to keep track of the causality.

We establish the case where $\alpha,\beta,\gamma$ are all future directed. The prospective inequality is only non-trivial if $\ma_x(\alpha,\beta)<\infty$ and $\ma_x(\beta,\gamma)<\infty$ --- so let that be the case. Moreover, assume to the contrary that the inequality does not hold, so that there is an $\varepsilon>0$ such that
  \begin{equation}
  \ma_x(\alpha,\gamma)> \ma_x(\alpha,\beta) + \ma_x(\beta,\gamma) + 4\varepsilon\,.
 \end{equation}
By definition of the upper angles, there is a $B'\in(0,B)$ small enough such that the initial segments (defined on $[0,B']$) of $\alpha,\beta,\gamma$ are contained in a causally closed neighborhood, where $\tau$ is finite and continuous, and for all $r,t,s\in(0,B')$ we have the upper bound
 \begin{align}
  \tilde\ma_x(\alpha(r),\beta(s))&<\ma_x(\alpha,\beta) + \varepsilon\,,\\
  \tilde\ma_x(\beta(s),\gamma(t))&<\ma_x(\beta,\gamma) + \varepsilon\,.
 \end{align}
whenever the comparison angles exist, i.e., whenever $\alpha(r)\leq\beta(s)$ and $\beta(s)\leq\gamma(t)$ or the other way round.
 
Furthermore, there are sequences $\tilde r_n\searrow 0$ and $t_n\searrow 0$ with 
 \begin{align}
  \tilde\ma_x(\alpha(\tilde r_n),\gamma(t_n))&>\ma_x(\alpha,\gamma) - \frac{\varepsilon}{2}\,.
 \end{align}
 if $\ma_x(\alpha,\gamma)<\infty$, otherwise we can have that $\tilde\ma_x(\alpha(\tilde r_n),\gamma(t_n))>C$ for any $C>0$.
 
In particular $\alpha(\tilde r_n),\gamma(t_n)$ are causally related, without loss of generality we assume $\alpha(r_n)\leq \gamma(t_n)$ for all $n$. For $n\in\N$ let $0<r_n < \tilde r_n$ be so close to $\tilde r_n$ that by the (local) continuity of $\tau$ and the continuity in the law of cosines we get 
 \begin{align}\label{eq-lem-ang-tri-ine-rt}
  \tilde\ma_x(\alpha(r_n),\gamma(t_n))&>\ma_x(\alpha,\gamma) - \varepsilon\,,
 \end{align}
 and $\alpha(r_n)\ll \gamma(t_n)$ for all $n\in\N$.

We define
\begin{align*}
s^n_-&:=\inf\{\beta^{-1}(I^+(\alpha(r_n)))\}\,,\\
s^n_+&:=\sup\{\beta^{-1}(I^-(\gamma(t_n)))\}\,.
\end{align*}

For large enough $n$, we have $s^n_-<B'$ and $s^n_+<B'$: $I^-(\beta(\frac{B'}{2}))$ contains an initial segment of $\alpha$ and $\gamma$, thus $\alpha(r_n),\gamma(r_n)\in I^-(\beta(\frac{B'}{2}))$ for large enough $n$. Then we have $\beta(\frac{B'}{2})\in I^+(\alpha(r_n))\cap I^+(\gamma(t_n))$, thus $s^n_-\leq\frac{
B'}{2}$. As $X$ is chronological, we have $\beta(s)\not\ll\gamma(t_n)$ for $s\geq\frac{B'}{2}$, thus $s^n_+\leq\frac{B'}{2}$. We pick a particular such $n$ and will drop the subscript $n$ from now on.

We distinguish the cases $s_-<s_+$ and $s_-\geq s_+$. Consider first the case $s_-<s_+$: It is easy to see that $\beta^{-1}(I(\alpha(r),\gamma(t)))$ is the open interval $(s_-,s_+)$. For $s\in(s_-,s_+)$, we get a \emph{timelike tetrad} $x\ll\alpha(r)\ll\beta(s)\ll\gamma(t)$.
 
 Set $a:=\alpha(r),b(s):=\beta(s),c:=\gamma(t)$ and note that by construction, our indirect assumption reads
 \begin{equation}\label{eq-lem-ang-tri-ine-con}
  \tilde\ma_x(a,c)>\tilde\ma_x(a,b(s))+\tilde\ma_x(b(s),c)\,,
 \end{equation}
for all $s\in(s_-,s_+).$

We now create a comparison situation where all side lengths and two of the angles are realized: Choose comparison triangles for $\Delta xab(s)$ and $\Delta xb(s)c$ in the Minkowski plane $\mb{R}^2_1$, which share the side $[\bar x\bar b_s]$ and such that the points $\bar a_s$ and $\bar c_s$ are on different sides of the line $[\bar{x}\bar{b}_s]$, i.e., we pick points $\bar{x},\bar{a}_s,\bar{b}_s,\bar{c}_s$, such that $\tau(x,a)=\bar\tau(\bar x,\bar a_s)$, $\tau(x,b(s))=\bar\tau(\bar x,\bar b_s)$, $\tau(x,c)=\bar\tau(\bar x,\bar c_s)$, $\tau(a,b(s))=\bar\tau(\bar a_s,\bar b_s)$, $\tau(b(s),c)=\bar\tau(\bar b_s,\bar c_s)$, see Figure \ref{fig-flattened_tetrahedron}. This realizes the angles $\tilde{\ma}_x(a,b(s))$ and $\tilde{\ma}_x(b(s),c)$. Note that $\bar a_s,\bar b_s,\bar c_s$ may depend on $s$ in general. This choice of points can be made continuously in $s$, even when one of the sides becomes null (i.e., when extending $s$ to $s=s_-$ or $s=s_+$).

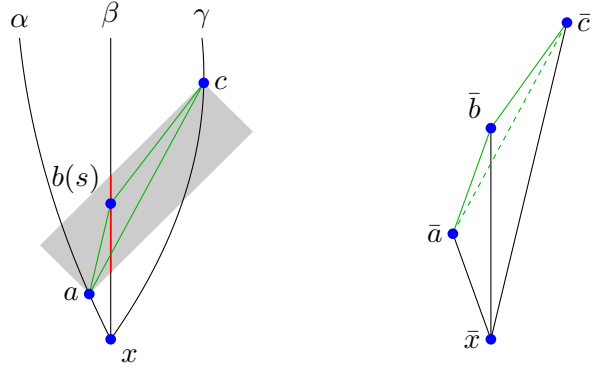
\begin{figure}[h!]
 \begin{center}
\begin{tikzpicture}[x=2cm,y=2cm]

\begin{scope} %the real situation
\fill[color=black!20] ({-0.5*0.3+0.1*0.3*0.3},{0.3}) -- ({(-0.5*0.3+0.1*0.3*0.3-0.3+0.7*1.7-0.2*1.7*1.7+1.7)/2},{(-(-0.5*0.3+0.1*0.3*0.3)+0.3+0.7*1.7-0.2*1.7*1.7+1.7)/2}) -- ({0.7*1.7-0.2*1.7*1.7},{1.7}) -- ({(-0.5*0.3+0.1*0.3*0.3+0.3+0.7*1.7-0.2*1.7*1.7-1.7)/2},{(-0.5*0.3+0.1*0.3*0.3+0.3-(0.7*1.7-0.2*1.7*1.7)+1.7)/2}) -- cycle;
%({a},{b}) -- ({(a-b+c+d)/2},{(-a+b+c+d)/2}) -- ({c},{d}) -- ({(a+b+c-d)/2},{(a+b-c+d)/2}) -- cycle;

\draw[smooth,samples=100,domain=0:2] plot({-0.5*\x+0.1*\x*\x},{\x});%\alpha
\draw[smooth,samples=100,domain=0:2] plot(0,{\x});%\beta
\draw[smooth,samples=100,domain=0:2] plot({0.7*\x-0.2*\x*\x},{\x});%\gamma
\draw[color=red, line width=0.6pt] (0,{0.3-(-0.5*0.3+0.1*0.3*0.3)/(((-0.5*0.3+0.1*0.3*0.3-0.3+0.7*1.7-0.2*1.7*1.7+1.7)/2)-(-0.5*0.3+0.1*0.3*0.3))*(((-(-0.5*0.3+0.1*0.3*0.3)+0.3+0.7*1.7-0.2*1.7*1.7+1.7)/2)-0.3)}) -- (0,{1.7-(0.7*1.7-0.2*1.7*1.7)/(((-0.5*0.3+0.1*0.3*0.3+0.3+0.7*1.7-0.2*1.7*1.7-1.7)/2)-(0.7*1.7-0.2*1.7*1.7))*(((-0.5*0.3+0.1*0.3*0.3+0.3-(0.7*1.7-0.2*1.7*1.7)+1.7)/2)-1.7)});%the allowed piece of \beta
%(a,b) -- (c,d) meets (0,*) in (0,b-a/(c-a)*(d-b))

\draw[color=green!70!black] ({-0.5*0.3+0.1*0.3*0.3},{0.3}) -- ({0.7*1.7-0.2*1.7*1.7},{1.7}) -- (0,0.9) -- cycle;

\fill[color=blue] (0,0) circle(2pt) ({-0.5*0.3+0.1*0.3*0.3},{0.3}) circle(2pt) ({0.7*1.7-0.2*1.7*1.7},{1.7}) circle(2pt) (0,0.9) circle(2pt); %(0,0.9) is some point on the red segment of \beta
\draw (0,0) node[anchor=north west]{$x$} 
({-0.5*2+0.1*2*2},{2}) node[anchor=south]{$\alpha$} 
({0},{2}) node[anchor=south]{$\beta$} 
({0.7*2-0.2*2*2},{2}) node[anchor=south]{$\gamma$} 
({-0.5*0.3+0.1*0.3*0.3},{0.3}) node[anchor=east]{$a$}
(0,0.9) node[anchor=south east]{$b(s)$}
({0.7*1.7-0.2*1.7*1.7},{1.7}) node[anchor=west]{$c$}
;

\end{scope}

\begin{scope}[shift={(2.5,0)}] %the comparison situation:
\draw (0,0) -- (-0.25,0.7) (0,0) -- (0,1.4) (0,0) -- (0.5,2.1);
\draw[color=green!70!black] (-0.25,0.7) -- (0,1.4) (0.5,2.1) -- (0,1.4);
\draw[color=green!70!black,dash pattern=on 2pt off 2pt] (-0.25,0.7) -- (0.5,2.1);
\fill[color=blue] (0,0) circle(2pt) (-0.25,0.7) circle(2pt) (0,1.4) circle(2pt) (0.5,2.1) circle(2pt);
\draw (0,0) node[anchor=east]{$\bar{x}$} (-0.25,0.7) node[anchor=east]{$\bar{a}$} (0,1.4) node[anchor=south east]{$\bar{b}$} (0.5,2.1) node[anchor=west]{$\bar{c}$};
\end{scope}

\end{tikzpicture}
 \end{center}
\caption{On the left, the configuration in $X$ is shown. The grey area is $I(a,c)$, the red line is $\beta([s_-,s_+])$. In the comparison picture on the right, two points are connected by a line if this line has the same length as the corresponding line in $X$. The dashed line need not have the same length.}\label{fig-flattened_tetrahedron}
\end{figure}

We claim that there is an $s\in (s_-,s_+)$ such that $\bar{a}_s,\bar{b}_s,\bar{c}_s$ lie on a line (which then is timelike). For $s\in (s_-,s_+)$ let $L_s$ be the line connecting $\bar{a}_s$ with $\bar{c}_s$. Setting $s=s_-$, we get that $a$ and $b(s_-)$ are null related, but we can still form the comparison situation (denoted as above). In this new comparison configuration, it is obvious that $\bar{b}(s_-)=:\bar b_-$ lies below $L_{s_-}=:L_-$ as both $\bar{b}_-$ and $\bar{c}_s$ lie on the same side of $[\bar x \bar{a}_s]$, and $\bar{b}_-$ is in the causal future of $\bar a_s$ and $\bar{c}_s$ is in the timelike future of $\bar{a}_s$ as the family of lines $(L_s)_s$ are causal and hence cannot tilt too much, see Figure \ref{fig-flattened_tetrahedron_null_above_below}.

Similarly, setting $s=s_+$, we get that $b(s_+)$ and $c$ are null related. Here, it is obvious that $\bar{b}_{s_+}$ lies above $L_{s_+}$.

Now, by continuity of all the points, there is a value $s\in(s_-,s_+)$ such that $\bar{b}_s$ lies on $L_s$.
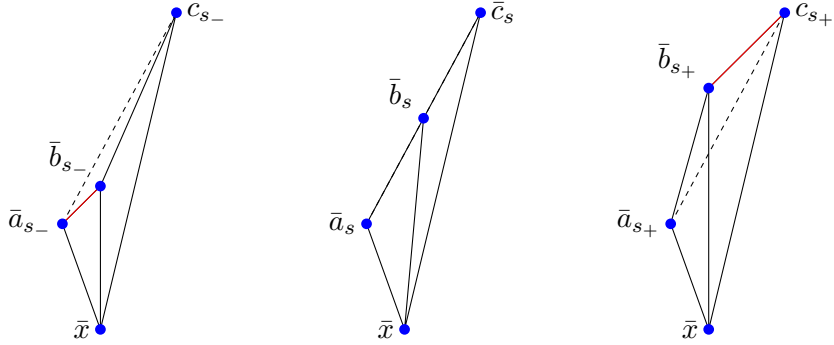
\begin{figure}[h!]
\begin{center}
\begin{tikzpicture}[x=2cm,y=2cm]

\begin{scope}[xshift=0cm]
\draw (0,0) -- (-0.25,0.7) -- (0,0.95) -- cycle -- (0.5,2.1) -- (0,0.95);
\draw[color=red] (-0.25,0.7) -- (0,0.95);
\draw[dash pattern=on 2pt off 2pt] (-0.25,0.7) -- (0.5,2.1);
\fill[color=blue] (0,0) circle(2pt) (-0.25,0.7) circle(2pt) (0,0.95) circle(2pt) (0.5,2.1) circle(2pt);
\draw (0,0) node[anchor=east]{$\bar{x}$} (-0.25,0.7) node[anchor=east]{$\bar{a}_{s_-}$} (0,0.95) node[anchor=south east]{$\bar{b}_{s_-}$} (0.5,2.1) node[anchor=west]{$\bar{c}_{s_-}$};
\end{scope}

\begin{scope}[shift={(2,0)}]
\draw (0,0) -- (-0.25,0.7) -- ({(-0.25+0.5)/2},{(0.7+2.1)/2}) -- cycle -- (0.5,2.1) -- ({(-0.25+0.5)/2},{(0.7+2.1)/2});
\draw[dash pattern=on 2pt off 2pt] (-0.25,0.7) -- (0.5,2.1);
\fill[color=blue] (0,0) circle(2pt) (-0.25,0.7) circle(2pt) ({(-0.25+0.5)/2},{(0.7+2.1)/2}) circle(2pt) (0.5,2.1) circle(2pt);
\draw (0,0) node[anchor=east]{$\bar{x}$} (-0.25,0.7) node[anchor=east]{$\bar{a}_s$} ({(-0.25+0.5)/2},{(0.7+2.1)/2}) node[anchor=south east]{$\bar{b}_s$} (0.5,2.1) node[anchor=west]{$\bar{c}_s$};
\end{scope}

\begin{scope}[shift={(4,0)}]
\draw (0,0) -- (-0.25,0.7) -- (0,1.6) -- cycle -- (0.5,2.1) -- (0,1.6);
\draw[color=red] (0.5,2.1) -- (0,1.6);
\draw[dash pattern=on 2pt off 2pt] (-0.25,0.7) -- (0.5,2.1);
\fill[color=blue] (0,0) circle(2pt) (-0.25,0.7) circle(2pt) (0,1.6) circle(2pt) (0.5,2.1) circle(2pt);
\draw (0,0) node[anchor=east]{$\bar{x}$} (-0.25,0.7) node[anchor=east]{$\bar{a}_{s_+}$} (0,1.6) node[anchor=south east]{$\bar{b}_{s_+}$} (0.5,2.1) node[anchor=west]{$\bar{c}_{s_+}$};
\end{scope}
\end{tikzpicture}
\end{center}
\caption{As the red lines are null and the dashed lines are timelike, we deduce that $\bar{b}_{s_-}$ is below $L_{s_-}$ and $\bar{b}_{s_+}$ is above $L_{s_+}$.}\label{fig-flattened_tetrahedron_null_above_below}
\end{figure}

Now we use this $s$ and drop the $s$-subscript, i.e., $\bar b:=\bar b_s$. Then we have $\bar\tau(\bar{a},\bar{c})=\bar\tau(\bar{a},\bar{b})+\bar\tau(\bar{b},\bar{c})$ as they lie on a straight line (a timelike Minkowski-geodesic). We choose $\tilde{c}$ such that $\Delta \bar{x} \bar{a} \tilde{c}$ is a comparison triangle of $\Delta xac$, with $\bar{c}$ and $\tilde{c}$ on the same side of $[x\bar{a}]$. In this situation, the angles under consideration are given by
\begin{itemize}
\item $\tilde{\ma}_x(a,b)={\ma}_{\bar{x}}(\bar{a},\bar{b})$,
\item $\tilde{\ma}_x(b,c)={\ma}_{\bar{x}}(\bar{b},\bar{c})$,
\item $\tilde{\ma}_x(a,c)={\ma}_{\bar{x}}(\bar{a},\tilde{c})$.
\end{itemize}
By the reverse triangle inequality, we know that $\tau(a,c)\geq\tau(a,b)+\tau(b,c)$, so $\bar{\tau}(\bar{a},\tilde{c})\geq\tau(\bar{a},\bar{b})+\tau(\bar{b},\bar{c})=\tau(\bar{a},\bar{c})$. Thus the triangles $\Delta\bar{x}\bar{a}\bar{c}$ and $\Delta\bar{x}\bar{a}\tilde{c}$ have all side lengths equal except possibly $[\bar{a}\bar{c}]$ and $[\bar{a}\tilde{c}]$, where we know $[\bar{a}\tilde{c}]$ is not the shorter one. Consequently, by the monotonicity statement in the law of cosines (Remark \ref{rem-loc-mon} with $\sigma=-1$, $K=0$), we conclude that ${\ma}_{\bar{x}}(\bar{a},\tilde{c})\leq{\ma}_{\bar{x}}(\bar{a},\bar{c})={\ma}_{\bar{x}}(\bar{a},\bar{b})+{\ma}_{\bar{x}}(\bar{b},\bar{c})$, where the last equality is due to the additivity of hyperbolic angles in the plane. This is a contradiction to the assumption, i.e., Equation \eqref{eq-lem-ang-tri-ine-con}.

At this point we consider the case $s_-\geq s_+$. Now we may not get a single point on $\beta$, but we may get two of them: Set $b_\pm=\beta(s_\pm)$. Then we have $x\ll\alpha(r)\leq b_-$ and $x\ll b_+\leq\gamma(t)$. We can form the following comparison situation:

For the triangle $\Delta x \alpha(r) \gamma(t)$ we get a comparison triangle $\Delta \bar x \bar{a} \bar{c}$. In the same diagram, we add a comparison triangle for $\Delta x \alpha(r) b_-$ and call the additional point $\bar b_-$, and add a comparison triangle for $\Delta x b_+ \gamma(t)$ and call the additional point $\bar b_+$. We set it up so that $\bar{a}$ is to the left of $\bar{c}$, $\bar b_-$ is to the right of $\bar{a}$ and $\bar b_+$ is to the left of $\bar{c}$, see Figure \ref{fig-flattened_tetrahedron_two_bs}. We again define $L$ as the (future directed timelike) straight line from $\bar{a}$ to $\bar{c}$. For the angles we get that
\begin{align*}
\omega_1&:=\tilde\ma_{\bar{x}}(\bar{a},\bar{b}_-)<\ma_x(\alpha,\beta)+\varepsilon\,,\\
\omega_2&:=\tilde \ma_{\bar{x}}(\bar{b}_+,\bar{c})<\ma_x(\beta,\gamma)+\varepsilon\,,\\
\omega_3&:=\tilde\ma_{\bar{x}}(\bar{a},\bar{c})>\ma_x(\alpha,\gamma)-\varepsilon\,.
\end{align*}

\begin{figure}[h!]
 \begin{center}
\begin{tikzpicture}[x=2cm,y=2cm]
%the first scope is not search+replace-able, see "calcd"

\begin{scope} %the real situation

\fill[color=black!20] ({(-0.5*0.3+0.1*0.3*0.3)},{0.3}) -- ({((-0.5*0.3+0.1*0.3*0.3)-0.3+0+0.9)/2},{(-(-0.5*0.3+0.1*0.3*0.3)+0.3+0+0.9)/2}) -- ({0},{0.9}) -- ({((-0.5*0.3+0.1*0.3*0.3)+0.3+0-0.9)/2},{((-0.5*0.3+0.1*0.3*0.3)+0.3-0+0.9)/2}) -- cycle;%({a},{b}) -- ({(a-b+c+d)/2},{(-a+b+c+d)/2}) -- ({c},{d}) -- ({(a+b+c-d)/2},{(a+b-c+d)/2}) -- cycle;

\draw[smooth,samples=100,domain=0:2] plot({-0.5*\x+0.1*\x*\x},{\x});%\alpha
\draw[smooth,samples=100,domain=0:2] plot(0,{\x});%\gamma
\draw[smooth,samples=100,domain=0:2] plot({0.7*\x-0.2*\x*\x},{\x});%\beta

\draw[color=green!70!black] ({-0.5*0.3+0.1*0.3*0.3},{0.3}) -- (0,0.9);
\draw[color=red] (0.47258047596141257430774001539149,0.91358047596141257430774001539149) -- ({-0.5*0.3+0.1*0.3*0.3},{0.3}) (0,0.9) -- (0.33272996566405883964384364758751,0.56727003433594116035615635241249);%calcd

\fill[color=blue] (0,0) circle(2pt) ({-0.5*0.3+0.1*0.3*0.3},{0.3}) circle(2pt) 
%({0.7*1.7-0.2*1.7*1.7},{1.7}) circle(2pt) 
(0,0.9) circle(2pt)
(0.33272996566405883964384364758751,0.56727003433594116035615635241249) circle(2pt) %calcd, c+(y,-y) meets \beta
(0.47258047596141257430774001539149,0.91358047596141257430774001539149) circle(2pt) %calcd, a+(y,y) meets \beta
; 
\draw (0,0) node[anchor=north west]{$x$} 
({-0.5*2+0.1*2*2},{2}) node[anchor=south]{$\alpha$} 
({0},{2}) node[anchor=south]{$\gamma$} 
({0.7*2-0.2*2*2},{2}) node[anchor=south]{$\beta$} 
({-0.5*0.3+0.1*0.3*0.3},{0.3}) node[anchor=east]{$a$}
(0,0.9) node[anchor=south east]{$c$}
%({0.7*1.7-0.2*1.7*1.7},{1.7}) node[anchor=west]{$b$}
(0.33272996566405883964384364758751,0.56727003433594116035615635241249) node[anchor=west]{$b_+$}%calcd
(0.47258047596141257430774001539149,0.91358047596141257430774001539149) node[anchor=west]{$b_-$}%calcd
;

\end{scope}

\begin{scope}[shift={(2.5,0)},x=3cm,y=3cm]

\draw (0,0) -- (-0.25,0.7) (0,0) -- (0.5,2.1) (0,0) -- ({-0.25+0.25},{0.7+0.25}) (0,0) -- ({0.5-0.25},{2.1-0.25});
\draw[color=green!70!black] (-0.25,0.7) -- (0.5,2.1);
\draw[color=red] (-0.25,0.7) -- ({-0.25+0.25},{0.7+0.25}) (0.5,2.1) -- ({0.5-0.25},{2.1-0.25});

\fill[color=blue] (0,0) circle(2pt) (-0.25,0.7) circle(2pt) (0.5,2.1) circle(2pt) ({-0.25+0.25},{0.7+0.25}) circle(2pt) ({0.5-0.25},{2.1-0.25}) circle(2pt);
\draw (0,0) node[anchor=east]{$\bar{x}$} (-0.25,0.7) node[anchor=east]{$\bar{a}$} (0.5,2.1) node[anchor=west]{$\bar{c}$} ({-0.25+0.25},{0.7+0.25}) node[anchor=south]{$\bar b_-$} ({0.5-0.25},{2.1-0.25}) node[anchor=south]{$\bar b_+$};

\coordinate (x) at (0,0);
\coordinate (a) at (-0.25,0.7);
\coordinate (c) at (0.5,2.1);
\coordinate (bm) at ({-0.25+0.25},{0.7+0.25});
\coordinate (bp) at ({0.5-0.25},{2.1-0.25});

\draw 
pic[draw,angle radius={1.5*0.4cm},color=orange]{angle = c--x--a} 
pic[draw,angle radius={1.5*0.8cm},color=orange]{angle = bm--x--a} 
pic[draw,angle radius={1.5*0.8cm},color=orange]{angle = c--x--bp} 
;
\draw[color=orange] 
({0.5*0.2/sqrt(0.5*0.5+2.1*2.1)},{2.1*0.2/sqrt(0.5*0.5+2.1*2.1)}) node[anchor=west]{$\omega_3$}
({0.5*0.4/sqrt(0.5*0.5+2.1*2.1)},{2.1*0.4/sqrt(0.5*0.5+2.1*2.1)}) node[anchor=west]{$\omega_2$}
({-0.25*0.4/sqrt(0.25*0.25+0.7*0.7)},{0.7*0.4/sqrt(0.25*0.25+0.7*0.7)}) node[anchor=east]{$\omega_1$}
;

\end{scope}

\end{tikzpicture}
 \end{center}
\caption{As the red lines are null and the green line is timelike, $b_-$ is below $L$ and $b_+$ is above $L$.}\label{fig-flattened_tetrahedron_two_bs}
\end{figure}
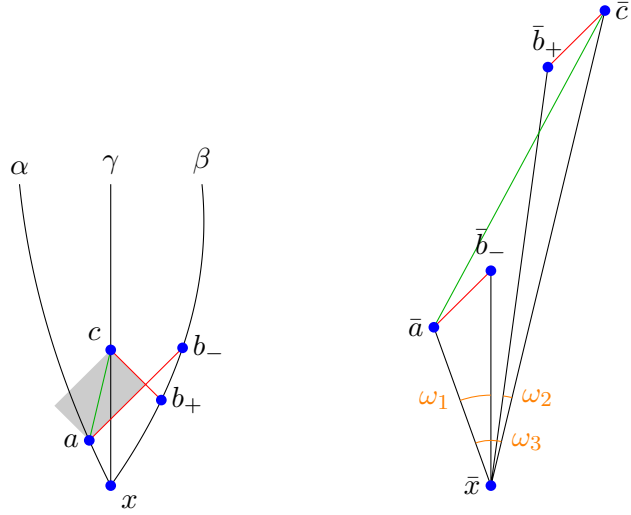

By assumption, we obtain $\omega_3>\omega_1+\omega_2+\varepsilon$, thus the line $[\bar{x}\bar{b}_-]$ is to the left of $[\bar{x}\bar{b}_+]$ (this follows from angle additivity in the plane). As $[\bar{a}\bar{b}_-]$ is future directed null and $L=[\bar{a}\bar{c}]$ is future directed timelike and both go to the right, $\bar{b}_-$ lies below $L$. As $[\bar{b}_+ \bar{c}]$ is null and $L=[\bar{a}\bar{c}]$ is timelike and both go to the right, $\bar{b}_+$ lies above $L$. 

We now construct a future directed timelike curve from $\bar{b}_-$ to $\bar{b}_+$ in the comparison space (that means $b_+\leq b_-$, but $\bar{b}_+\gg\bar{b}_-$): 

We begin at $\bar{b}_-$. As $\bar{b}_-$ lies below $L$, we use the extension of the line $[\bar{x}\bar{b}_-]$ to reach the line $L$, then go along a part of $L$. As $\bar{b}_+$ lies above $L$, we use the end of the line $[\bar{x}\bar{b}_+]$ to reach $\bar{b}_+$. 

In particular, $\bar{\tau}(\bar{b}_-,\bar{b}_+)>0$, and by the reverse triangle inequality $\tau(x,b_-)=\bar{\tau}(\bar{x},\bar{b}_-)<\bar{\tau}(\bar{x},\bar{b}_+)=\tau(x,b_+)$. However, in $X$ we have $\tau(x,b_-)\geq\tau(x,b_+)$, a contradiction. 

Finally, note that in case $\ma_x(\alpha,\gamma)=\infty$, one obtains analogous contradictions in the cases $s_-<s_+$ and $s_-\geq s_+$, as the comparison angle $\tilde\ma_{x}(a,c)$ is finite but arbitrarily large.
\end{pr}

\begin{pop}[Angle bound implies quantitative timelike relation] \label{angleBoundImpliesTimelikeRelation}
Let $X$ be chronological, $\tau$ be locally continuous, and let $\alpha,\beta:[0,\eta)\to X$ be two future directed timelike geodesics parametrized by $\tau$-arclength with $x:=\alpha(0)=\beta(0)$ such that $\ma_x(\alpha,\beta)=:\omega<\infty$.

Then for all $0<c<e^{-\omega}$ there is a $T>0$ small enough such that for all $0<t<T$ we have $\alpha(ct)\ll\beta(t)$.

% This also works if $\alpha,\beta$ are not geodesics parametrized by $\tau$-arclength, but requires that $\tau(x,\alpha(s))=c\cdot \tau(x,\beta(t))$. \todo{I now see a possible problem when defining angles for non-geodesics: in the definition, one could (should?) exchange $\tau(x,\alpha(s))$ by the $\tau$-arclength $L_\tau(\alpha|_{[0,s]})$ (and require $\alpha,\beta$ to be rectifiable).}
\end{pop}
\begin{proof}
By the definition of upper angles, we get that for all $\varepsilon>0$ there is a $\delta>0$ such that for all $s<\delta$, $t<\delta$ such that $\alpha(s)\ll\beta(t)$, $\tilde{\ma}_x(\alpha(s),\beta(t))<\omega+\varepsilon$.

Let $T_1=\min(\inf(\beta^{-1}(I^+(\alpha(\delta))),\delta)$. Then for $t<T_1$ and $s$ such that $\alpha(s)\ll\beta(t)$, we automatically have $s,t<\delta$ (and $T_1>0$ as $X$ is chronological and $\tau$ continuous).

For each $t$, we consider the set of parameters $s$ such that $\alpha(s)\ll\beta(t)$, i.e., $\alpha^{-1}(I^-(\beta(t)))=[0,s_t)\subseteq[0,\eta)$. 
% Note it contains $0$. By the reverse triangle inequality, it is an interval, and as $I^-(\beta(t))$ is open, it is a half-open interval (i.e.\ open as a subset of $[0,\eta)$). Thus, it has the form $S_t=[0,s_t)$ for some $s_t$.
By continuity of $\tau$, we have $\tau(\alpha(s),\beta(t))\to 0$ as $s\nearrow s_t$. 

Now by the definition of comparison angles we obtain
\begin{align}
\cosh(\omega+\varepsilon)&>\cosh(\tilde{\ma}_x(\alpha(s),\beta(t)))\\
&=\frac{s^2+t^2-\tau(\alpha(s),\beta(t))^2}{2\,s\,t}\,.
\end{align}
Again by continuity of $\tau$ this also holds in the limit $s\nearrow s_t$, i.e.,
\begin{equation}
 \cosh(\omega+\varepsilon)\geq\frac{s_t^2+t^2}{2\,s_t\,t}\,.
 \end{equation}
For simplicity, we set
% $a:=\tau(x,\alpha(s_t))$, $b:=\tau(x,\beta(t))$,
$o:=\cosh(\omega+\varepsilon)$ and get the quadratic inequality
\begin{equation*}
0\geq s_t^2-2o s_t t+t^2\,,
\end{equation*}
yielding $t(o-\sqrt{o^2-1})\leq s_t\leq t(o+\sqrt{o^2-1})$. Note $o-\sqrt{o^2-1}=\cosh(\omega+\varepsilon)-\sqrt{\cosh(\omega+\varepsilon)^2-1}=\cosh(\omega+\varepsilon)-\sinh(\omega+\varepsilon)=\exp(-\omega-\varepsilon)$.
 In particular, we have that
\begin{equation}
s_t\geq t \exp(-\omega-\varepsilon)\,.
\end{equation}
That is, for $s<t \exp(-\omega-\varepsilon)$ we have $s<s_t$ and thus $\alpha(s)\ll\beta(t)$ by construction.

Finally, for all $c<e^{-\omega}$ we set $\varepsilon:=-\log(ce^\omega)>0$ and obtain $\alpha(s)\ll\beta(t)$ for all $0<t < T_1=T_1(\varepsilon)$, yielding the claim.
\end{proof}

At this point we introduce the notion of \emph{direction} of a timelike geodesic.

\begin{defi}[Direction]
Let $\Xll$ be a \LpLS with $\tau$ locally finite-valued.
For $x\in X$ we define $D_x^+:=\{\alpha\colon[0,\eps)\rightarrow X$ future directed timelike geodesic with $\alpha(0)=x$ for some $\eps>0\}$. We define a relation $\sim$ on $D_x^+$ by saying $\alpha\sim\beta$ if $\ma_x(\alpha,\beta)=0$. If $\alpha\sim\beta$, we say that $\alpha$ and $\beta$ \emph{have the same direction at $x$}. Similarly, we define $D_x^-:=\{\alpha\colon[0,\eps)\rightarrow X$ past directed timelike geodesic with $\alpha(0)=x$ for some $\eps>0\}$, and define a relation on $D_x^-$ by $\alpha\sim\beta$ if $\ma_x(\alpha,\beta)=0$.
\end{defi}

\begin{lem}[Properties of angles and direction]\label{lem-ang-geo}
 Let $\Xll$ be a strongly causal and locally causally closed \LpLS with $\tau$ locally finite-valued.
 \begin{enumerate}
  \item\label{lem-ang-geo-same} Let $\alpha\colon [a,b)\rightarrow X$ be a future directed timelike geodesic, let $t_0\in [a,b)$ and set $x:=\alpha(t_0)$, then $\ma_x(\alpha,\alpha)=0$. 
  \item\label{lem-ang-geo-int} Let $\alpha\colon (a,b)\rightarrow X$ be a future directed timelike geodesic, let $t_0\in (a,b)$ and set $x:=\alpha(t_0)$. Then $\ma_x(\alpha,\alpha^-)=0$, where $\alpha^-(t)=\alpha(-t)$ ($t\in (-b,-a)$) is the reversed curve, which is a past directed timelike geodesic.
  \item Let $x\in X$, the relation $\sim$ of having the same direction at $x$ is an equivalence relation on $D_x^\pm$.
 \end{enumerate}
\end{lem}
\begin{pr}
 \begin{enumerate}
  \item[(i),(ii)] \setcounter{enumi}{2} Let $t_0<s<t<b$ or $a<s<t_0<t<b$ be close enough to $t_0$ such that $\alpha$ maximizes on $[\min(t_0,s),t]$. Then the triangle $x\ll \alpha(s)\ll\alpha(t)$ or $\alpha(s)\ll x\ll \alpha(t)$ is degenerate and the corresponding comparison triangle $\ct$ in $\mi=\lm{0}$ is a straight timelike segment, and so $\tilde\ma_x(\alpha(s),\alpha(t))=\ma_{\bar x}^{\lm{0}}(\bar y,\bar z)=0$. Thus, the angle between $\alpha$ and itself exists and is zero.
  \item Symmetry of $\sim$ is clear from the definition and reflexivity follows from point (i) above. To show transitivity let $\alpha,\beta,\gamma\in D_x^+$ with $\alpha\sim\beta$ and $\beta\sim\gamma$. Then by the triangle inequality for upper angles (Theorem \ref{thm-ang-tri-equ}) we get that
  \begin{equation}
   0\leq \ma_x(\alpha,\gamma)\leq \ma_x(\alpha,\beta)+\ma_x(\beta,\gamma)=0\,,
  \end{equation}
which shows that the angle between $\alpha$ and $\gamma$ exists and equals zero. So we have $\alpha\sim\gamma$, as required.
 \end{enumerate}
\end{pr}

Let us emphasize point \ref{lem-ang-geo-int} above: Contrary to the metric case the hyperbolic angle between the incoming and outgoing segments of an interior point of a timelike geodesic is zero and not $\pi$, cf.\ Figure \ref{fig-zeroAngle}.
\medskip

The notion of direction allows us to define the space of (future or past directed) timelike directions at a point.
\begin{lem}
Let $\Xll$ be a strong\-ly causal, locally causally closed \LpLS with $\tau$ locally finite-valued and so that angles between (future directed) timelike geodesics always exist, and $x\in X$. Then $\D_x^+:=D_x^+\slash\sim$ and $\D_x^-:=D_x^-\slash\sim$ are metric spaces with metric $\ma_x$.
\end{lem}
\begin{defi}[Space of timelike directions]
  Let $\Xll$ be a strong\-ly causal, locally causally closed \LpLS with $\tau$ locally finite-valued and so that angles between (future directed) timelike geodesics always exist. We call the metric space $(\D_x^+,\ma_x)$ the \emph{space of future directed timelike directions} at $x$, and $(\D_x^-,\ma_x)$ the \emph{space of past directed timelike directions} at $x$. Furthermore, we denote the metric completion of $(\D_x^\pm,\ma_x)$ by $(\Sigma_x^\pm,\ma_x)$.
\end{defi}

This definition makes sense as the angle $\ma_x$ is symmetric and satisfies the triangle inequality by Theorem \ref{thm-ang-tri-equ}. Moreover, the angle does not depend on the choice of representative and then by construction $\ma_x$ is positive definite.
\bigskip

For technical reasons it is useful to also introduce the completion $\Sigma_x^\pm$ of the space of timelike directions $\D_x^\pm$. However, we will later show in Proposition \ref{pop-D-plus-complete} that for a large class of \LpLSn s with a lower curvature bound the space of timelike directions $\D_x^\pm$ is already complete.
\bigskip

At this point we use angles to introduce another notion of curvature bound by demanding that the comparison angles behave monotonically along the sides of a timelike geodesic triangle. The analogous concept in metric geometry is called \emph{monotonicity condition} in \cite[Def.\ 4.3.1]{BBI:01}.
\begin{defi}[Future $K$-monotonicity comparison]\label{def-ang-com-fu}
 Let $\Xll$ be a \LpLS and let $K\in\R$. We say that $X$ satisfies \emph{future timelike $K$-monotonicity comparison from below (above)} if every point in $X$ possesses a neighborhood $U$ such that
\begin{enumerate}
\item $\tau|_{U\times U}$ is finite and continuous.
\item $U$ is strictly timelike geodesically connected.
\item Let $x\ll y\ll z$ be timelike related and forming a timelike geodesic triangle in $U$, whose side lengths satisfy timelike size bounds for $K$. Let $z'\in[x,z]$ and $y'\in[x,y]$ with $y'\neq x\neq z'$ and $z'$ causally related to $y'$, then we have
\begin{equation}
 \tilde\ma_x^K(y',z')\geq \tilde\ma_x^K(y,z)\qquad (\tilde\ma_x^K(y',z')\leq \tilde\ma_x^K(y,z))\,.
\end{equation}
%Additionally, in case of $K$-monotonicity comparison from above we require that if $y'$ and $z'$ are not causally related then $\bar y'$ and $\bar z'$ are not timelike related in $\lm{K}$, where $\bar y',\bar z'$ are corresponding points of $y',z'$ on any comparison triangle of $\Delta x y z$.
\end{enumerate}
\end{defi}
We can define \emph{past $K$-monotonicity comparison} analogously by choosing $x'\in[x,z]$ and $y'\in[y,z]$ using the comparison angles at $z$ instead. Note that $y'$ is on a different side here as compared to the future case. Moreover, let us remark that $x'$, $y'$ or $z'$ are named as such because they \emph{seem to be in the same direction} as $x$, $y$ or $z$ from the point we are measuring the angle at. For an illustration of $K$-monotonicity comparison see Figure \ref{fig-angle_comparison-fu}.

\begin{figure}[h!]
\begin{center}
\begin{tikzpicture}[line cap=round,line join=round,x=2.0cm,y=2.0cm]
\draw [shift={(0.,-1.)},color=green,fill=green,fill opacity=0.1] (0,0) -- (45.:0.3) arc (45.:116.56505117707799:0.3) -- cycle;
\draw[ smooth,samples=100,domain=-1.0:1.0] plot({\x*\x/4.0-1.0/4.0},\x);
\draw[ smooth,samples=100,domain=0.0:1.0] plot({-\x*\x/4.0-\x/2.0+1.0/2.0+1.0/4.0},\x);
\draw[ smooth,samples=100,domain=-1.0:0.0] plot({-\x*\x/4.0+\x/2.0+1.0/2.0+1.0/4.0},\x);
\draw[color=blue, smooth,samples=100,domain=0.0:0.89564392373896] plot({-\x*\x/4.0-\x/2.0+0.4375},{\x-0.5});

\begin{scope}[xshift=4cm]
\fill [shift={(-0.5,-1.)},draw=green,fill=green!30,] (0,0) -- (74.9598086914095:0.4) arc (74.9598086914095:92.56650651538125:0.4) -- cycle;
\fill [shift={(-0.5,-1.)},draw=green,fill=green!30,] (0,0) -- (63.43494882292201:0.3) arc (63.43494882292201:105.34925915763955:0.3) -- cycle;
\draw[] (-0.5,-1.) -- (0.,0.) -- (-1.0128640109411684,0.868402399749306) -- cycle;
\draw[color=blue] (-0.5418540240528695,-0.06625804298215877) -- (-0.5,-1.) -- (-0.3578657479246279,-0.4710322504642892) -- cycle;
\end{scope}
\end{tikzpicture}
\caption{When comparing with $\lm{K}$ the interior side is shorter than expected --- the $K$-angle has to behave accordingly.}\label{fig-angle_comparison-fu}
\end{center}
\end{figure}
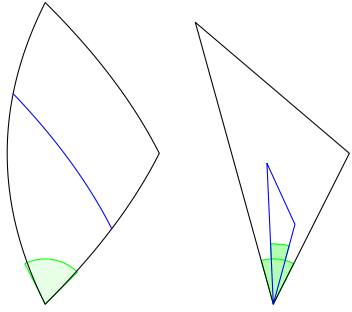

A direct consequence of $K$-monotonicity and Proposition \ref{lorSphericalangleNoDepK} is the following corollary.
\begin{cor}[Monotonicity implies bound on angle of geodesics]\label{cor-K-mon-ang-bou}
Let $\Xll$ be a \LpLS that satisfies future/past $K$-monotonicity comparison from below (above) for some $K\in\R$. Then for any $x\in X$ and $\alpha,\beta\colon[0,B]\rightarrow X$ future/past directed timelike geodesics with $\alpha(0)=\beta(0)=x$ one has that
\begin{equation}
 \ma_x(\alpha,\beta)\geq \tilde\ma_x^ K(\alpha(s),\beta(t)) \qquad \Bigl(\ma_x(\alpha,\beta)\leq \tilde\ma_x^ K(\alpha(s),\beta(t))\Bigr)\,,
\end{equation}
for all $s,t\in[0,B]$ small enough such that timelike size bounds for $K$ are satisfied and chosen such that $\alpha(s)$ and $\beta(t)$ are causally related.
\end{cor}

Moreover, $K$-monotonicity yields that the limit superior in the definition of the angle between geodesics is actually a (not necessarily finite) limit, see e.g. \cite[Prop.\ II.3.1]{BH:99} or \cite[Prop.\ 4.3.2]{BBI:01} for an analogous result in the metric case.
\begin{lem}[Monotonicity implies existence of angles]\label{lem-K-ang-com-ang-ex-fudi}
Let $\Xll$ be a \LpLSn. If it satisfies timelike $K$-monotonicity comparison from above or below, the angle between any two future directed timelike geodesics starting at the same point $x$ exists if it is finite (the limit superior in the definition of upper angles is a limit, but we do not know if it is finite).

If it satisfies timelike $K$-monotonicity comparison from below, we know the limit is finite (in particular, the angle exists).
\end{lem}
\begin{proof}
This will be proven later in more generality in Lemma \ref{lem-K-ang-com-ang-ex}.
\end{proof}

The connection of future $K$-monotonicity comparison to timelike curvature bounds is as follows.
\begin{thm}[Triangle comparison implies future $K$-monotonicity comparison]\label{thm-K-ang-com-fu}
 Let $\Xll$ be a locally strictly timelike geodesically connected Lor\-entzian pre-length space and let $K\in\R$. If $X$ has timelike curvature bounded below (above) by $K$, it also satisfies future $K$-monotonicity comparison from below (above).
\end{thm}
\begin{proof}
This will be proven later in more generality in Theorem \ref{thm-K-ang-com}.
\end{proof}
\bigskip

Analogously to the metric case (cf.\ \cite[Thm.\ 4.3.11]{BBI:01} and \cite[Prop.\ II.3.3]{BH:99}), angles are semi-continuous.
\begin{pop}[Semi-continuity of angles] \label{pop-ang-semicont-fudi}
 Let $\Xll$ be a Lor\-entzian pre-length space with timelike curvature bounded below. Then angles of timelike geodesics of the same time orientation are lower semicontinuous. To be precise, for $x\in X$ and $\alpha,\alpha_n,\beta,\beta_n\colon[0,b]\rightarrow X$ all future or all past directed timelike geodesics starting at $x$ with $\alpha_n\to\alpha$, $\beta_n\to\beta$ pointwise, then
\begin{equation}
\ma_x(\alpha,\beta)\leq \liminf_n \ma_x(\alpha_n,\beta_n)\,. %\qquad (\ma_x(\alpha,\beta)\geq \limsup_n \ma_x(\alpha_n,\beta_n))\,.
\end{equation}
\end{pop}
\begin{pr}
 This will be proven later in more generality in Proposition \ref{pop-ang-semicont}.
\end{pr}

In case of timelike curvature bounded above, the angle between geodesics is in fact continuous.

\begin{pop}[Angle between geodesics is continuous]\label{pop-ang-cont-geodesics-fudi}
Let $\Xll$ be a \LpLS with timelike curvature bounded above. Then angles are continuous for geodesics, i.e., 
for $x\in X$ and $\alpha,\alpha_n,\beta,\beta_n:[0,b]\to X$ all future or all past directed timelike geodesics starting at $x$ with $\alpha_n\to\alpha$, $\beta_n\to\beta$ pointwise, then 
\begin{equation}
\ma_x(\alpha,\beta)= \lim_n\ma_x(\alpha_n,\beta_n)\,.
\end{equation}
\end{pop}
\begin{pr}
  This will be proven later in more generality in Proposition \ref{pop-ang-cont-geodesics}.
\end{pr}

If the curves considered are not geodesics, then there is no semi-continuity of angles, as is well-known in the smooth spacetime case. We include an example below for the convenience of the reader.

\begin{ex}
Let $X=\mb{R}^{2}_1$ be two-dimensional Minkowski spacetime and consider the future directed timelike curves $\gamma_n(t)=(t,\frac{1}{2}\frac{t}{1+nt})$ ($t\in[0,\infty)$). They all start at $\gamma_n(0)=(0,0)=:x$ and  converge uniformly to the curve $\gamma(t)=(t,0)$. Nevertheless, they do not converge in a $\Con^1$-way. We calculate the angles in the spacetime sense (note that by Proposition \ref{hypAnglesAgreeSpt} these angles agree with the angles in the synthetic sense), so $\gamma_n'(0)=(1,\frac{1}{2})$, making $\ma_x(\gamma_n,\gamma)=\arcosh(\frac{2}{\sqrt{3}}) > 0=\ma_x(\gamma,\gamma)$, contradicting upper semicontinuity.

For $\beta(t)=(t,\frac{2}{3}t)$ we have $\ma_x(\gamma_n,\beta)=\arcosh(\frac{4}{\sqrt{15}}) < \arcosh(\frac{3}{\sqrt{5}})=\ma_x(\gamma,\beta)$, contradicting lower semicontinuity, but the $\gamma_n$s are not geodesics (although the limits $\beta$ and $\gamma$ are).
\end{ex}

As there is no automatic local existence of timelike geodesics and, moreover, the time of existence need not be locally uniform, we have to exclude these situations when considering the question of whether the space of timelike directions is complete. For convenience we give a name to the following condition which will be used several times.
\begin{defi}[Non-lingering]\label{def:nonLing}
Let $\Xll$ be a \LpLSn. A point $x\in X$ has the \emph{uniform future/past non-lingering property} if there is an $s>0$ and an $\varepsilon>0$ such that all future/past directed timelike geodesics $\gamma$ starting at $x=\gamma(0)$, when parametrized by $d$-arclength and extended maximally, are defined at least up to parameter $s$, and $d(x,\gamma(s))>\varepsilon$. Moreover, we say that $\Xll$ is \emph{future/past non-lingering} if every point $x\in X$ has the uniform future/past non-lingering property.
\end{defi}
By using cylindrical neighborhoods and the fact that the time coordinate is a temporal function on such a neighborhood (cf.\ \cite[Prop.\ 1.10]{CG:12}), it directly follows that any continuous spacetime is future and past non-lingering.
\bigskip

As pointed out before, given a lower curvature bound, we can give sufficient conditions so that the space of timelike directions is complete. We are now ready to provide a proof of this fact.

\begin{pop}[Space of timelike directions complete]\label{pop-D-plus-complete}
Let $\Xll$ be a locally compact, strongly causal, locally causally closed, $d$-compatible and regularly localizable Lorentzian length space with timelike curvature bou\-nded below. If $x\in X$ has the uniform future non-lingering property, then $\D_x^+$ is a complete metric space, i.e., $\Sigma_x^+=\D_x^+$.
\end{pop}
\begin{proof}
Note that, by strong causality and local compactness of $(X,d)$, small causal diamonds are compact.

Let $\gamma_n:[0,L^d(\gamma_n)]\to X$ be future directed timelike geodesics, $d$-unit-speed-parametrized, with $\gamma_n(0)=x$ and $[\gamma_n]\in \D_x^+$ forming a Cauchy sequence. We can assume by the non-lingering property that all $\gamma_n$ are initially but not completely contained in a compact set $J(x,y)$ (for some $y\gg x$), which is at the same time a $d$-compatible neighborhood. Denoting the initial segments by $\gamma_n\rvert_{[0,b_n]}$, we can assume that $\gamma_n([0,b_n])\subseteq J(x,y)$, $\gamma_n(b_n)\in \partial J(x,y)\cap \partial J^-(y)$ for all $n\in\N$. By being in a $d$-compatible neighborhood, we get that the $b_n$s stay bounded.

By the limit curve theorem \cite[Thm.\ 3.7]{KS:18}, we find a subsequence of $(\gamma_n)_n$ which converges uniformly to a future directed causal curve $\gamma\colon[0,b]\rightarrow X$ which also has $\gamma(b)\in\partial J(x,y)\cap\partial J^-(y)$. Without loss of generality we denote this subsequence again by $(\gamma_{n})_n$ and note that $\gamma$ is non-constant. By \cite[Prop.\ 3.17]{KS:18}, the limit curve $\gamma$ is maximizing. We claim it is timelike and $[\gamma_n]\to[\gamma]$.

Now let $\varepsilon>0$, and $n$ be large enough such that $\ma_x(\gamma_n,\gamma_m)<\varepsilon$ for all $m\geq n$. Let $s,t$ be such that $\gamma(s)\ll\gamma_n(t)$ (these exist: for each $t$ there is an $s>0$ such that this holds as $\gamma$ is future directed causal). As $\gamma_k(s)\to\gamma(s)$, we get that for large enough $k$, also $\gamma_k(s)\ll\gamma_n(t)$ and we can calculate the comparison angle using the law of cosines (Lemma \ref{lorLawOfCosines} with $\sigma=-1$) as follows (we only do the $K=0$ case, the others are similar, just a bit more involved).
\begin{align}\label{eq-pop-D-plus-complete}
\cosh(\tilde{\ma}_x(\gamma_k(s),\gamma_n(t)))&=\\
&\frac{\tau(x,\gamma_k(s))^2+\tau(x,\gamma_n(t))^2-\tau(\gamma_k(s),
\gamma_n(t))^2}{2\tau(x,\gamma_k(s))\tau(x,\gamma_n(t))}\,.
\end{align}
Lemma \ref{lem-K-ang-com-ang-ex-fudi} implies that $X$ satisfies future $K$-monotonicity (where $K$ is the timelike curvature bound from below). Moreover, by Corollary \ref{cor-K-mon-ang-bou} we conclude that $\varepsilon>\ma_x(\gamma_k,\gamma_n)\geq \tilde \ma_x^{K}(\gamma_k(s),\gamma_n(t))$. When taking the limit as $k\to\infty$, we have to be careful. Denoting the enumerator in the left-hand-side of Equation \eqref{eq-pop-D-plus-complete}  by $x_k$ and the denominator by $y_k$, we have $1\leq \frac{x_k}{y_k}<\cosh(\varepsilon)=:1+\tilde\varepsilon$. We have that $y_k\to0$ precisely when $x_k\to0$. If they do not converge to $0$, $\tau(x,\gamma(s))>0$ and thus $\gamma$ is timelike by regular localizability. Thus we have a valid timelike triangle. Taking $k\to\infty$, and then $s,t\to 0$, we get that the upper angle satisfies $\ma_x(\gamma_n,\gamma)\leq\varepsilon$. Letting $\varepsilon\to0$ yields a $\ma_x$-convergent subsequence of $([\gamma_n])_n$ and hence we are done in this case.

Now we exclude the other cases: If $y_k\to 0$, we have $\tau(x,\gamma_k(s))\to 0$ and together with $x_k\to 0$ this implies 
$\tau(\gamma_k(s),\gamma_n(t))\to\tau(x,\gamma_n(t))$. Note that this also works for $K\neq 0$ (e.g.\ we have a $\cosh(\tau(x,\gamma_k(s)))$ factor on one of the terms, which converges to $1$). Denoting the points $u:=\gamma(s)$ and $z:=\gamma_n(t)$, we have $x\leq u$ but $x\not\ll u$ and $\tau(u,z)=\tau(x,z)>0$, as $\gamma_n$ is timelike. If $x\neq u$, we concatenate the geodesics from $x$ to $u$ and from $u$ to $z$ having length $0+\tau(u,z)=\tau(x,z)$, i.e., a distance realizer containing a null segment (the segment from $x$ to $u$), in contradiction to regular localizability. Thus $x=u$ and hence $s=0$ --- a contradiction to $s>0$.
\end{proof}

\subsection{Timelike tangent cones}

In metric geometry the \emph{tangent cone} is a generalization of tangent spaces of smooth manifolds (see e.g.\ \cite[p.\ 321]{BBI:01}) and a valuable tool, especially in Alexandrov spaces with curvature bounded below. Using the space of timelike directions at a point we introduce in an analogous way the \emph{timelike tangent cone} at a point in an \LpLSn. To this end we use the concept of the \emph{Minkowski cone} over a metric space, introduced in \cite[Sec.\ 2]{AGKS:21}. 

\begin{defi}[Minkowski cone]
Let $(Y,d)$ be a metric space. Then the \emph{Minkowski cone} $\mathrm{Cone}(Y)$ of $Y$ is a \LpLSn, where the underlying metric space is $([0,\infty)\times Y)/(\{0\}\times Y)$, equipped with the cone metric $d_c$ (cf.\ \cite[Def.\ 3.6.16]{BBI:01})
\[
d_c((s,p),(t,q))=\begin{cases}
\sqrt{s^2+t^2-2st\cos d(p,q)} &\text{if }d(p,q)\leq \pi\,,\\
s+t&\text{if }d(p,q)\geq \pi\,.
\end{cases}
\]
Moreover, the causal relation $\leq$ and the time separation function $\tau$ are defined via
\[
(s,p)\leq(t,q)\Leftrightarrow s^2+t^2-2st\cosh d(p,q)\geq 0\text{ and }s\leq t\,,
\]
\[\tau((s,p),(t,q))=\sqrt{s^2+t^2-2st\cosh d(p,q)}\text{ if $(s,p)\leq(t,q)$}\,,
\]
and $\ll$ induced by $\tau$ (i.e., $(s,p)\ll(t,q)\Leftrightarrow \tau((s,p),(t,q))>0$). See Section 2 in \cite{AGKS:21} for more details, and how Minkowski cones (without the vertex) can be viewed as instances of \emph{generalized cones} (\cite[Ex.\ 3.31]{AGKS:21}). 
\end{defi}

\begin{rem}[Minkowski cone not localizable]
Proposition 2.2 and Corollary 2.4 in \cite{AGKS:21} establish that the Minkowski cone over a geodesic length space is a geodesic \LpLS and that the time separation $\tau$ is continuous. However, it cannot be localizable as the vertex $0$ is \emph{isolated} with respect to $\ll$, i.e., there is no point $x$ such that $x\ll 0$.
\end{rem}

However, Minkowski cones have localizable neighborhoods at all points except the vertex $0$, which we establish below but first we need the following lemma, whose proof is elementary.

\begin{lem}[Useful properties for Minkowski cones]\label{lem-min-con-use}
Let $(Y,d)$ be a metric space and $X:=\mathrm{Cone}(Y)$ the Minkowski cone over $(Y,d)$. Let $p=(s,y)\leq q=(t,y')$.
\begin{enumerate}
 \item If $s>0$ one has that $d(y,y')\leq \log(t)-\log(s)$.
 \item If $d(y,y')\leq \pi$, then $d_c(p,q)\leq (t-s) + t d(y,y')$.
\end{enumerate}
\end{lem}

\begin{prop}[Minkowski cones nearly localizable]\label{prop-min-con-nea-loc}
Let $(Y,d)$ be a metric space which is strictly intrinsic. Then any $0\neq x\in X:=\mathrm{Cone}(Y)$ has a localizable neighborhood. Moreover, $0\in X$ has a neighborhood $\Omega_0$ that satisfies all the conditions of a localizable neighborhood except that $I^-(0)\cap \Omega_0=\emptyset$. We call such a neighborhood \emph{nearly localizable}. 
\end{prop}
\begin{pr}
  As (nearly) localizable neighborhoods, we choose the sets $\{(t,y):t_-<t<t_+\}$ and $\{(t,y):0\leq t<t_+\}$, respectively, 
for any $0<t_-<t_+$. These neighborhoods are causally convex, and any two causally related points can be connected by a 
$\tau$-distance realizer (\cite[Cor.\ 2.4]{AGKS:21}). Furthermore, by \cite[Prop.\ 2.2]{AGKS:21} $\tau$ is continuous, and 
except at $0$ there are no $\ll$-isolated points. For $d$-compatibility, the $d_c$-length of causal curves within such a neighborhood is at most $4t_+$. This can be shown by using Lemma \ref{lem-min-con-use}:
Parametrize a (future directed) causal curve $\gamma$ from $p$ to $q$ as $\gamma(r)=(r,\beta(r))$, where $r\in[s,t]$ 
and $\beta\colon[s,t]\rightarrow Y$. We can without loss of generality assume that $s=0$, hence $p=0\in\mathrm{Cone}(Y)$ 
(otherwise extend $\gamma$ to $0$ --- this is always possible, but not necessarily within the neighborhood). At this point, let $0=s_0 < s_1 < \ldots < s_M=t$ be a 
partition of $[0,t]$. We choose a refinement $0=s_0=s_0' < s_1' < \ldots < s_{M'}'=t$ of $(s_i)_{i=0}^M$ such that 
$s_{i+1}'\leq 2 s_i'$ and $d(\beta(s_i'),\beta(s_{i+1}'))\leq \pi$ for all $i\neq 0$. Thus
\begin{align}
 \sum_{i=0}^{M-1} d_c(\gamma(s_i),\gamma(s_{i+1})) &\leq d_c(0,\gamma(s_1')) + \sum_{i=1}^{M'-1} 
d_c(\gamma(s_i'),\gamma(s_{i+1}'))\\
&\leq s_1' + \sum_{i=1}^{M'-1} (s_{i+1}'-s_i')\,+ s_{i+1}'\, d(\beta(s_i'),\beta(s_{i+1}'))\\
&\leq 2 t + \sum_{i=1}^{M'-1}s_{i+1}'\,(\log(s_{i+1}')-\log(s_i'))\\
% &= 2 t + \sum_{i=1}^{M'-1} s_{i+1}' \int_{s_i'}^{s_{i+1}'} \frac{1}{r}\,\mathrm{d}r\\
&\leq 2 t + \sum_{i=1}^{M'-1} \frac{s_{i+1}'}{s_i'}\, (s_{i+1}'-s_i')\leq 2t + 2 (t-s_1') \leq 4t\,.
\end{align}
where we used the concavity of the logarithm. Taking the supremum over all partitions of $[0,t]$ yields $L^{d_c}(\gamma)\leq 4 t$ and $t\leq t_+$ if $q$ lies in such a neighborhood.
\end{pr}

Localizability of the Minkowski cone away from the vertex also works if $Y$ is only a locally strictly intrinsic metric space by using timelike diamonds $I((t_-,y),(t_+,y))$ for $y\in Y$ and $t_-,t_+$ close enough to each other so that $\mathrm{pr}_Y(I((t_-,y),(t_+,y)))$ is contained in a strictly intrinsic neighborhood.%=\{(s,\tilde{y}): t_-<s<t_+, \cosh(d(y,\tilde{y}))<\min(\frac{t_-^2+s^2}{2t_-s},\frac{t_+^2+s^2}{2t_+s})\}$.

% These are clearly causally convex and contained in the neighbourhood above, so the above proof still works if $t_+-t_-$ is small enough and we can make the $Y$-component to stay inside a strictly intrinsic neighbourhood. 
% 
% So let $B_\varepsilon(y)\subset Y$ be strictly intrinsic. The $s$-slice $I((t_-,y),(t_+,y))\cap\{(s,\tilde{y})\}$ has maximum $Y$-radius around $y$ at the geometric mean $s=\sqrt{t_-t_+}$ with a radius of $\arcosh(\frac{t_++t_-}{2\sqrt{t_+t_-}})=\arcosh(\frac{1}{2}(\sqrt{\frac{t_+}{t_-}}+\sqrt{\frac{t_-}{t_+}}))$, and for $\frac{t_+}{t_-}$ close enough to $1$, this is $<\varepsilon$. One also sees that this will never work at $t_-=0$.)}

\begin{rem}[Minkowski cone geodesic]
 The Minkowski cone is geodesic precisely when the base $Y$ is. This can be seen easily from \cite[Lem.\ 2.3]{AGKS:21} and adding the vertex. In the metric case an analogous statement holds, see \cite[Thm.\ 3.6.17]{BBI:01}.
\end{rem}

\begin{rem}[Cone metric does not induce Euclidean topology]\label{rem-dc-not-euc-top}
This is a warning that usually, e.g., for spacetimes, the cone metric $d_c$ induces a coarser topology than one might at first expect. In any case we have $d_c(0,p)=\tau(0,p)$ for all $p\in\mathrm{Cone}(Y)$ and hence
\begin{equation}
 B^{d_c}_1(0)=\{p\in \mathrm{Cone}(Y): \tau(0,p)<1\}\,,
\end{equation}
i.e., the hyperboloid $\{p\in \mathrm{Cone}(Y):\tau(0,p)=r\}$ of radius $r>0$ is bounded, closed and not compact for non-compact $Y$.

For example, if $Y$ is the $n-1$-dimensional hyperbolic space $\mathbb{H}^{n-1}$, then the Minkowski cone $\mathrm{Cone}(Y)$ can be identified with $I^+(0)\cup\{0\}$ in $n$-dimensional Minkowski spacetime $\R^n_1$, and clearly the topology does not agree with the Euclidean one of $\R^n$. This can be most easily seen (in without loss of generality $n=1+1$ dimensions) by considering the sequence $p_k:=(\frac{1}{k}, y_k)$, where $y_k\in\mathbb{H}^{1}\cong\R^1$ is given by $y_k=(k^2, \sqrt{k^4-1})$ in the realization $\mathbb{H}^{1}\subseteq \mb{R}^2_1$. Then $d_c(0,p_k)=\frac{1}{k}$, hence $p_k\to 0\in \mathrm{Cone}(Y)$, but $p_k\cong \frac{1}{k}\cdot y_k=(k,\sqrt{k^2-\frac{1}{k^2}})\not\to0$ in $\R^2$ with respect to the Euclidean topology ($\frac{1}{k}\cdot y_k$ even diverges). Here we used the identification of $p=(t,y)$ with $t\cdot \psi(y)\in I^+(0)\cup\{0\}\subseteq\R^n_1$, given in \cite[Rem.\ 2.1]{AGKS:21}.
\end{rem}

Despite Remark \ref{rem-dc-not-euc-top} we will use the cone metric $d_c$ on the Minkowski cone $\mathrm{Cone}(Y)$ over a metric space $(Y,d)$, as it is a \emph{canonical} choice (in contrast to e.g.\ the usual metric on Minkowski space).
\medskip

In a Minkowski cone, the space of directions at the vertex is, as expected, essentially the base space.

\begin{pop}[Space of timelike directions in Minkowski cone]
 Let $(Y,d)$ be any metric space and consider the Minkowski cone $X:=\mathrm{Cone}(Y)$ over $Y$. Then the space of 
timelike directions at $0\in X$, $(\mc{D}_0,\ma_0)$, is isometric to $(Y,d)$ as metric spaces.
\end{pop}
\begin{proof}
We consider the map $\Phi\colon Y\to\mc{D}_0$ that maps $y\in Y$ to the (equivalence class of the) curve 
$\gamma_y(t):=(t,y)$ $(t\in[0,1])$. The curve $\gamma_y\colon[0,1]\rightarrow X$ is a maximizer as 
$\tau((s,y),(t,y))=\sqrt{s^2+t^2-2st}=|s-t|$. For two such curves $\gamma_{y_1},\gamma_{y_2}\colon[0,1]\rightarrow X$, we 
have that $\ma_0(\gamma_{y_1},\gamma_{y_2})=d(y_1,y_2)$ as the definition of the time separation function $\tau$ in the 
Minkowski cone is made precisely such that $\tilde{\ma}_0(\gamma_{y_1}(s),\gamma_{y_2}(t))=d(y_1,y_2)$ whenever $s,t$ are 
such that $\gamma_{y_1}(s),\gamma_{y_2}(t)$ are timelike related (Law of Cosines). Thus, $\Phi$ is an isometric embedding. 

Conversely, let $\tilde{\gamma}$ be any geodesic in $X$ starting at $0$, without loss of generality parametrized by the 
$t$-coordinate.  Let $[0,\varepsilon]$ be a part of the domain where $\tilde{\gamma}$ maximizes. So let $0<s<t\leq\varepsilon$. Define $(s,p):=\tilde{\gamma}(s)$ and $(t,q):=\tilde{\gamma}(t)$, we indirectly assume $p\neq q$. Then $\tau(0,(t,q))=t$, 
$\tau(0,(s,p))=s$ and $\tau((s,p),(t,q))=\sqrt{s^2+t^2-2st\cosh(d(p,q))}<t-s$ as $d(p,q)>0$, contradicting $\tilde{\gamma}$ being a distance realizer when restricted to $[0,\varepsilon]$. Thus $\tilde{\gamma}$ can only realize $\tau$ if $p=q$ for all $0<s<t\leq\varepsilon$, hence $\tilde{\gamma}\sim\gamma_p$ and so $\Phi$ is surjective.
\end{proof}
\bigskip

At this point we are in the position to introduce the notion of a timelike tangent cone.

\begin{defi}[Timelike tangent cone]
Let $\Xll$ be a \LpLS with $\tau$ locally finite-valued and so that angles between future/past directed timelike 
geodesics always exist and let $x\in X$. The \emph{future/past timelike tangent cone} $T^\pm_x$ at $x$ is the Minkowski 
cone $\mathrm{Cone}(\Sigma_x)$ over the metric space $\Sigma^\pm_x$ of the completion of future/past directed timelike 
directions $\D^\pm_x$.
\end{defi}

The following is straightforward:
\begin{lem}\label{spaceOfDirectionsForSpt}
Let $M$ be an $n+1$-dimensional strongly causal spacetime with metric of $\Con^2$-regularity, and consider it as a 
Lorentzian length space. Let $p\in M$, then the space of future directed timelike directions at $p$, $\D_p^+$, is isometric to $n$-dimensional hyperbolic space $\mb{H}^n$ and the tangent cone $T_x^+$ is isometric to $I^+(0)\subseteq T_xM$.
\end{lem}
\begin{proof}
In a $\Con^2$-spacetime, geodesics are solutions to the geodesic equation. Thus, future directed timelike geodesics passing through $x$ are in one-to-one correspondence to future directed timelike vectors in $T_xM\cong\mb{R}^{n+1}_1$. As angles agree (Lemma \ref{hypAnglesAgreeSpt}), geodesics have the same angle as the vectors they correspond to, and we can restrict to unit future directed timelike vectors, which form $\mb{H}^n$. The usual metric on $\mb{H}^n$ is just defined as $d(p,q)=\ma_0(p,q)$, which is the metric on $\D_x^+$. For the tangent cone, we just note that $\mathrm{Cone}(\mb{H}^n)\cong I^+(0)\subseteq \R^{n+1}_1$, see \cite[Rem.\ 2.1]{AGKS:21}.
\end{proof}

A natural question is that of whether $\Sigma_x^\pm$  is a geodesic length space. Moreover, in the metric case it is not hard to see that the completion of a length space is a length space (\cite[Exc.\ 2.4.18]{BBI:01}). We give an example below, where $\D_x^+$ is not intrinsic, and then later in Proposition \ref{SigmaIsIntrinsic} give sufficient conditions for $\D^\pm_x$ to be intrinsic.

\begin{ex}
In essence, this example is the Minkowski cone over a non-intrinsic space. However, it is defined as a subset of Minkowski spacetime and not as a Minkowski cone over a metric space.
Let $X=\{(t,x,y)\in\R^3_1: x^2\geq yt\}\subseteq \mb{R}^3_1$ be the closed exterior of a tilted double cone, given by $yt=\frac{1}{4}((y+t)^2-(y-t)^2)=x^2$, see Figure \ref{fig-non-con-con}. It is a non-convex, conical (with respect to $0$) subset of Minkowski spacetime, and we consider it as a \LpLS by restricting $\tau$ from Minkowski space. We only have to consider $\tilde{\ma}$ as angles exist and are invariant under scaling. The timelike geodesics through $0$ are just straight lines. Thus, the space of directions $\D_0^+$ at $0$ is a non-convex subset of $\mb{H}^n$, and thus not intrinsic. Note $X$ was not intrinsic as well.

If we intrinsify $X$ (see \cite[Thm.\ 1.7.3]{Ber:20}), i.e., considering $X$ with the new time separation function $\hat\tau(p,q):=\sup\{L_\tau(\gamma): \gamma$ future directed causal from $p$ to $q\}\cup\{0\}$, the lengths of geodesics stay the same, but angles will change. We need the other geodesics as well: For $p,q$ not both on the cone such that the straight line connection is not contained in $X$, the connecting geodesic will be straight outside the cone, will be tangential to the cone where it touches it, and at such points it might switch over to have a part contained in the cone (which is a geodesic in the part of the cone which can be considered as a $1+1$-dimensional Lorentzian manifold). Now we have a distance realizer $\gamma$ from $p$ to $q$. It is intuitively clear that the straight lines from $0$ to $\gamma(t)$ form a distance realizer in $\D_0^+$, making it strictly intrinsic. 

\begin{figure}[h!]
\begin{center}
\begin{tikzpicture}[x=1.5cm,y=1.5cm]
\draw[rotate=-45] (0,1) ellipse(1.5cm and 0.375cm);
\draw[rotate=-45] (0,-1) ellipse(1.5cm and 0.375cm);
\draw[rotate=-45] (1,1) -- (-1,-1) (1,-1)--(-1,1);
\draw[color=red] (1.5,-1)--(0.2,0.5) arc (221:180:1) ;
\end{tikzpicture}
\end{center}
\caption{The example where $\D_x^+$ is not intrinsic, but neither is $X$ itself. The red curve maximizes the length of all causal curves connecting the endpoints, while not realizing the time separation.}\label{fig-non-con-con}
\end{figure}
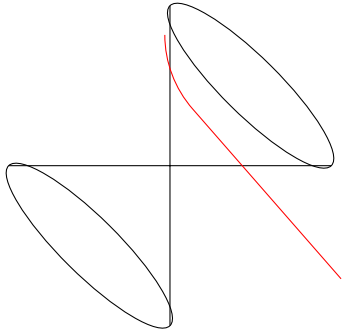
\end{ex}

We have the following sufficient criterion for the completion of the space of timelike directions to be intrinsic.
\begin{pop}[Sufficient conditions for space of timelike directions to be intrinsic]\label{SigmaIsIntrinsic} 
Let \Xll be a strongly causal, locally causally closed \LpLS with $\tau$ locally finite valued and timelike curvature bounded above by $K_+\in\R$ and below by $K_-\in\R$. Then $\Sigma_x^\pm$ is intrinsic. If additionally $X$ is locally compact, $d$-compatible, regularly localizable and future and past non-lingering, then already $\D_x^\pm$ itself is intrinsic.
\end{pop}
\begin{proof}
We only establish the future case, the past case is completely analogous.

We want to employ \cite[Theorem 2.4.16]{BBI:01}, which states that a complete metric space which has $\varepsilon$-midpoints 
is already intrinsic. Note that we only have to find $\varepsilon$-midpoints for a dense subset, for which we use 
$\D_x^+\subseteq\Sigma_x^+$. So let $\varepsilon>0$, and let $[\alpha],[\beta]$ be two points in $\D_x^+$, with representatives $\alpha\in[\alpha]$, $\beta\in[\beta]$ parametrized by $\tau$-arclength. Note that by Lemma \ref{lem-K-ang-com-ang-ex-fudi}, the angle between $\alpha$ and $\beta$ exists and is finite. By strong causality and the curvature bound we find a neighborhood of $x$ such that all $\tau$-distances can be realized locally. Moreover, by the definition of the angle, Proposition \ref{lorSphericalangleNoDepK} and Proposition \ref{angleBoundImpliesTimelikeRelation}, for small enough $t>0$ we have that $s:=t e^{-\ma_x(\alpha,\beta)-1}$ satisfies $\alpha(s)\ll\beta(t)$ and both 
$|\tilde{\ma}_x^{K_-}(\alpha(s),\beta(t))-\ma_x(\alpha,\beta)|<\varepsilon$ and 
$|\tilde{\ma}_x^{K_+}(\alpha(s),\beta(t))-\ma_x(\alpha,\beta)|<\varepsilon$. Let $\gamma_{t}\colon[0,1]\rightarrow X$ be a 
distance realizer (i.e., a timelike geodesic) from $\alpha(s)$ to $\beta(t)$. For each $r\in[0,1]$ let 
$\eta_{r,t}\colon[0,1]\rightarrow X$ be a distance realizer from $x$ to $\gamma_{t}(r)$. We now claim that, for small 
enough $t$ and some $r$, the direction $[\eta_{r,t}]$ is a $4\varepsilon$-midpoint of $[\alpha]$ and $[\beta]$. We denote the points by $a=\alpha(s)$, $b=\gamma_{t}(r)$, $c=\beta(t)$. (See Figure \ref{fig-Sigma+_has_eps_midpoints} for a drawing.) Note that $a\ll b\ll c$.
\begin{figure}[h!]
\begin{center}
\begin{tikzpicture}[line cap=round,line join=round,x=2cm,y=2cm]
\draw[ smooth,samples=100,domain=-1.0:0.0] plot({\x*\x/4.0-\x/2.0-1.0/2.0-1.0/4.0},\x) node[anchor=south]{$a$};%\alpha
\draw[ smooth,samples=100,domain=-1.0:1.0] plot({-\x*\x/4.0+1.0/4.0},\x) node[anchor=south]{$c$};%\beta
\draw[ smooth,samples=100,domain=0.0:1.0,color=blue] plot({\x*\x/4.0+\x/2.0-1.0/2.0-1.0/4.0},\x);%\gamma_t
\draw [shift={(0.,-1.)},color=green,fill=green,fill opacity=0.1] (0,0) -- (135.:0.3) arc (135.:67:0.3) -- cycle;

\draw[ smooth,samples=100,domain=-1.0:0.3,color=red] plot(\x*\x/8-0.35673076923076923076923076923077*\x-0.48173076923076923076923076923077,\x) node[anchor=south]{$b$};%\eta_{r,t}
\draw [shift={(0.,-1.)},color=orange,fill=orange,fill opacity=0.1] (0,0) -- (135.:0.3) arc (180.:155:0.3) arc (95:60:0.5) -- cycle;

\draw (0,-1) node[anchor=north]{$x$}
({(-0.5)*(-0.5)/4.0-(-0.5)/2.0-1.0/2.0-1.0/4.0},-0.5) node[anchor=north east]{$\alpha$}
({-0.3*0.3/4.0+1.0/4.0},0.3) node[anchor=west]{$\beta$};
\draw[color=blue] ({0.5*0.5/4.0+0.5/2.0-1.0/2.0-1.0/4.0},0.5) node[anchor=south east]{$\gamma_t$};
\draw[color=red] (-0.43194444444444444444444444444445,-0.13333333333333333333333333333333) node[anchor=south west]{$\eta_{r,t}$};

\end{tikzpicture}
\end{center}
\caption{For a suitably chosen point on $\gamma_{t}$, the curve $\eta_{r,t}$ is an $4\varepsilon$-midpoint.} \label{fig-Sigma+_has_eps_midpoints}
\end{figure}
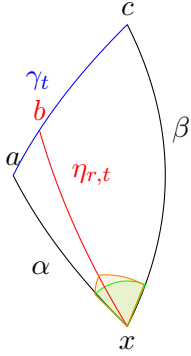

We realize this in a comparison configuration as follows: Form a comparison triangle $\Delta \bar{x}_-\bar{a}_-\bar{c}_-$ in $\lm{K_-}$, and analogously form a comparison triangle $\Delta \bar{x}_+\bar{a}_+\bar{c}_+$ in $\lm{K_+}$. As $b\in [a,c]$ and denoting the corresponding points by $\bar b_\pm\in[\bar a_\pm,\bar c_\pm]$, we conclude by using the curvature bounds, that $\tau(x,b)\leq\ubar\tau(\bar{x}_-,\bar{b}_-)$ and $\tau(x,b)\geq\bar\tau(\bar{x}_+,\bar{b}_+)$, where $\ubar\tau$ and $\bar\tau$ denote the time separation functions on $\lm{K_-}$ and $\lm{K_+}$, respectively. By the law of cosines \ref{lorLawOfCosines}, we can translate this to statements on the angles and conclude that
\begin{align}
\label{eq-sig-int-abm}\tilde{\ma}_x^{K_-}(a,b)&\leq{\ma}_{\bar{x}_-}^{\lm{K_-}}(\bar{a}_-,\bar{b}_-)\,,\\
\tilde{\ma}_x^{K_-}(b,c)&\geq{\ma}_{\bar{x}_-}^{\lm{K_-}}(\bar{b}_-,\bar{c}_-)\,,\\
\tilde{\ma}_x^{K_+}(a,b)&\geq{\ma}_{\bar{x}_+}^{\lm{K_+}}(\bar{a}_+,\bar{b}_+)\,,\\
\label{eq-sig-int-bcp}\tilde{\ma}_x^{K_+}(b,c)&\leq{\ma}_{\bar{x}_+}^{\lm{K_+}}(\bar{b}_+,\bar{c}_+)\,.
\end{align}

Moreover, since the comparison situations are planar we have that
\begin{align}
{\ma}_{\bar{x}_-}^{\lm{K_-}}(\bar{a}_-,\bar{b}_-)+{\ma}_{\bar{x}_-}^{\lm{K_-}}(\bar{b}_-,\bar{c}_-)&={\ma}_{\bar{x}_-}^{\lm{K_-}}(\bar{a}_-,\bar{c}_-)\,,\\
{\ma}_{\bar{x}_+}^{\lm{K_+}}(\bar{a}_+,\bar{b}_+)+{\ma}_{\bar{x}_+}^{\lm{K_+}}(\bar{b}_+,\bar{c}_+)&={\ma}_{\bar{x}_+}^{\lm{K_+}}(\bar{a}_+,\bar{c}_+)\,.
\end{align}

We have by the triangle inequality for angles (Theorem \ref{thm-ang-tri-equ}) that 
\begin{align}
\ma_x(\alpha,\eta_{r,t})+\ma_x(\eta_{r,t},\beta)\geq\ma_x(\alpha,\beta)\,,
\end{align}
and by \eqref{eq-sig-int-abm}, \eqref{eq-sig-int-bcp} that
\begin{align}\label{eqn-SigmaIntr-angleBounds}
\tilde\ma_x^{K_-}(a,b)+\tilde\ma_x^{K_+}(b,c)\leq\ma_{\bar{x}_-}^{\lm{K_-}}(\bar{a}_-,\bar{b}_-)+\ma_{\bar{x}_+}^{\lm{K_+}}(\bar{b}_+, \bar { c } _+)\,.
\end{align}

At this point we claim that $\ma_{\bar{x}_-}^{\lm{K_-}}(\bar{b}_-,\bar{c}_-)\geq \ma_{\bar{x}_+}^{\lm{K_+}}(\bar{b}_+,\bar{c}_+)-\varepsilon$. To this end we compare the following four triangles and their angle at $x$:
\begin{enumerate}
\item the triangle $\Delta \bar{x}_-\bar{b}_-\bar{c}_-$ in $\lm{K_-}$ with the angle $\omega_1:={\ma}_{\bar{x}_-}^{\lm{K_-}}(\bar{b}_-,\bar{c}_-)$,
\item a comparison triangle $\Delta \tilde{x}_-\tilde{b}_-\tilde{c}_-$ in Minkowski space $\lm{0}$ of $\Delta 
\bar{x}_-\bar{b}_-\bar{c}_-$ and the angle 
${\ma}_{\tilde{x}_-}^{\lm{0}}(\tilde{b}_-,\tilde{c}_-)={\ma}_{\bar{x}_-}^{\lm{0}}(\bar{b}_-,\bar{c}_-)=:\omega_2$,
\item a comparison triangle $\Delta \tilde{x}_+\tilde{b}_+\tilde{c}_+$ in Minkowski space $\lm{0}$ of $\Delta \bar{x}_+\bar{b}_+\bar{c}_+$ and the angle ${\ma}_{\tilde{x}_+}^{\lm{0}}(\tilde{b}_+,\tilde{c}_+)={\ma}_{\bar{x}_+}^{\lm{0}}(\bar{b}_+,\bar{c}_+)=:\omega_3$, and
\item the triangle $\Delta \bar{x}_+\bar{b}_+\bar{c}_+$ in $\lm{K_+}$ with the angle $\omega_4:={\ma}_{\bar{x}_+}^{\lm{K_+}}(\bar{b}_+,\bar{c}_+)$.
\end{enumerate}
Recall that all the points depend on $r,t$. Our aim is to have uniform comparison estimates on the differences of the angles $\omega_1,\ldots,\omega_4$, while varying $r$ and $t$. We claim that for all $\varepsilon>0$ there is $T>0$ such that for all $0<t<T$ we have $|\omega_1-\omega_2|<~\varepsilon$, $|\omega_2-\omega_3|<\varepsilon$ and $|\omega_3-\omega_4|<\varepsilon$ hold uniformly in $r$.

For the first and last inequality, note the triangles are comparison triangles of each other, i.e., they have the same sidelengths but different $K$, we use the second-to-last equation in the proof of Proposition \ref{lorSphericalangleNoDepK} (or the corresponding equation for $K<0$). We only show $|\omega_1-\omega_2|<\varepsilon$, the argument for $|\omega_3-\omega_4|<\varepsilon$ only differs by replacing some $-$ indices by $+$. We temporarily set $u=\tau(x,\bar{b}_-)$, $v=\tau(\bar{b}_-,\bar{c}_-)$. Then the second-to-last equation in the proof of Proposition \ref{lorSphericalangleNoDepK} (or the corresponding equation for $K<0$) (and the law of cosines for $K=0$) read
\begin{align*}
\cosh(\omega_1)&=-\frac{(u^2+t^2-v^2)(1+o(1))+o(ut)}{2ut}(1+o(1))\,,\\
\cosh(\omega_2)&=-\frac{u^2+t^2-v^2}{2ut}\,,
\end{align*}
as $u,t,v\to 0$. Here, we have $0<u,v<t$ and let $t=\tau(x,\beta(t))=\tau(x,c)\to 0$. As $\frac{u^2+t^2-v^2}{2ut}$ stays bounded (the angle exists), we have $\cosh(\omega_1)-\cosh(\omega_2)=o(1)$ as $t\to0$, hence $\omega_1-\omega_2\to0$, and similarly 
$\omega_3-\omega_4\to0$.

For the second inequality, we compare the side-lengths 
$\lambda_-=\ubar\tau(\bar{x}_-,\bar{b}_-)=\tilde\tau(\tilde{x}_-,\tilde{b}_-)$ and 
$\lambda_+=\bar\tau(\bar{x}_+,\bar{b}_+)=\tilde\tau(\tilde{x}_+,\tilde{b}_+)$, where $\tilde\tau$ denotes the time 
separation function on Minkowski spacetime $\lm{0}$. We get by the first case of Lemma \ref{lorOneSidedCalcs} (for $K_+$ and 
$K_-$), that $\frac{|\lambda_+-\lambda_-|}{t}\to0$ as $t\to0$. By the choice of $s$ we obtain
\begin{equation}\label{eq:SigmaIsIntrinsic:quantitativeTimelikeBound}
 0<t e^{-\ma_x(\alpha,\beta)-1}=s<\lambda_-<t\,,
\end{equation}
and hence we get $\frac{\lambda_+}{\lambda_-}=:q\to1$. Now plugging this into the law of cosines \ref{lorLawOfCosines} ($K=0$, $\sigma=-1$) we obtain the following relation between $\omega_2$ and $\omega_3$:
\begin{align*}
\lambda_-^2 + t^2&=\tau(b,c)^2+2\lambda_-t\cosh(\omega_2)\,,\\
\lambda_+^2 + t^2&=\tau(b,c)^2+2\lambda_+t\cosh(\omega_3)\,.
\end{align*}
This yields 
\begin{align}
\cosh(\omega_3)-\cosh(\omega_2)&=\frac{\lambda_-^2(q^2-q) + t^2(1-q)-\tau(b,c)^2(1-q)}{2\lambda_-q t}\,,
\end{align}
which goes to $0$ by \eqref{eq:SigmaIsIntrinsic:quantitativeTimelikeBound}, $q\to 1$ and $\tau(b,c)<t$.

In total, we get $|\omega_1-\omega_4|<3\varepsilon$ for small $t$, so we have proven the claim for $4\varepsilon$ and can replace the ${\ma}_{\bar{x}_+}^{\lm{K_+}}(\bar{b}_+,\bar{c}_+)$ by a ${\ma}_{\bar{x}_-}^{\lm{K_-}}(\bar{b}_-,\bar{c}_-)$ in equation \eqref{eqn-SigmaIntr-angleBounds}, introducing only an error of $4 \varepsilon$. Then all the points are in the same comparison space and we have ${\ma}_{\bar{x}_-}^{\lm{K_-}}(\bar{a}_-,\bar{b}_-)+{\ma}_{\bar{x}_-}^{\lm{K_-}}(\bar{b}_-,\bar{c}_-)={\ma}_{\bar{x}_-}^{\lm{K_-}}(\bar{a}_-,\bar{c}_-)$.

To sum up, we have
\[
\ma_x(\alpha,\eta_{r,t})+\ma_x(\eta_{r,t},\beta)\leq\ma_x(\alpha,\beta)+5\varepsilon\,,
\]
where two $\varepsilon$ come from the approximation of angles, and three $\varepsilon$ comes from the difference of ${\ma}_{\bar{x}_+}^{\lm{K_+}}(\bar{b}_+,\bar{c}_+)$ and ${\ma}_{\bar{x}_-}^{\lm{K_-}}(\bar{b}_-,\bar{c}_-)$.

Note that this argument was independent of $r$. As $\tau$ is locally continuous, and we are considering it only in a comparison neighborhood, we get that $\tilde{\ma}_x^{K_+}(a,b)=\tilde{\ma}_x^{K_+}(\alpha(s),\gamma_{t}(r))$ is a continuous function in $r$, with values $0$ for $r=0$ and $\tilde{\ma}_x^{K_+}(a,c)>\frac{\ma_x(\alpha,\beta)}{2}$ for $r=1$ (if $\varepsilon$ is small enough). In particular, there is a value of $r$ such that $\tilde{\ma}_x^{K_+}(a,b)=\frac{\ma_x(\alpha,\beta)}{2}$. This then implies that
\begin{align*}
|\ma_x(\alpha,\eta_{r,t})-\frac{\ma_x(\alpha,\beta)}{2}|&<\varepsilon\,,\\
|\ma_x(\eta_{r,t},\beta)-\frac{\ma_x(\alpha,\beta)}{2}|&<6\varepsilon\,.
\end{align*}
As $\Sigma_x^+$ is complete, \cite[Theorem 2.4.16]{BBI:01} shows it is intrinsic.

Finally, if additionally, $X$ is locally compact, $d$-compatible, regularly localizable and future and past non-lingering, then the spaces of timelike directions $\D_x^\pm$ are complete for every $x\in X$ by Proposition \ref{pop-D-plus-complete} and hence intrinsic by the above.
\end{proof}

As the previous Proposition \ref{SigmaIsIntrinsic} used two-sided curvature bounds we comment on the relationship of different curvature bounds, which was not spelled out so explicitly before. In particular, a non-trivial space cannot have arbitrary curvature bounds from below and above.
\begin{lem}[Relation of different curvature bounds]\label{lem-rel-dif-K}
Let \Xll be a \LpLS with timelike curvature bounded above by $K_+$ and below by $K_-$. If there exists a non-degenerate timelike triangle satisfying timelike size bounds for $K_-$ and $K_+$ inside a neighborhood, which is simultaneously a comparison neighborhood for $K_-$ and for $K_+$, then $K_+\leq K_-$. 

For \Xll a \LpLS with timelike curvature bounded below by $K_-$, then it also has timelike curvature bounded below by $K$ for all $K\geq K_-$. And if \Xll has timelike curvature bounded above by $K_+$, it also has timelike curvature bounded above by $K$ for all $K\leq K_+$.

In particular, for $K_+\leq K_-$, $\lm{K_-}$ has timelike curvature bounded above by $K_+$, and $\lm{K_+}$ has timelike curvature bounded below by $K_-$.
\end{lem}
\begin{pr}
Let $\Delta=(p_1,p_2,p_3)$ with $p_1\ll p_2\ll p_3$ be a non-degenerate timelike triangle inside a curvature comparison neighborhood for both $K_-$ and $K_+$ satisfying the timelike size bounds for $K_-$ and $K_+$. Then we find comparison triangles $\bar{\Delta}^\pm= (\bar{p}_1^\pm,\bar{p}_2^\pm,\bar{p}_3^\pm)$ in $\lm{K_\pm}$. For points $q_1,q_2$ on the sides of $\Delta$, we get comparison points $\bar{q}_1^\pm,\bar{q}_2^\pm$ on the sides of $\bar{\Delta}^\pm$. Note we can also view $\bar{\Delta}^\pm,\bar{q}_1^\pm,\bar{q}_2^\pm$ as a comparison situation for $\bar{\Delta}^\mp,\bar{q}_1^\mp,\bar{q}_2^\mp$, respectively. 

By timelike curvature comparison, we get 
\begin{equation}\label{eq-Kpm}
\bar\tau_{\lm{K_+}}(\bar{q}_1^+,\bar{q}_2^+)\leq \tau(q_1,q_2)\leq\bar\tau_{\lm{K_-}}(\bar{q}_1^-,\bar{q}_2^-)\,.
\end{equation}
We assume indirectly that $K_+>K_-$ and choose $\bar q^\pm_1:=\bar p^\pm_1$ and $q^\pm_2\in [\bar p_2^\pm, \bar p_3^\pm]$. Then \cite[Lem.\ 6.1]{AGKS:21} yields that $\bar\tau_{\lm{K_+}}(\bar{p}_1^+,\bar{q}_2^+)> \bar\tau_{\lm{K_-}}(\bar{p}_1^-,\bar{q}_2^-)$ --- a contradiction to Equation \eqref{eq-Kpm}.
\end{pr}

Note this relation between the different curvature bounds is opposite to the direction one would expect, e.g.\ timelike curvature $\geq 1$ implies $\geq 2$, although in the metric case it is the other way round, and $\geq 1$ and $\leq 2$ contradict each other (making all small enough timelike triangles degenerate). However, it was chosen this way to be compatible with \cite{AB:08} and hence the smooth spacetime case.

\subsection{Exponential and logarithmic map}

The notion of a timelike tangent cone allows us to introduce an exponential and logarithmic map into our setting, which are indispensable geometric tools. For exponential and logarithmic maps in the metric case see \cite[p.\ 321ff.]{BBI:01}. We first define the logarithmic map, which (locally) assigns points to elements in the timelike tangent cone encoding the time separation and the equivalence class of the connecting geodesic, i.e., the timelike direction.

\begin{defi}[Logarithmic map]\label{def-log}
Let $\Xll$ be a locally uniquely timelike geodesic \LpLSn, i.e., every point has a neighborhood $U$ such that in $U$ two timelike related points can be joined by a unique timelike geodesic contained in $U$. Let $x\in X$ and let $U$ be such a neighborhood of $x$. Then, for all $y\in U$ with $x\ll y$ ($y\ll x$) there exists an unique future (past) directed timelike geodesic $\gamma$ from $x$ to $y$ contained in $U$. Then set $\log^+_x(y):=(\tau(x,y),[\gamma])\in T^+_x$ and $\log^-_x(y):=(\tau(y,x),[\gamma])\in T^-_x$, respectively. So, the logarithmic map $\log^\pm_x\colon U\cap I^\pm(x)\rightarrow T^\pm_x$ maps points to elements of the tangent cone. 
\end{defi}

Note that the image $\log_x^\pm(U\cap I^\pm(x))\subseteq T^\pm_x$ is star-shaped with respect to $(0,\star)\in T^\pm_x$ in the sense that for an element $(t,[\gamma])\in\log_x^\pm(U\cap I^\pm(x))$ and $\lambda\in[0,1]$, also the scaled element $(\lambda t,[\gamma])\in T^\pm_x$ is in $\log_x^\pm(U\cap I^\pm(x))$. 
\bigskip

At this point we introduce the \emph{exponential map}, that projects elements of the timelike tangent cone down to points in the space. Moreover, it is the inverse of the logarithmic map.

\begin{defi}[Exponential map]
  Let $\Xll$ be a locally uniquely timelike geodesic \LpLS with timelike curvature bounded below,
  let $x\in X$ and $U$ as in Definition \ref{def-log}. Then the \emph{future/past exponential map} $\exp_x^\pm\colon T^\pm_x\supseteq\log^\pm_x(U\cap I^\pm(x))\rightarrow U$ is defined by
\begin{equation}
 (r,[\gamma])\mapsto \tilde\gamma(r)\,,
\end{equation}
where $\tilde\gamma\in[\gamma]$ is parametrized with respect to $\tau$-arclength (cf.\ \cite[Cor.\ 3.35]{KS:18}).
\end{defi}

\begin{rem}
 The exponential map is well-defined: Let $\alpha,\beta$ be two timelike geodesics starting at $x$, defined on 
$[0,r]$, parametrized with respect to $\tau$-arclength, and ending at different points $\alpha(r)=p\neq q=\beta(r)$. Then, 
$\alpha,\beta$ agree on a closed set and so there is a point $\alpha(s)\in I^-(q)$ which is not on $\beta$, and 
by uniquely timelike geodesicness the concatenation of $\alpha\rvert_{[0,s]}$ and the unique timelike geodesic from $\alpha(s)$ to $q$ is strictly shorter than $\beta$. Thus we get a non-degenerate triangle $\Delta x\alpha(s)\beta(r)$. Then we conclude $\ma_x(\alpha,\beta)\geq\tilde\ma_x(\alpha(s),\beta(r))>0$ by using the timelike curvature bound from below, Corollary \ref{cor-K-mon-ang-bou} and Theorem \ref{thm-K-ang-com-fu}. Thus $[\alpha]$ and $[\beta]$ are different points in $\D_x^+$. Moreover, as we restrict to the image of $U$ under $\log^\pm_x$ we have $\tilde\gamma(r)\in U$.
\end{rem}

As we assume that the space is locally uniquely timelike geodesic and has timelike curvature bounded from below the exponential map is well-defined. In principle, one could define an exponential map also for general \LpLSn s as set-valued maps, cf.\ the discussion in the metric case \cite[Rem.\ 9.1.43]{BBI:01}.

\begin{rem}
Let $X\subseteq Y$ be two Lorentzian pre-length spaces included in each other (i.e.\ the relations and time separation function on $X$ is just the restriction of those of $Y$) with logarithm and exponential map defined, and let $x\in X$. We can consider the space of directions with respect to $X$ and with respect to $Y$. Then $(\D_x^+)^X\subseteq (\D_x^+)^Y$ and in particular, $(T_x^+)^X$ is a subset of $(T_x^+)^Y$. We then have that the logarithmic map and exponential map with respect to $Y$ are extensions of the logarithmic map and exponential map with respect to $X$.
\end{rem}
% \begin{proof}
% For $\alpha,\beta:[0,\varepsilon)\to X$ with $\alpha(0)=\beta(0)=:x$ we have $\ma_x(\alpha,\beta)$ is independent of being constructed in $X$ or $Y$, as $\tau$ agrees. Thus, $\Phi:(D_x^+)^X\to(D_x^+)^Y$, $
% \Phi(\alpha)=\alpha$ descends to the quotient space, i.e.\ $\bar{\Phi}:(\D_x^+)^X\to(\D_x^+)^Y$ and is an inclusion of metric spaces. This lifts to an inclusion of \LpLS $(T_x^+)^X\to (T_x^+)Y$. For $\exp$ and $\log$, just note that for $(t,[\gamma])\in(T_x^+)^X$, we get the same geodesic $\gamma$ as when considering $(t,[\gamma])\in(T_x^+)^X$ as an element in $(T_x^+)^Y$ and conversely.
% \end{proof}

\begin{ex}
The image of the logarithm, hence the domain of the exponential map, need not be open, even in well-behaved spaces (as happens in metric geometry). Let $X=\{(t,x)\in\mb{R}^2_1:t^2+(x-1)^2<1\}\cup\{0\}\subseteq\mb{R}^2_1$ (see Figure 
\ref{fig-D+_not_complete}). As a subset of a spacetime, $X$ is a \LpLSn.
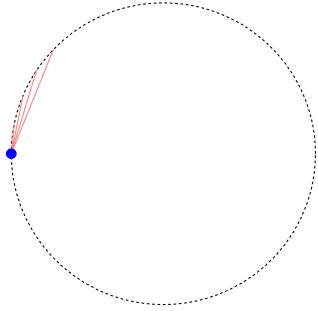
\begin{figure}[h!]
\begin{center}
\begin{tikzpicture}[x=2cm,y=2cm]
%\draw[dash pattern= on 1pt off 1pt] (0,0) arc(180:0:1);
\draw[dash pattern= on 1pt off 1pt] (1,0) circle(1);
\foreach \a in {0.1,0.2,0.3,0.4} 
  \draw[color=red!50] (0,0) -- ({2*\a*\a/(\a*\a+1)},{2*\a/(\a*\a+1});
\fill[color=blue] (0,0) circle(2pt);
\end{tikzpicture}
\end{center}
\caption{The counterexample where the space of future directed timelike directions $\D^+$ is not complete.}\label{fig-D+_not_complete}
\end{figure}
It is a quite well-behaved space: It is geodesically convex and thus strictly intrinsic, $\tau$ is continuous, there exist no $\ll$-isolated points, it is a $d$-compatible set and distance realizing curves do not change their causal character, so the whole space is a regularly localizable neighborhood, and is even strongly regularly localizable. It is strongly causal and even causally simple (i.e., the relation $\leq$ is a closed subset of $X\times X$), but not globally hyperbolic. It has timelike curvature bounded above and below by $K=0$. 

However, the space of directions $\D_0^+$ at $p=0$ is isometric to $(0,\infty)$ and thus not complete, and the image $\log_0^+(I^+(0))$, which is contained in $T_0^+=\{0\}\cup I^+(0)$ of Minkowski spacetime, is not open at $0$: The sequence $(x_n)_{n\geq 2}$, $x_n\in T_0^+$ corresponding to the points $(\sqrt{1-(1-\frac{1}{n})^2},\frac{1}{n})\in I^+(0)$ in $\mb{R}^2_1)$ (on the considered circle, transferred to $T_0^+$ by the logarithm on Minkowski space) converges to $0$, but is not contained in $\log_0^+(I^+(0))$. 
\end{ex}

% \begin{defi}
%   Let $\Xll$ be a localizable, timelike non-branching \LpLS with $\tau$ continuous and finite-valued and so that angles between (future/past directed) timelike geodesics always exist and let $x\in X$. Let $r>0$, set $\Sigma^\pm_x(r):=\{[\gamma]: \exists \gamma'\in[\gamma]$ such that $L_\tau(\gamma')\geq r\}$ and $\tilde 
% T^\pm_x:=\{(1,[\gamma]):[\gamma]\in\Sigma^\pm_x(1)\}\subseteq T^\pm_x$. Then the \emph{future/past exponential map} $\exp_x^\pm\colon\Sigma^\pm_x(1)\rightarrow X$ is defined by
% \begin{equation}
%  [\gamma]\mapsto \gamma'(1)\,,
% \end{equation}
% where $\gamma'\in[\gamma]$ is parametrized with respect to $\tau$-arclength (cf.\ \cite[Cor.\ 3.35]{KS:18}). The exponential map can also be considered as map $ \tilde T^\pm_x\rightarrow X$ as $ \{1\}\times\Sigma_x^\pm(1)=\tilde T^\pm_x$.
% \end{defi}
% 
% This is well-defined as the space is timelike non-branching.
%
%
%\begin{lem}
% $\tilde T^\pm_x$ open in $T^\pm_x$; curvature bounds of $T^\pm_x$ imply curvature bounds of $\tilde T^\pm_x$
%\end{lem}
\medskip

On the one hand, it is necessary that all curves are timelike when considering their angles. On the other hand, the limit curve theorem gives only causal limit curves in general. However, when the sequence of angles is bounded the limit curve has to be timelike, given that one has curvature bounded below. 
\begin{lem}[Bounded angles imply limit curve timelike]\label{lem-bou-ang-lim-cur-tl}
 Let $\Xll$ be a strongly causal and regularly localizable \LpLS with timelike curvature bounded from below. Let $\gamma_n\colon[0,\varepsilon]\rightarrow X$ be future directed timelike geodesics, starting at $\gamma_n(0)=x$ and converging uniformly to a causal curve $\gamma\colon[0,\varepsilon]\rightarrow X$. If $(\ma_x(\gamma_n,\gamma_0))_n$ is bounded the limit curve $\gamma$ is timelike.
\end{lem}
\begin{pr}
First, by \cite[Prop.\ 3.17]{KS:18} the limit curve $\gamma$ is maximal and by \cite[Thm.\ 3.18]{KS:18} it is either timelike or null. Indirectly assume that $\gamma$ is null and without loss of generality we can assume that all curve (segments) are in one regularly localizable neighborhood, which is a comparison neighborhood as well. Let $C>0$ be such that $\ma_x(\gamma_n,\gamma_0)\leq C$ for all $n\in\N$ and let $K\in\R$ be the timelike curvature bound from below. Let $t>0$ and $s>0$ such that $\gamma(s)\ll \gamma_0(t)$ and $\gamma_0$ is a distance realizer on $[0,t]$. Then for $n$ large enough we have that $\gamma_n(s)\ll \gamma_0(t)$ and hence
\begin{align}
 C\geq \ma_x(\gamma_n,\gamma_0) \geq \tilde\ma_x^K(\gamma_n(s),\gamma_0(t))=:\omega_n\,,
\end{align}
by using Corollary \ref{cor-K-mon-ang-bou} and Theorem \ref{thm-K-ang-com-fu}.
At this point we consider the triangle $\Delta_n:=\Delta x \gamma_n(s) \gamma_0(t)$, which has side-lengths $a_n:=\tau(x,\gamma_n(s))$, $b:=\tau(x,\gamma_0(t))$ and $c_n:=\tau(\gamma_n(s),\gamma_0(t))$. Note that $a_n\to 0$ and $c_n\to \tau(\gamma(s),\gamma_0(t))=:c\leq b$, by the reverse triangle inequality. For simplicity we assume that $K=0$ --- the other cases being analogous. By the law of cosines Lemma \ref{lorLawOfCosines} for $K=0$ and $\sigma=-1$ we conclude that
\begin{equation}
 \cosh(\omega_n)=\frac{a_n^2 + b^2 - c_n^2}{2 a_n b}\,,
\end{equation}
If $c<b$, the enumerator stays positive in the limit, and the denominator converges to $0$, yielding $\omega_n\to\infty$, contradicting $\omega_n<C$ for all $n\in\N$.

If $c=b$, this is a contradiction to $X$ being regular (cf.\ \cite[Thm.\ 3.18]{KS:18}): We get that the concatenation of $\gamma$ from $x$ to $\gamma(s)$ with the null segment $\gamma(s)$ to $\gamma_0(t)$, of overall length $0+c$ has the same length as the distance realizer $\gamma_0\rvert_{[0,t]}$, which goes from $x$ to $\gamma_0(t)$ and which has the length $b=c$. Note that a maximal causal curve from $\gamma(s)$ to $\gamma_0(t)$ exists by localizability and strong causality.
\end{pr}

With the above result in hand we are able to establish that the exponential map is continuous.
 \begin{lem}[Exponential map continuous]
 Let $\Xll$ be a locally uniquely timelike geodesic, globally hyperbolic, locally causally closed, regularly localizable and future or past non-lingering, respectively, \LpLS with timelike curvature bounded below. Then $\exp_x^\pm$ is continuous away from $0\in T_x^\pm$ for all $x\in X$.
\end{lem}
\begin{pr}
We only establish the future directed case, i.e., for $\exp_x^+$. The past directed case is completely analogous. Let $x\in X$ and $U$ as in definition \ref{def-log}. By global hyperbolicity and local timelike geodesic connectedness, we can shrink $U$ to be with compact closure. Let $(t_n,[\gamma_n])\in\log^+_x(I^+(x)\cap U)$ converge to $(t,[\gamma])$, and let $\gamma_n$ and $\gamma$ be defined up to parameter $t_n$ and $t$.  Then by the limit curve theorem \cite[Thm.\ 3.7]{KS:18}, we can assume (without loss of generality) that $(\gamma_n)_n$ converges to a future directed causal curve $\tilde{\gamma}$. Note that $\tilde\gamma$ is non-constant by the non-lingering property (Definition \ref{def:nonLing}). We have to show that $[\tilde{\gamma}]=[\gamma]$ and to this end we need to establish that $\tilde\gamma$ is timelike. As $[\gamma_n]\to [\gamma]$ the sequence $(\ma_x(\gamma_n,\gamma_0))_n$ is bounded and because $\gamma_n\to\tilde\gamma$ uniformly we obtain that $\tilde\gamma$ is timelike by Lemma \ref{lem-bou-ang-lim-cur-tl}.  Now, by semi-continuity of angles, Proposition \ref{pop-ang-semicont-fudi}, we get $\ma_x(\gamma,\tilde{\gamma})\leq\lim_n\ma_x(\gamma,\gamma_n)=0$, so $[\gamma]=[\tilde{\gamma}]$ and we are done.
\end{pr}

Similarly, we can show that the logarithm is continuous.
 \begin{lem}[Logarithmic map continuous]
Let $\Xll$ be a locally uniquely timelike geodesic, globally hyperbolic, locally causally closed, regularly localizable \LpLS with timelike curvature bounded above. Then $\log_x^\pm$ is continuous for all $x\in X$.
\end{lem}
\begin{pr}
We only establish the future directed case, i.e., for $\log_x^+$. The past directed case is completely analogous. Let $x\in X$ and $U$ as in Definition \ref{def-log}, and by global hyperbolicity and local timelike geodesic connectedness we can assume without loss of generality that $U$ is relatively compact and by shrinking $U$ further that $U$ is a regularly localizing neighborhood. Indirectly, assume that $\log_x^+$ is not continuous, so there is a sequence $(y_n)_n$ in $U\cap I^+(x)$ with $y_n\to y\in U\cap I^+(x)$ and $\log_x^+(y_n)\not\to \log_x^+(y)$. Thus, by definition, we have $\log(y_n)=(t_n,[\gamma_n])$ and $\log_x^+(y)=(t,[\gamma])$, where $t_n=\tau(x,y_n)$, $t=\tau(x,y)$, $\gamma_n$ are future directed timelike geodesics from $x$ to $y_n$ and $\gamma$ is a future directed timelike geodesic from $x$ to $y$. As $t_n=\tau(x,y_n)$ and $\tau$ is continuous (on $U$), $t_n\to t=\tau(x,y)$. By assumption we know that $\gamma_n(t_n)=y_n\to y$. By the limit curve theorem, we get a limit curve $\tilde{\gamma}$ of a subsequence of $(\gamma_n)_n$ from $x$ to $y$ and we indirectly assume $[\gamma]\neq[\tilde\gamma]$, hence $[\gamma_n]\not\to[\gamma]$. As $x\ll y$ and $X$ is chronological, $\tilde\gamma$ is non-constant. By \cite[Prop.\ 3.17]{KS:18} $\tilde{\gamma}$ is also maximizing and still contained in $\bar{U}$. From regular localizability and \cite[Thm.\ 3.18]{KS:18} we conclude that $\tilde\gamma$ is timelike, hence by uniqueness of geodesics in $\bar{U}$, $\gamma$ and $\tilde{\gamma}$ are reparametrizations of each other. Finally, by Proposition \ref{pop-ang-semicont-fudi} we estimate
\begin{align}
 0= \ma_x(\tilde\gamma,\gamma) \geq \limsup_n \ma_x(\gamma_n,\gamma)\,,
\end{align}
a contradiction to $[\gamma_n]\not\to[\gamma]$.
\end{pr}

Combining the last two results, we obtain 
\begin{cor}
The exponential map $\exp_x^\pm$ is a homeomorphism away from $0$ for locally uniquely timelike geodesic, globally hyperbolic, locally causally closed, regularly localizable, future or past non-lingering \LpLSn s with timelike curvature bounded above and below.
\end{cor}

\section{Angles between timelike curves of arbitrary time orientation}\label{sec-ang-arb-to}
In this final section we consider angles between curves of different time orientation. This is needed to complete the characterization of (timelike) curvature bounds via $K$-angle monotonicity and for the triangle inequality of angles.
\medskip

In fact, for a special case of the triangle inequality of angles, we need the following corollary of the straightening lemma for shoulder angles \cite[Lemma 2.4]{AB:08} (see also \cite[Lemma 3.3.2]{Kir:18}).
\begin{cor}[Straightening lemma]\label{cor-straightening-of-middle-line}
Let $K\in\mb{R}$. Let $p\ll m\ll q\leq r$ be points in $\lm{K}$ such that $\tau(p,q)+\tau(q,r)\leq \tau(p,m)+\tau(m,r)$\footnote{Note this is needed for the existence of $\tilde{\Delta}$.} and such that $p$ and $r$ lie on opposite sides of the geodesic segment $[mq]$. Also assume that $\tau(p,q)+\tau(q,r)<D_K$ to ensure that all relevant triangles satisfy timelike size bounds for $K$. We consider the causal triangles $\Delta_1=\Delta pmq$, $\Delta_2=\Delta mqr$ and an additional causal triangle $\Delta'=\Delta p'q'r'$ in $\lm{K}$ with side-lengths
\begin{align*}
\tau(p',q')&=\tau(p,q)\,,\\
\tau(q',r')&=\tau(q,r)\,,\\
\tau(p',r')&=\tau(p,m)+\tau(m,r)\,,
\end{align*}
and a point $m'$ on the side $[p'r']$ with $\tau(p',m')=\tau(p,m)$. For an illustration of the set-up see Figure \ref{fig-str-lem}.

Then we have that  $\ma_m^{\lm{K}}(p,q)> \ma_m^{\lm{K}}(q,r)$ holds if and only if $\tau(m,q)<\tau(m',q')$ holds. Moreover, the same equivalence holds for the reversed inequalities and equality on one side holds if and only if equality holds on the other.

\end{cor}
\begin{proof}
First we assume that $q\ll r$. We want to apply the straightening Lemma for shoulder angles (or Alexandrov's Lemma), cf.\ \cite[Lemma 2.4]{AB:08}. For this, we have to make sense of the required inequality in \cite[Lemma 2.4]{AB:08}, i.e.,
\begin{equation}\label{eq:Kir}
 (1-\lambda)\angle pmq + \lambda\angle rmq \geq 0\,.
\end{equation}
To this end let $\gamma_{p',r'}:[0,1]\to\lm{K}$ be the timelike geodesic from $p'$ to $r'$, then there is a parameter $\lambda$ such that $m'=\gamma_{p',r'}(\lambda)$, which we calculate as $\lambda=\frac{\tau(p,m)}{\tau(p,m)+\tau(m,r)}$. Then the required inequality \eqref{eq:Kir} translates to 
\begin{align}
\qquad &\tau(m,r) \tau(p,m) \tau(m,q) \cosh(\ma_m^{\lm{K}}(p,q))\\
&- \tau(p,m)\tau(m,r)\tau(m,q) \cosh(\ma_m^{\lm{K}}(q,r))\geq 0\,,
\end{align}
where we used the formula of the non-normalized angle including the two side-lengths, the cosh of the angle and the sign of the angle. This reduces to 
\[
\ma_m^{\lm{K}}(p,q)\geq \ma_m^{\lm{K}}(q,r)\,,
\]
and similarly for the inequality reversed. Under this condition, the straightening Lemma yields that 
\[
\langle\dot\gamma_{p',q'}(0),\dot\gamma_{p',m'}(0)\rangle =:   \angle q'p'm'\geq \angle qpm = \langle\dot\gamma_{p,q}(0),\dot\gamma_{p,m}(0)\rangle\,,
\]
which in our terminology means that
\[
{\ma}^{\lm{K}}_{p'}(m',q')\leq{\ma}^{\lm{K}}_p(m,q)\,.
\]

Now we note that $\Delta p'm'q'$ and $\Delta pmq$ have two equal sides (all but the $[mq]$ side) and use the monotonicity statement in the law of cosines \ref{rem-loc-mon} ($\sigma=-1$, varying a short side) to get $\tau(m',q')\geq\tau(m,q)$. 

Similarly, we get $\tau(m',q')\leq\tau(m,q)$ when $\ma_m^{\lm{K}}(p,q)\leq\ma_m^{\lm{K}}(q,r)$. Finally, observe that the equality $\ma_m^{\lm{K}}(p,q)=\ma_m^{\lm{K}}(q,r)$ implies that $p,m,r$ lie on a straight line, making $p'=p$, $m'=m$, $q'=q$ and $r'=r$ satisfy all the requirements on $\Delta'$. 

Finally, we establish the general case where $q\leq r$. Let $\eps>0$ be small enough. We consider $\tilde q\in[mq]$ with $\tau(\tilde q,q)=\eps$. Then $\tilde q\ll r$, so we can apply the above: Let the causal triangle $\tilde{\Delta}'=\Delta \tilde p'\tilde q'\tilde r'$ in $\lm{K}$ with side-lengths $\tau(\tilde p',\tilde q')=\tau(p,\tilde q)$, $\tau(\tilde q',\tilde r')=\tau(\tilde q,r)$, $\tau(\tilde p',\tilde r')=\tau(p,m)+\tau(m,r)$ and a point $\tilde m'$ on the side $[\tilde p' \tilde r']$ with $\tau(\tilde p',\tilde m')=\tau(p,m)$. Thus we get that 
\begin{align}
 \ma_m^{\lm{K}}(p,q)&=\ma_m^{\lm{K}}(p,\tilde q)> \ma_m^{\lm{K}}(\tilde q,r)=\ma_m^{\lm{K}}(q,r)\\
 &\Leftrightarrow\tau(m,\tilde q)<\tau(\tilde m',\tilde q')\,,
\end{align}
the same for flipped inequalities and equality on one side holds if and only if equality holds on the other. By continuity of $\tau$ and $\ma^{\lm{K}}$ we get as $\eps\to 0$ that
\begin{align}
 \ma_m^{\lm{K}}(p,q)> \ma_m^{\lm{K}}(q,r) \Rightarrow \tau(m,q)\leq\tau(m',q')\,,\\
 \tau(m,q)<\tau(m',q') \Rightarrow \ma_m^{\lm{K}}(p,q)\geq \ma_m^{\lm{K}}(q,r)\,,
\end{align}
or with reversed inequalities, so we still need to show strictness in either of these. For this, we consider what happens in case of equality.

If $\tau(m,q)=\tau(m',q')$, we get that $\Delta_1$ and $\Delta p'm'q'$ have the same side-lengths and $\Delta_2$ and $\Delta m'q'r'$ have the same side-lengths. Thus, angle additivity in $\lm{K}$ gives us that $\ma_m^{\lm{K}}(p,q)=\ma_m^{\lm{K}}(p',q')= \ma_m^{\lm{K}}(q',r')=\ma_m^{\lm{K}}(q,r)$, so equality in the other inequality.

If $\ma_m^{\lm{K}}(p,q)=\ma_m^{\lm{K}}(q,r)$ we use angle additivity in $\lm{K}$ to get $\ma_m^{\lm{K}}(p,r)=0$ and thus  $p,m,r$ lie on a geodesic. Consequently, $p'=p$, $q'=q$, $r'=r$, $m'=m$ is actually a valid choice for $\Delta'$ and $m'$ and thus $\tau(m,q)=\tau(m',q')$.
\end{proof}

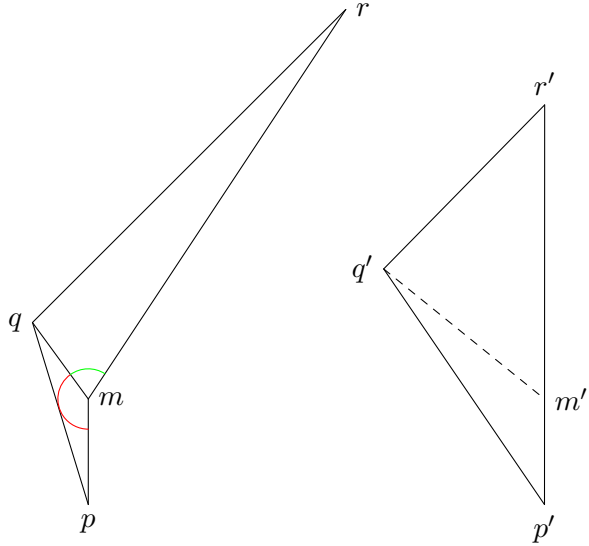
\begin{figure}[h!]
\begin{center}
 \begin{tikzpicture}[line cap=round,line join=round,>=triangle 45,x=1cm,y=1cm]
\draw (0,0)-- (-0.7340132902094361,2.414285714285714);
\draw (-0.7340132902094361,2.414285714285714)-- (0,1.4);
\draw (0,1.4)-- (0,0);
\draw (0,1.4)-- (3.3867606432599557,6.565282921073598);
\draw (3.3867606432599557,6.565282921073598)-- (-0.7340132902094361,2.414285714285714);
\draw (3.8850696994038274,3.124568798471209)-- (6,0);
\draw (3.8850696994038274,3.124568798471209)-- (6.002261946852696,5.300000482679559);
\draw (6.002261946852696,5.300000482679559)-- (6,0);
\draw [dashed] (3.8850696994038274,3.124568798471209)-- (6.000597495395052,1.4000001275002611);
\coordinate (p) at (0,0);
\coordinate (m) at (0,1.4);
\coordinate (q) at (-0.7340132902094361,2.414285714285714);
\coordinate (r) at (3.3867606432599557,6.565282921073598);
\draw (0,0) node[anchor=north]{$p$} (0,1.4) node[anchor=west]{$m$} (-0.7340132902094361,2.414285714285714) node[anchor=east]{$q$} (3.3867606432599557,6.565282921073598) node[anchor=west]{$r$} (6,0) node[anchor=north]{$p'$} (3.8850696994038274,3.124568798471209) node[anchor=east]{$q'$} (6.002261946852696,5.300000482679559) node[anchor=south]{$r'$} (6.000597495395052,1.4000001275002611) node[anchor=west]{$m'$};
\draw pic[draw,angle radius={0.4cm},color=red]{angle = q--m--p} ;
\draw pic[draw,angle radius={0.4cm},color=green]{angle = r--m--q} ;
\end{tikzpicture}
\caption{The set-up of the straightening lemma.}\label{fig-str-lem}
\end{center}
\end{figure}

\begin{defi}[Geodesic prolongation]
A \LpLS\\ $\Xll$ is said to satisfy \emph{geodesic prolongation} if all causal geodesics can be prolonged (extended) to a causal geodesic which has an open domain, i.e., a geodesic defined on a closed interval $[b,c]$ can be prolonged (as a geodesic) to $(a,e)$, where $a<b<c<e$.
\end{defi}

\begin{rem}
 Note we do not require this extension of a geodesic to be unique. Although under some assumptions, it will be, see e.g.\ \cite[Theorem 4.12]{KS:18}.
\end{rem}

This definition is complementary to \cite[Def.\ 4.5]{GKS:19}, where extendibility of geodesics was introduced. Thus we opted for calling it \emph{prolongation} instead of \emph{extendibility} to easily distinguish them.

Smooth spacetimes satisfy geodesic prolongation. However, smooth spa\-cetimes with boundary do not satisfy geodesic prolongation: There, timelike geodesics which pass non-tangentially into the boundary have (half-)closed domains and cannot be extended. For details see the following example.
\begin{ex}
Closed half-Minkowski spacetime $H=\{(t,x):x\geq 0\}\subseteq\mb{R}^{2}_1$ has very nice properties. In fact, it has timelike curvature bounded below and above by $0$, it is globally hyperbolic, causally path connected, strongly regularly localizable and it is non-lingering at each point. However, the distance realizer $\alpha(t)=(t,\frac{t_0-t}{2})$ is only defined on $(-\infty,t_0]$, has $\alpha(t_0)=(t_0,0)$ and cannot be extended as a geodesic. Thus $H$ fails to satisfy geodesic prolongation at each point on the boundary, see Figure \ref{fig-clo-hal-min}.
\end{ex}

\begin{figure}[h!]
\begin{center}
 \begin{tikzpicture}
\fill[color=black!15] (0,-3) -- (0,3) -- (-1,3) --(-1,-3) -- cycle;
\draw (0,-3) -- (0,3);

\draw[color=blue] (1,-1) -- (0,0);
\fill[color=blue] (0,0) circle(2pt);
\draw[color=blue] (1,-2) -- (0,-1);
\fill[color=blue] (0,-1) circle(2pt);
\draw[color=blue] (1,0) -- (0,1);
\fill[color=blue] (0,1) circle(2pt);
\draw[color=blue] (1,1) -- (0,2);
\fill[color=blue] (0,2) circle(2pt);

\end{tikzpicture}
\caption{Four timelike geodesics in the closed half-Minkowski spacetime that cannot be extended.}\label{fig-clo-hal-min}
\end{center}
\end{figure}
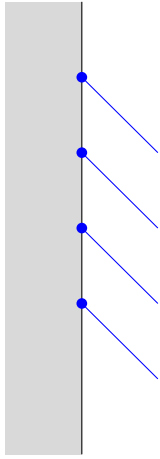

\medskip

Finally, we are in a position to establish the triangle inequality for angles for arbitrarily time oriented timelike curves. This complements Theorem \ref{thm-ang-tri-equ}.

\begin{thm}[Triangle inequality for (upper) angles --- the case of mixed time orientation]\label{thm-ang-tri-equ-oth-cas}
 Let \Xll be a strongly causal \LpLS with $\tau$ locally finite-valued and locally continuous and let $\alpha,\beta,\gamma\colon$ $[0,B)\rightarrow X$ be timelike curves starting at $x:=\alpha(0)=\beta(0)=\gamma(0)$.
 \begin{enumerate}
  \item\label{thm-ang-tri-equ-oth-cas-fufupa} If $\alpha$, $\beta$ are future directed, $\gamma$ is past directed and the angle $\ma_x(\alpha,\gamma)$ exists, we get the \emph{triangle inequality of angles}
\[
\ma_x(\alpha,\gamma)\leq\ma_x(\alpha,\beta)+\ma_x(\beta,\gamma)\,.
\]
\item\label{thm-ang-tri-equ-oth-cas-fupafu} If we assume that $X$ is locally causally closed, has timelike curvature bounded below and satisfies geodesic prolongation, the triangle inequality of angles also holds for $\alpha$, $\gamma$ being future directed curves and $\beta$ being a past directed geodesic. 
 \end{enumerate}
 See Figure \ref{fig-tri-ineq-cases} for an illustration of the two cases.
\end{thm}
By time reversal and switching the roles of $\alpha$ and $\gamma$, this statement and that of Theorem \ref{thm-ang-tri-equ} cover all possible combinations of time orientation.
\begin{pr}
(i): We restrict to a neighborhood of $x$ such that $\tau$ is finite and continuous and all curve segments are contained in it. Let $\alpha$, $\beta$ be future directed, $\gamma$ past directed and such that the angle $\ma_x(\alpha,\gamma)$ exists.
By definition (and existence) of the angle $\ma_x(\alpha,\gamma)$ we have that for all $\eps>0$ there is a $\delta>0$ such that $\forall r,s,t\in (0,\delta)$ we have the upper bound
\begin{align*}
\tilde{\ma}_x(\alpha(r),\beta(s))&<\ma_x(\alpha,\beta)+\eps\,,\\
\tilde{\ma}_x(\beta(s),\gamma(t))&<\ma_x(\beta,\gamma)+\eps\,,\\
\tilde{\ma}_x(\alpha(r),\gamma(t))&>\ma_x(\alpha,\gamma)-\eps\,,
\end{align*}
whenever $\alpha(r)\leq\beta(s)$ or the other way around ($\gamma(t)\ll\beta(s)$ and $\gamma(t)\ll\alpha(r)$ are always satisfied, as $\gamma$ is past directed).

We denote $a=\alpha(r)$, $b=\beta(s)$ and $c=\gamma(t)$. We let $a\leq b$ (fixing $s>0$, there is an $r$ such that this is satisfied). For now assume that $\tilde\ma_x(\beta(s),\gamma(t))>0$.

We form a comparison situation in $\lm{0}$ as follows: Let $\bar{x},\bar{a},\bar{b},\bar{c}$ containing comparison triangles for $\Delta cxb$ and $\Delta xab$, where we let $\bar{a}$ and $\bar{c}$ on the same side of (the line given by the segment) $[\bar{x}\bar{b}]$. We consider the segment $L=[\bar c\bar b]$ and its relation to $\bar a$. We vary $r$ and will get two cases either $\bar{a}$ lying on $L$ is possible or if not, we approximate appropriately. For this, we use continuity of $\bar{a}$ when varying $r$. For small $r$, it is clear $\bar{a}$ will lie below $L$ (as $\bar{x}$ is below $L$). As $r\to\sup \alpha^{-1}(J^-(b))=: r^+$ we have $\tau(a,b)\to 0$ and thus $\bar{a}$ has to approach $\partial I^-(\bar{b})$. As the part of $\partial I^-(\bar{b})$ lying on the same side of $[\bar{x}\bar{b}]$ as $\bar{c}$ lies entirely above or on $L$ except for $\bar{b}$ itself. Thus either $\bar{a}$ is eventually above $L$ or $\bar{a}\to\bar{b}$ (which is on $L$). 

First, we discuss the case where $\bar{a}$ is eventually on or above $L$, so we have a value for $r$ where $\bar a$ is on $L$. Now we have the situation where $\bar{c}\ll\bar{a}\leq\bar{b}$ all lie on a straight line, i.e., $\bar\tau(\bar{c},\bar{b})=\bar\tau(\bar{c},\bar{a})+\bar\tau(\bar{a},\bar{b})$. 
As we have $\tau(c,b)\geq\tau(c,a)+\tau(a,b)$ and the side-lengths agree with the ones in the comparison triangles $\Delta\bar c\bar x\bar b$, $\Delta \bar x\bar a\bar b$ except $\tau(c,a)$, we get that $\bar\tau(\bar{c},\bar{a})\geq\tau(c,a)$. This yields for the angles that
\begin{align*}
\tilde{\ma}_x(a,b)&={\ma}_{\bar x}^{\lm{0}}(\bar{a},\bar{b})\,,\\
\tilde{\ma}_x(c,b)&={\ma}_{\bar x}^{\lm{0}}(\bar{c},\bar{b})\,,\\
\tilde{\ma}_x(c,a)&\leq{\ma}_{\bar x}^{\lm{0}}(\bar{c},\bar{a})\,,
\end{align*}
where the last inequality holds by the monotonicity of the law of cosines (Remark \ref{rem-loc-mon}, varying the longest side, $\bar\tau(\bar{c},\bar{a})\geq\tau(c,a)$). In the comparison configuration, i.e., in Minkowski spacetime we can use additivity of angles, and thus we get that
\[
\tilde{\ma}_x(c,a)\leq{\ma}_{\bar{x}}(\bar{c},\bar{a})={\ma}_{\bar{x}}(\bar{a},\bar{b})+{\ma}_{\bar{x}}(\bar{c},\bar{b})=\tilde{\ma}_x(a,b)+\tilde{\ma}_x(c,b)\,,
\]
which is what we need. We have proven that the triangle inequality of angles is not violated by more than $3\eps$. 
% (Note we varied $r$ and used the estimate for $\ma_x(\alpha,\gamma)$, so we used that the angle $\ma_x(\alpha,\gamma)$ exists.)

Now let $\bar a$ be under $L$ for all $r$\footnote{This case will not occur if the neighborhood is causally closed, since then the limit value for $r=r^+$ is allowed, for which $\tau(a,b)=0$, so $\bar a=\bar b$ or $\bar a$ is above $L$.}, then $\bar a\to\bar b$ as $r\to r^+$. We adapt the above proof to this case, and can assume that we have the above construction except for the location of $\bar a$. Now, we can take $r$ close enough to $r^+$ such that $\bar\tau(\bar{c},\bar{a})\geq\tau(c,a)-\tilde{\eps}$ for some constant $\tilde{\eps}>0$ to be specified later. This is possible as $\bar{a}\to\bar{b}$, $\tau(c,b)\geq\tau(c,a)$ and $\tau$ is continuous. Using the law of cosines (\ref{lorLawOfCosines}), we get
\begin{align*}
\cosh(\tilde{\ma}_x(a,c))&=\frac{\tau(c,a)^2-\tau(x,a)^2-\tau(c,x)^2}{2\tau(x,a)\tau(c,x)}\,,\\
\cosh({\ma}_{\bar x}^{\lm{0}}(\bar{a},\bar{c}))&=\frac{\bar\tau(\bar{c},\bar{a})^2-\tau(x,a)^2-\tau(c,x)^2}{2\tau(x,a)\tau(c,x)}\,,
\end{align*}
where we used the equality of all distances except $\tau(c,a)$ and $\bar\tau(\bar{c},\bar{a})$ in the comparison situation. The difference is
\begin{align}
\cosh({\ma}_{\bar x}^{\lm{0}}(\bar{a},\bar{c}))-\cosh(\tilde{\ma}_x(a,c))&=\frac{\bar\tau(\bar{c},\bar{a})^2-\tau(c,a)^2}{2\tau(x,a)\tau(c,x)}\\
&\geq\frac{-(\bar\tau(\bar{c},\bar{a})+\tau(c,a))\tilde{\eps}}{2\tau(x,a)\tau(c,x)}\,.
\end{align}
We have that $\tau(c,a),\bar\tau(\bar c,\bar a)\leq \tau(c,b)$ stay bounded as $r\to r^+$ and $\tau(x,a)$ is increasing and thus bounded away from $0$. Thus the right-hand-side goes to $0$ as $r\to r^+$ and thus $\tilde \eps\to 0$. As the inverse function of $\cosh$ is uniformly continuous, we find a small enough $\tilde{\eps}>0$ for each $\eps>0$ such that 
\[
{\ma}_{\bar x}^{\lm{0}}(\bar{a},\bar{c})-\tilde{\ma}_x(a,c)\geq-\eps\,.
\]
Now we have a chain of inequalities as above
\[
\tilde{\ma}_x(a,c)-\eps\leq\ma_{\bar x}^{\lm{0}}(\bar{a},\bar{c})=\ma_{\bar x}^{\lm{0}}(\bar{a},\bar{b})+\ma_{\bar x}^{\lm{0}}(\bar{b},\bar{c})=\tilde{\ma}_x(a,b)+\tilde{\ma}_x(b,c)\,,
\]
which gives analogous to the above that the triangle inequality of angles is not violated by more than $4\eps$.

For the case where we cannot achieve $\tilde\ma_x(\beta(s),\gamma(t))>0$ for any small $s,t>0$, we conclude that the concatenation of $\beta,\gamma$ is a geodesic, this means that $\tilde{\ma}_x(b,c)=\ma_x(\beta,\gamma)=0$. We now look at how the distances behave as $r\to0$: $\tau(c,b)$ is constant, and by applying appropriate limits in the law of cosines (\ref{lorLawOfCosines}) we get 
\begin{align*}
\tau(c,a(r))&=t+r\cosh(\tilde{\ma}_x(c,a(r)))+O(r^2)\,,\\
\tau(a(r),b)&=s-r\cosh(\tilde{\ma}_x(a(r),b))+O(r^2)\,.
\end{align*}

We now sum these two, noting that $\tau(c,b)=t+s$. The reverse triangle inequality yields $\tau(c,b)\geq\tau(c,a(r))+\tau(a(r),b)$, which in the limit as $r\to0$ implies

\[
\tilde{\ma}_x(a(r),b)\geq\tilde{\ma}_x(c,a(r))\,.
\] 

Together with $\tilde{\ma}_x(c,b)=0$, we have proven that the triangle inequality of angles is not violated by more than $2\varepsilon$.
\medskip

(ii): At this point we assume that $X$ is locally causally closed, has timelike curvature bounded below by $K\in\R$ and satisfies geodesic prolongation. Let $\alpha$, $\gamma$ be future directed timelike curves and $\beta$ a past directed geodesic.

First, we reduce this to the case where $\beta$ and $\gamma$ join up to a geodesic. By geodesic prolongation, we can extend $\beta$ to the future of $x$ by a future directed geodesic $\tilde{\beta}$. By Lemma \ref{lem-ang-geo}.\ref{lem-ang-geo-int}, we have that $\ma_x(\beta,\tilde{\beta})=0$. We claim that the following triangle inequalities of angles hold
\begin{align*}
\ma_x(\alpha,\tilde{\beta})&\leq\ma_x(\alpha,\beta)+\underbrace{\ma_x(\beta,\tilde{\beta})}_{=0}\,,\\
\ma_x(\tilde{\beta},\gamma)&\leq\underbrace{\ma_x(\tilde{\beta},\beta)}_{=0}+\ma_x(\beta,\gamma)\,.
\end{align*}
Note these claimed triangle inequalities are just like the desired triangle inequality, i.e., a future-past-future case. Assuming this, we get the desired triangle inequality
\begin{equation*}
\ma_x(\alpha,\gamma)\leq\ma_x(\alpha,\tilde{\beta})+\ma_x(\tilde{\beta},\gamma)\leq\ma_x(\alpha,\beta)+\ma_x(\beta,\gamma)\,,
\end{equation*}
where the first inequality is by the future-only triangle inequality of angles, i.e., Theorem \ref{thm-ang-tri-equ}. 

 So we are left with (without loss of generality) establishing $\ma_x(\alpha,\beta)\geq \ma_x(\alpha,\tilde{\beta})$. Let $U$ be a comparison neighborhood of $x$, and by strong causality we can assume that all the points and geodesic segments are contained in $U$. Let $\eps>0$ and choose some parameters $r,s,t>0$ small enough such that
\begin{align*}
\tilde{\ma}^K_x(\alpha(r),\beta(s))&<\ma_x(\alpha,\beta)+\eps\,,\\
\tilde{\ma}^K_x(\alpha(r),\tilde\beta(t))&>\ma_x(\alpha,\tilde\beta)-\eps\,,\\
\alpha(r)&\leq \tilde\beta(t)\,,
\end{align*}
where we used Proposition \ref{lorSphericalangleNoDepK} to replace the Minkowski comparison angle with the $K$-comparison angle and without loss of generality choose the time orientation in the third line (the other case being analogous). We denote $a:=\alpha(r)$, $b:=\beta(s)$ and $c:=\tilde\beta(t)$.

We now construct the situation of the straightening lemma \ref{cor-straightening-of-middle-line}: We form comparison triangles $\Delta \bar{x}\bar{a}\bar{c}$ for $\Delta xac$ and $\Delta \bar{b}\bar{x}\bar{a}$ for $\Delta bxa$ in $\lm{K}$ such that they share the side $[\bar x\bar a]$ and $\bar{b}$, $\bar{c}$ lie on different sides of (the line extending) $[\bar x\bar a]$. Note that $\bar\tau(\bar b,\bar a) + \bar\tau(\bar a,\bar c) = \tau(b,a) + \tau(a,c)\leq \tau(b,c) = \tau(b,x) + \tau(x,c) = \bar\tau(\bar b,\bar x) + \bar\tau(\bar x,\bar c)$, as required by the straightening lemma. These triangles realize the $K$-comparison angles $\ma_{\bar x}^{\lm{K}}(\bar a,\bar c) = \tilde\ma_x^K(a,c)$ and $\ma_{\bar x}^{\lm{K}}(\bar b,\bar a) = \tilde\ma_x^K(b,a)$.  Let $\Delta \bar{a}'\bar{b}'\bar{c}'$ be a comparison triangle for $\Delta abc$ and let $\bar{x}'\in[\bar b',\bar c']$ with $\bar\tau(\bar b',\bar x')=\bar\tau(\bar b,\bar x) =\tau(b,x)$. Also note that $\bar\tau(\bar b',\bar a')=\bar\tau(\bar b,\bar a)$, $\bar\tau(\bar a',\bar c') = \bar\tau(\bar a,\bar c)$ and $\bar\tau(\bar b',\bar c')= \tau(b,c) = \tau(b,x) + \tau(x,c) = \bar\tau(\bar b,\bar x) + \bar\tau(\bar x,\bar c)$. 

At this point we use the lower curvature bound to obtain $\tau(x,a)=\bar{\tau}(\bar{x},\bar{a})\leq\bar\tau(\bar{x}',\bar{a}')$. Thus by the straightening lemma \ref{cor-straightening-of-middle-line}, we get that $\tilde\ma_x^K(b,a)\geq\tilde\ma_x^K(a,c)$ which implies
\[
\ma_x(\alpha,\beta)+\eps>\tilde\ma_x^K(a,b)\geq\tilde\ma_x^K(a,c)>\ma_x(\alpha,\tilde\beta)-\eps\,.
\]
As this holds for all $\eps>0$ we conclude that $\ma_x(\alpha,\tilde\beta)\leq\ma_x(\alpha,\beta)$ and we are done.
\end{pr}

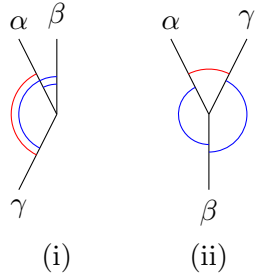
\begin{figure}[h!]
\begin{center}
 \begin{tikzpicture}
\draw (0,0) -- (-0.5,1) node[anchor=south]{$\alpha$};
\draw (0,0) -- (0,1) node[anchor=south]{$\beta$};
\draw (0,0) -- (-0.5,-1) node[anchor=north]{$\gamma$};

\coordinate (x) at (0,0);
\coordinate (a) at (-0.5,1);
\coordinate (b) at (0,1);
\coordinate (c) at (-0.5,-1);

\draw pic[draw,angle radius={0.4cm},color=blue]{angle = b--x--a} ;
\draw pic[draw,angle radius={0.5cm},color=blue]{angle = b--x--c} ;
\draw pic[draw,angle radius={0.6cm},color=red]{angle = a--x--c} ;
\draw (0,-1.9) node{(i)};

\begin{scope}[shift={(2,0)}]
\draw (0,0) -- (-0.5,1) node[anchor=south]{$\alpha$};
\draw (0,0) -- (0,-1) node[anchor=north]{$\beta$};
\draw (0,0) -- (0.5,1) node[anchor=south]{$\gamma$};

\coordinate (x) at (0,0);
\coordinate (a) at (-0.5,1);
\coordinate (b) at (0,-1);
\coordinate (c) at (0.5,1);

\draw pic[draw,angle radius={0.4cm},color=blue]{angle = a--x--b} ;
\draw pic[draw,angle radius={0.5cm},color=blue]{angle = b--x--c} ;
\draw pic[draw,angle radius={0.6cm},color=red]{angle = c--x--a} ;
\draw (0,-1.9) node{(ii)};
\end{scope}

\end{tikzpicture}
\caption{The configuration of the three geodesics in the triangle inequality for angles.}\label{fig-tri-ineq-cases}
\end{center}
\end{figure}

\begin{ex}[Necessity of the assumptions in the triangle inequality with mixed time orientations]
     If $\alpha$, $\gamma$ are future directed and $\beta$ is past directed, the triangle inequality may fail as the following example shows. We consider a causal funnel (cf.\ \cite[Ex.\ 3.19]{KS:18}) with a (Minkowski-)timelike curve $\beta$, i.e., $X=\beta((-\infty,0])\cup J^+(\beta(0))$ in $n$-dimensional Minkowski spacetime $\R^n_1$, and set $x:=\beta(0)$, see Figure \ref{fig-causal_funnel_branching}. It exhibits timelike branching, say $\beta$ branches into future directed timelike geodesics $\alpha$ and $\gamma$ at $x$. Then $\alpha$ and $\gamma$ have $\ma_x(\alpha,\gamma)>0$. However, as both the concatenation of $\beta$ and $\alpha$ and the concatenation of $\beta$ and $\gamma$ is a geodesic, we have that $\ma_x(\alpha,\beta)=\ma_x(\beta,\gamma)=0$ by Lemma \ref{lem-ang-geo},\ref{lem-ang-geo-int}. Thus the triangle inequality of angles is not satisfied: $0=\ma_x(\alpha,\beta)+\ma_x(\beta,\gamma)\not\geq\ma_x(\alpha,\gamma)$. As a causal funnel has timelike curvature unbounded below, this  shows the necessity of the timelike curvature bound below in Theorem \ref{thm-ang-tri-equ-oth-cas}.

\begin{figure}[h!]
\begin{center}
\begin{tikzpicture}
\fill[color=black!15] (0,0) -- (-2,2) -- (2,2) -- (0,0);
\draw (0,-2) -- (0,0) -- (-2,2) (0,0) -- (2,2);

\draw[color=blue] (1pt,-2) -- (1pt,0) (0,0) -- (-1,2) (0,0) -- (1,2);

\draw [color=orange,fill=orange,fill opacity=0.1] (0,0) -- (0,-0.3) arc (-90:63.45:0.3) -- cycle;
\draw [color=orange,fill=orange,fill opacity=0.1] (0,0) -- (0,-0.4) arc (-90:{-180-63.45}:0.4) -- cycle;
\draw [color=orange,fill=orange,fill opacity=0.1] (0,0) -- ({1*0.5/sqrt(5)},{2*0.5/sqrt(5)}) arc (63.45:{180-63.45}:0.5) -- cycle;

\draw[color=orange!70!black] ({2*0.3/sqrt(5)},{-1*0.3/sqrt(5)}) node[anchor=west]{$\omega_1=0$} ({-2*0.4/sqrt(5)},{-1*0.4/sqrt(5)}) node[anchor=east]{$\omega_2=0$} (0,{1*0.5})[anchor=south] node{$\omega_3>0$};

\draw[color=blue] (0,-2) node[anchor=west]{$\beta$} (-1,2) node[anchor=west]{$\gamma$} (1,2) node[anchor=east]{$\alpha$};

% causal funnel together with counterexample FuPaFu triangle inequality of angles.
\end{tikzpicture}
\caption{When timelike geodesics branch with a positive angle (e.g.\ the causal funnel here), one of the triangle inequalities for angles does not hold. }\label{fig-causal_funnel_branching}
\end{center}
\end{figure}
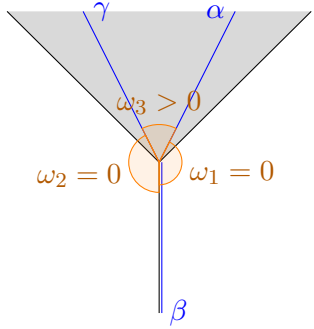
\end{ex}

At this point, we use the triangle inequality of angles to get an analogue of the metric fact that the angles along a geodesic add to $\pi$.
\begin{cor}[Triangle equality along geodesics]\label{cor-tri-ine-alo-geo}
Let \Xll be a strongly causal and locally causally closed \LpLS with timelike curvature bounded below. Let $\alpha_+\colon[0,\varepsilon)\to X$ and $\alpha_-\colon[0,\varepsilon)\to X$ be a future and a past directed geodesic with $x:=\alpha_+(0)=\alpha_-(0)$ which join up to a geodesic. Let $\beta\colon[0,\varepsilon)\to X$ be a future directed timelike curve. If the angle $\ma_x(\beta,\alpha_-)$ exists, then 
\[
\ma_x(\beta,\alpha_+)=\ma_x(\beta,\alpha_-)\,.
\]
\end{cor}
\begin{proof}
As the angle $\ma_x(\beta,\alpha_-)$ exists, we use the triangle inequality of angles (Theorem \ref{thm-ang-tri-equ-oth-cas},\ref{thm-ang-tri-equ-oth-cas-fufupa}) to get that
\[
\ma_x(\beta,\alpha_-)\leq\ma_x(\beta,\alpha_+)+\underbrace{\ma_x(\alpha_+,\alpha_-)}_{=0}\,.
\]
Here we used that $\ma_x(\alpha_+,\alpha_-)=0$ by Lemma \ref{lem-ang-geo},\ref{lem-ang-geo-int}. We also apply the triangle inequality of angles to $\alpha_+,\alpha_-,\beta$: We assume a timelike curvature bound below, but technically, we did not assume geodesic prolongation, however for the proof of \ref{thm-ang-tri-equ-oth-cas},\ref{thm-ang-tri-equ-oth-cas-fupafu} we only assumed $\alpha_-$ to be extendible as a geodesic, which it is (with extension $\alpha_+$). Thus we conclude that
\[
\ma_x(\beta,\alpha_+)\leq\ma_x(\beta,\alpha_-)+\underbrace{\ma_x(\alpha_-,\alpha_+)}_{=0}\,.
\]
In total, we have $\ma_x(\beta,\alpha_+)=\ma_x(\beta,\alpha_-)$.
\end{proof}

In \cite[Thm.\ 4.12]{KS:18} we established that given a timelike lower curvature bound timelike distance realizers do not branch using either one of two additional hypotheses. The first being that the space is timelike locally uniquely geodesic and the second being that locally in $X$ maximal future directed timelike curves are strictly longer than any future directed causal curve with the same endpoints that contains a null segment. These two hypotheses are only used to establish the existence of a non-degenerate timelike triangle. Using angles we were able to remove these two hypotheses and therefore establish that $K$-monotonicity comparison from below implies timelike non-branching. However, after the fact we were able to avoid the use of angles completely, and therefore make the result even stronger because we do not need local timelike geodesicness for 
applying Theorem \ref{thm-K-ang-com-fu}, i.e., that a lower timelike curvature bound implies future $K$-monotonicity comparison from below. Thus the result obtained is now completely analogous to the metric case, see e.g.\ \cite[Lem.\ 2.4]{Shi:93} where a lower bound on curvature (in the Alexandrov sense) implies non-branching of distance realizers.

\begin{thm}[Timelike non-branching]\label{thm-non-bra}
Let $X$ be a strongly causal \LpLS with timelike curvature bounded below by some $K\in\mb{R}$. Then timelike distance realizers cannot branch.
\end{thm}
\begin{proof}
Let indirectly $\alpha,\beta:[-B,B]\to X$ be timelike distance realizers parametrized by $\tau$-arclength such that $\alpha|_{[-B,0]}=\beta|_{[-B,0]}$ and $\alpha((0,B))\neq\beta((0,B))$. Define $x=\alpha(0)=\beta(0)$ the last known common point. 

We distinguish two cases. First, we assume that for all $s,t>0$, such that $\alpha(s)\ll\beta(t)$, we have $\tau(\alpha(s),\beta(t))=t-s$ and for $\beta(t)\ll\alpha(s)$, we have $\tau(\beta(t),\alpha(s))=s-t$. Thus, e.g.\ in the former case we can form the timelike triangle $x\ll\alpha(s)\ll\beta(t)$ and for any $r>0$ the degenerate timelike triangle $\alpha(-r)=\beta(-r)\ll\alpha(s)\ll\beta(t)$. Similarly, in the latter case. 
Now for $t>0$ consider $s_t:=\sup\{s:\alpha(s)\ll\beta(t)\}$. Then $\lim_{s\nearrow s_t} \tau(\alpha(s),\beta(t))=t-s_t$ and $\tau(\alpha(s),\beta(t))=0$ for $s>s_t$, so $s\mapsto\tau(\alpha(s),\beta(t))$ has a jump of size $=t-s_t$ at $s=s_t$. As $\tau$ is continuous, the jump-size has to be $t-s_t=0$. Consequently, for all $s<t$ we have $\alpha(s)\ll\beta(t)$ and hence $\tau(\alpha(s),\beta(t))=t-s$. Analogously we have for all $t<s$ that $\beta(t)\ll\alpha(s)$ and $\tau(\beta(t),\alpha(s))=s-t$. Now, we get that $\alpha(s-\varepsilon)\ll\beta(s)\ll\alpha(s+\varepsilon))$, i.e., $\beta(s)\in I(\alpha(s-\varepsilon),\alpha(s+\varepsilon))$ for all $\varepsilon>0$. By strong causality these timelike diamonds form a neighborhood base of $\alpha(s)$, and as this being along a timelike curve, we conclude that $\alpha(s)=\beta(s)$. As this holds for all $s$ we obtain $\alpha=\beta$ --- a contradiction.

In the second case, there are $s,t\in(0,B)$ such that, without loss of generality, $x\ll\alpha(s)\ll\beta(t)$ is a non-degenerate timelike triangle. Thus also for every $r>0$ the timelike triangle $p:=\alpha(-r)=\beta(-r)\ll\alpha(s)=:q\ll\beta(t)=:z$ is non-degenerate. 
At this point we can argue as in the proof of \cite[Thm.\ 4.12]{KS:18}. For the sake of completeness we include the argument here. Let $\bar\Delta=\Delta \bar p \bar q \bar z$ be a comparison triangle of $\Delta p q z$ in $\lm{K}$. Let $\bar x_1\in[\bar p,\bar q]$ correspond to $x$ and $\bar x_2\in[\bar p,\bar z]$ correspond to $x$ as well, i.e., $\bar\tau(\bar p,\bar x_i)=\tau(p,x)=r$ for $i=1,2$. As $\bar\Delta$ is non-degenerate we have $\bar x_1\not\in [\bar p,\bar z]$. Consequently, $\bar\tau(\bar x_1,\bar z)<\bar\tau(\bar x_2,\bar z)$ as otherwise the broken geodesic going from $\bar p$ to $\bar x_1$ to $\bar z$ is as long as the unbroken geodesic $[\bar p,\bar z]$. Finally, we arrive at
\begin{equation}
 \tau(x,z) = \bar\tau(\bar x_2,\bar z) > \bar\tau(\bar x_1,\bar z) \geq \tau(x,z)\,,
\end{equation}
where we used the timelike curvature bound from below in the last inequality --- a contradiction.
\end{proof}

\subsection{Monotonicity comparison}\label{subsec-ang-com}
In this final subsection we introduce a version of $K$-monotonicity which holds for an arbitrary angle in a timelike geodesic triangle by using the signed angle $\ma^{K,\mathrm{S}}$. The main result is that $K$-monotonicity is equivalent to curvature bounded by $K$.

\begin{defi}[$K$-monotonicity]\label{def-ang-com}
 Let $\Xll$ be a \LpLS and let $K\in\R$. We say that $X$ satisfies \emph{timelike $K$-monotonicity comparison from below (above)} 
if every point in $X$ possesses a neighborhood $U$ such that
\begin{enumerate}
\item $\tau|_{U\times U}$ is finite and continuous.
\item \label{def:cbb-max-cc} Whenever $x$, $y \in U$ with $x\ll y$, there exists a future-directed maximal timelike geodesic $\alpha$ in $U$ from $x$ to $y$.
% \item Let $x,y,z$ be timelike related and forming a timelike geodesic triangle in $U$, whose side lengths satisfy timelike size bounds for $K$. Let $z'\in[x,z]$ and $y'\in[x,y]$ with $y'\neq x\neq z'$ and $z'$ timelike related to $y'$, then we have
% \begin{equation}
%  \tilde\ma_x^{K,\mathrm{S}}(y',z')\leq \tilde\ma_x^{K,\mathrm{S}}(y,z)\qquad (\tilde\ma_x^{K,\mathrm{S}}(y',z')\geq \tilde\ma_x^{K,\mathrm{S}}(y,z))\,.
% \end{equation}
% 
% Additionally, in case of $K$-monotonicity comparison from above we require that if $y'$ and $z'$ are not timelike related then $\bar y'$ and $\bar z'$ are not timelike related in $\lm{K}$, where $\bar y',\bar z'$ are corresponding points of $y',z'$ on any comparison triangle of $\Delta xyz$\footnote{This is only relevant for $\sigma=-1$, otherwise $y',z'$ and $\bar y',\bar z'$ are always timelike related.}.

\item Whenever $\alpha:[0,a]\to U,\beta:[0,b]\to U$ are timelike distance realizers in $U$ with $x:=\alpha(0)=\beta(0)$, we define the function $\theta\colon (0,a]\times(0,b]\supseteq D\rightarrow \mb{R}$ by
\begin{equation}
 \theta(s,t):=\tilde\ma_x^{K,\mathrm{S}}(\alpha(s),\beta(t)),
\end{equation}
where $(s,t)\in D$ precisely when $\alpha(s),\beta(t)$ are causally related. We require this to be monotonically increasing (decreasing) in $s$ and $t$ for timelike $K$-monotonicity comparison from below (above).

%Additionally, in the case of $K$-monotonicity comparison from above we consider $0<s'\leq s$ and $0<t'\leq t$ and require that if $\alpha(s'),\beta(t')$ are not causally related but $\alpha(s),\beta(t)$ are, then the comparison points for $\alpha(s'),\beta(t')$ in a comparison triangle for $\Delta x \alpha(s) \beta(t)$ are not timelike related.
\end{enumerate}
\end{defi}

A direct consequence of $K$-monotonicity comparison is that it implies the existence of angles between timelike geodesics depending on monotonicity comparison from below or above and the time orientation of the geodesics. The analogous result in metric geometry (see e.g.\ \cite[Prop.\ 4.3.2]{BBI:01} or \cite[Prop.\ II.3.1]{BH:99}) does not have this conditional dependence, of course.

\begin{lem}[Monotonicity implies existence of angles between geodesics of arbitrary time orientation]\label{lem-K-ang-com-ang-ex}
Let $\Xll$ be a \LpLSn. If $X$ satisfies timelike $K$-monotonicity comparison from above or below, the angle between any two timelike geodesics starting at the same point $x$ (both future or past directed or one future, one past directed) exists if it is finite\footnote{The limit superior in the definition of upper angles is a limit, but we do not know if it is finite}.

Moreover, in case $X$ satisfies timelike $K$-monotonicity comparison from below and one of the two timelike geodesics is future directed and one is past directed (i.e., the $\sigma=1$ case), the angle between them exists.

Similarly, in case $X$ satisfies timelike $K$-monotonicity comparison from above and the two timelike geodesics are both future directed or both past directed (i.e., the $\sigma=-1$ case), the angle between them exists.

%If it satisfies timelike $K$-angle comparison from below, the angle between any two future directed (or two past directed) timelike geodesics starting at the same point exists and is finite (i.e.\ the $\sigma=-1$ case).

%If it satisfies timelike $K$-angle comparison from above, the angle between a future directed timelike geodesic and a past directed timelike geodesic starting at the same point exists and is finite  (i.e.\ the $\sigma=1$ case).
\end{lem}
\begin{pr}
Let $U$ be a comparison neighborhood given by Definition \ref{def-ang-com} above and let $\alpha,\beta:[0,\varepsilon)\to U$ be two future or past directed timelike geodesics with $x:=\alpha(0)=\beta(0)$. Without loss of generality we can assume that either both are future directed or $\alpha$ is past directed and $\beta$ is future directed.
Let $s>0$ and $t>0$ such that $\alpha(s)$ and $\beta(t)$ are causally related. Then by $K$-monotonicity comparison from below (or above) we get that $\theta(s,t)=\tilde\ma_x^{K,\mathrm{S}}(\alpha(s),\beta(t))$ is monotone (decreasing if we have $K$-monotonicity comparison from below, increasing if we have it from above), and so this sequence has a limit. 
By Corollary \ref{lorSphericalangleNoDepK} this limit is $\ma_x^\mathrm{S}(\alpha,\beta)$, but we do not yet know it is finite. Now if $\sigma=1$, we have $0$ as a lower bound for all signed comparison angles, and if $\sigma=-1$, we have $0$ as an upper bound. So the angles are finite (and thus exist) in two cases: In case of $K$-monotonicity comparison from below and $\sigma=1$, and in case of $K$-monotonicity comparison from above and $\sigma=-1$.
\end{pr}

Analogous to Corollary \ref{cor-K-mon-ang-bou} a direct consequence of $K$-monotonicity comparison and Proposition \ref{lorSphericalangleNoDepK} is that, given a curvature bound, $K$-comparison angles bound the angle between geodesics.
\begin{cor}[Monotonicity implies bound on signed angle of geodesics]\label{cor-K-mon-ang-bou-sig}
Let $\Xll$ be a \LpLS that satisfies $K$-monoton\-ic\-ity comparison from below (above) for some $K\in\R$. Then for any $x\in X$ and $\alpha,\beta\colon[0,B]\rightarrow X$ timelike geodesics with $\alpha(0)=\beta(0)=x$ one has that
\begin{equation}
 \ma_x^{\mathrm{S}}(\alpha,\beta)\leq \tilde\ma_x^{K,\mathrm{S}}(\alpha(s),\beta(t)) \qquad \Bigl(\ma_x^{\mathrm{S}}(\alpha,\beta)\geq \tilde\ma_x^ {K,\mathrm{S}}(\alpha(s),\beta(t))\Bigr)\,,
\end{equation}
for all $s,t\in[0,B]$ small enough (with respect to timelike size bounds for $K$) such that $\alpha(s)$ and $\beta(t)$ are causally related.
\end{cor}
\medskip

Monotonicity comparison also directly influences the length of the third side in a comparison hinge.
\begin{cor}[Monotonicity comparison implies hinge comparison]
Let \Xll be a \LpLS that satisfies $K$-monotonicity comparison from below (above) for some $K\in\R$. Let $(\alpha,\beta)$ be a hinge, i.e., $x\in X$ and $\alpha\colon[0,A]\to X$, $\beta\colon[0, B]\to X$ timelike geodesics of arbitrary time orientation with $\alpha(0)=\beta(0)=x$, which satisfies the timelike size bounds for $K$. Then one can form a comparison hinge $(\bar{\alpha},\bar{\beta})$ in $\lm{K}$ and we have 

\begin{equation}
 \tau(\alpha(A),\beta(B))\geq \bar{\tau}(\bar\alpha(A),\bar\beta(B)) \qquad \Bigl(\tau(\alpha(A),\beta(B))\leq \bar{\tau}(\bar\alpha(A),\bar\beta(B))\Bigr)\,.
\end{equation}
\end{cor}
\begin{proof}
We only establish the case of $K$-monotonicity comparison from below, in the other case the previous Corollary \ref{cor-K-mon-ang-bou-sig} yields a reversed inequality and so gives an analogous proof.

We denote the points by $x$, $a=\alpha(A)$ and $b=\beta(B)$. We first assume that they are causally related, say $a\leq b$. We consider the comparison hinge $(\bar{\alpha},\bar{\beta})$ in $\lm{K}$, and consider it as a triangle $\bar{\Delta}=\Delta\bar{x} \bar{a} \bar{b}$, where $\bar{x}:=\bar{\alpha}(0)=\bar{\beta}(0)$, $\bar{a}:=\bar{\alpha}(A)$ and $\bar{b}:=\bar{\beta}(B)$. We assume $\bar a\leq \bar b$.
Furthermore, we consider a comparison triangle $\bar{\Delta}'=\Delta\bar{x}' \bar{a}' \bar{b}'$ for $\Delta x a b$. These two triangles have two agreeing side-lengths, i.e.,
\begin{align*}
\bar\tau(\bar{x}',\bar{a}')&=\tau(x,a)=\bar{\tau}(\bar{x},\bar{a}) \text{ (or with arguments flipped if $\alpha$ is past directed)}\,,\\
\bar\tau(\bar{x}',\bar{b}')&=\tau(x,b)=\bar{\tau}(\bar{x},\bar{b}) \text{ (or with arguments flipped if $\beta$ is past directed)}\,.
\end{align*}
For the angles we have
\begin{align*}
\ma_{\bar{x}'}^{\lm{K}}(\bar{a}',\bar{b}')&=\tilde{\ma}_x^K(a,b)\,,\\
\ma_{\bar{x}}^{\lm{K}}(\bar{a},\bar{b})&=\ma_x(\alpha,\beta)\,.
\end{align*}
By the previous Corollary \ref{cor-K-mon-ang-bou-sig}, we know that $\ma_x^{\mathrm{S}}(\alpha,\beta)\leq\tilde\ma_x^{K,\mathrm{S}}(a,b)$.% (or $\geq$ for an upper bound of timelike curvature)

Now we distinguish the cases where the angle is at. If $\sigma=1$, we have that $\ma_{\bar{x}'}^{\lm{K}}(\bar{a}',\bar{b}')$ is at least as big as $\ma_{\bar{x}}^{\lm{K}}(\bar{a},\bar{b})$, and the monotonicity statement of the law of cosines (Remark \ref{rem-loc-mon}, varying the longest side) gives that 
\begin{equation}
\bar\tau(\bar{a},\bar{b})\leq \bar\tau(\bar{a}',\bar{b}')=\tau(a,b)\,.
\end{equation}

If $\sigma=-1$, we have that $\ma_{\bar{x}}^{\lm{K}}(\bar{a},\bar{b})$ is at least as big as $\ma_{\bar{x}'}^{\lm{K}}(\bar{a}',\bar{b}')$, and the monotonicity statement of the law of cosines (Remark \ref{rem-loc-mon}, varying a non-longest side) gives that 
\begin{equation}
\bar\tau(\bar{a},\bar{b})\leq \bar\tau(\bar{a}',\bar{b}')=\tau(a,b)\,.
\end{equation}

Now for the case that $a$ and $b$ are not causally related (thus $\sigma=-1$). In case of monotonicity comparison from above there is nothing to do, so let $X$ satisfy monotonicity comparison from below. We need to establish that $\bar\tau(\bar a, \bar b)=0$. Let $\eps>0$ small enough, then $s_\eps=\sup\{s\in[0,A]:\tau(\alpha(s),b)>\eps\}<A$, so $a_\eps:=\alpha(s_\eps)$ is the last point on $\alpha$ with $\tau(a_\eps,b)=\eps$. Then as above we get a point $\bar a_\eps$ in the comparison hinge, and the above argument yields that  $\eps=\tau(a_\eps,b)\geq\bar\tau(\bar a_\eps,\bar b)\geq \bar\tau (\bar a,\bar b)$. Now let $\eps\to0$ to get $\bar\tau (\bar a,\bar b)=0$.

Now for the case that $\bar a$ and $\bar b$ are not causally related (thus $\sigma=-1$). In case of monotonicity comparison from below there is nothing to do, so let $X$ satisfy monotonicity comparison from above. We need to establish that $a, b$ are not timelike related, so indirectly assume $a\ll b$. We have two triangles with two equal side-lengths: $\bar\Delta$ with a non-timelike side and $\bar\Delta'$ which is timelike. Now the monotonicity statement of the extended law of cosines (Corollary \ref{cor-ext-loc}) yields that the angle in the non-causal triangle is bigger, i.e., $\ma_{\bar x'}^{\lm{K}}(\bar a',\bar b')<\ma_{\bar x}^{\lm{K}}(\bar a,\bar b)$. But by monotonicity comparison from above we get that $\ma_x(\alpha,\beta)\leq \tilde\ma_x^K(a,b)$, i.e., $\ma_{\bar x}^{\lm{K}}(\bar a,\bar b)\leq\ma_{\bar x'}^{\lm{K}}(\bar a',\bar b')$, a contradiction. 
\end{proof}

We are now in a position to prove the main result of this final section. We establish that timelike curvature bounds via triangle comparison (Definition \ref{def-tri-com}) are equivalent to monotonicity comparison (Definition \ref{def-ang-com}).

\begin{thm}[Equivalence of triangle and monotonicity comparison]\label{thm-K-ang-com}
 Let $\Xll$ be a locally strictly timelike geodesically connected \LpLS and let $K\in\R$. Then $X$ has timelike curvature bounded below (above) by $K$ if and only if it satisfies $K$-monotonicity comparison from below (above).
\end{thm}
\begin{pr}
The first two conditions to be satisfied in timelike triangle comparison and $K$-monotonicity comparison are the same under the condition of local timelike geodesic connectedness when shrinking $U$ to be a strictly timelike geodesically connected neighborhood. Moreover, let $U$ be a comparison neighborhood given either by the timelike triangle comparison or by $K$-monotonicity comparison.

 First, let $X$ have timelike curvature bounded from below (above) by $K$ and let $\Delta=\Delta x y z$ be a timelike geodesic triangle in $U$ that satisfies timelike size bounds for $K$. We only prove the case where the vertex at which we consider the angle at is $x$, and we will point out where the other cases are not analogous. Let $z'\in[x,z]$, $y'\in[x,y]$ with $y'\neq x\neq z'$. We assume $y'< z'$, the case $z'< y'$ is analogous% and the case where they are not causally related is dealt with later
. Moreover, let $\bar\Delta=\ct$ be a comparison triangle of $\Delta$ in $\lm{K}$ and let $\bar y',\bar z'$ be the corresponding points of $y',z'$ assuming first that $\bar y'\leq\bar z'$. Then by the timelike curvature bound we get that $\tau(y',z')\leq \bar\tau(\bar y',\bar z')$ (or $\tau(y',z')\geq \bar\tau(\bar y',\bar z')$). Set $\alpha:=\tilde\ma_x^K(y,z)=\ma_{\bar x}^{\lm{K}}(\bar y,\bar z)=\ma_{\bar x}^{\lm{K}}(\bar y',\bar z')$. Let $(\bar x,\bar y'',\bar z'')$ be a comparison triangle of $\Delta x y' z'$, see Figure \ref{fig-angle_comparison}. If indirectly $\tilde\ma_x^K(y',z')< \alpha$ for $\sigma=-1$ or $\tilde\ma_x^K(y',z')> \alpha$ for $\sigma=1$, the monotonicity statement of the law of cosines (Remark \ref{rem-loc-mon}) implies that $\tau(y',z')=\bar\tau(\bar y'',\bar z'')>\tau(\bar y',\bar z')$ --- a contradiction to $\tau(y',z')\leq \bar\tau(\bar y',\bar z')$.  
Analogously for the case of timelike curvature bounded above. 

However, in case of timelike curvature bounded above there is an issue if $\bar y'$ is not causally related to $\bar z'$, so $\sigma=-1$.  Then the monotonicity statement of the extended law of cosines (Corollary \ref{cor-ext-loc}) gives $\alpha=\ma_{\bar x}^{\lm{K}}(\bar y,\bar z)=\ma_{\bar x}^{\lm{K}}(\bar y',\bar z')>\ma_{\bar x}^{\lm{K}}(\bar y'',\bar z'')$, which is what we wanted to prove.

%If $X$ has timelike curvature bounded from above and $y'$ and $z'$ are not causally related (so $\sigma=-1$), we get that $0\leq \bar\tau(\bar y',\bar z')\leq\tau(y',z')=0$ and analogously $\bar\tau(\bar z',\bar y')=0$. Thus $\bar y'$ and $\bar z'$ are not timelike related, as required by the $K$-monotonicity comparison from above. % If $y'$ and $z'$ are not timelike related, we have to show that $\bar y'$ and $\bar z'$ are not timelike related in $\lm{K}$ in the case of $K$-angle comparison from above, thus only in the case that $X$ is timelike curvature bounded from above. Consequently, $\bar\tau(\bar y',\bar z')\leq\tau(y',z')=0$ and analogously $\bar\tau(\bar z',\bar y')=0$.
In total, $X$ satisfies $K$-monotonicity comparison from below (or above, respectively).
 \medskip
 
Second, let $X$ satisfy $K$-monotonicity comparison from below (above), let $\Delta=\Delta x y z$ be a timelike geodesic triangle in $U$ satisfying timelike size bounds for $K$ and let $\bar\Delta=\ct$ be a comparison triangle of $\Delta$ in $\lm{K}$. Let $p\in[x,y]$, $q\in[x,z]$, the other cases work nearly analogously by considering the angles at the other vertices. Moreover, if $p=x$ or $q=x$ there is nothing to do so let $p\neq x\neq q$. First we assume that $p$ and $q$ are causally related, say $p< q$. Let $\bar p,\bar q$ be corresponding points of $p,q$ on $\bar \Delta$, and we assume $\bar p\leq \bar q$. By assumption we have $\tilde\ma_x^{K,\mathrm{S}}(p,q)\leq\tilde\ma_x^{K,\mathrm{S}}(y,z)$ (or $\tilde\ma_x^{K,\mathrm{S}}(p,q)\geq\tilde\ma_x^{K,\mathrm{S}}(y,z)$). Let $(\bar x,\bar p',\bar q')$ be a comparison triangle of $\Delta x p q$. Note that $\tilde\ma_x^K(p,q)=\ma_{\bar x}^{\lm{K}}(\bar p',\bar q')$ and $\tilde\ma_x^K(y,z)=\ma_{\bar x}^{\lm{K}}(\bar p,\bar q)$. Then the law of cosines (Remark \ref{rem-loc-mon}) implies that $\tau(p,q)=\bar\tau(\bar p',\bar q')\leq\bar\tau(\bar p,\bar q)$ (or $\tau(p,q)\geq\bar\tau(\bar p,\bar q)$, respectively). 

Now we consider the case where $\bar p$ and $\bar q$ not causally related, which can only happen if $\sigma=-1$. We need to exclude $p\ll q$, so let indirectly $p \ll q$. For the timelike curvature bounded above case, we need $\tau(p,q)\geq \bar\tau(\bar p,\bar q)=0$ which is automatically satisfied. For the timelike curvature bounded below case, we need $\tau(p,q)\leq \bar\tau(\bar p,\bar q)=0$, i.e.\ we need to exclude this case. We compare the triangles $\Delta\bar x \bar p \bar q$ and $\Delta\bar x \bar p' \bar q'$. Two of the sides are equal, and in the first triangle the other side is spacelike. Thus the monotonicity statement of the extended law of cosines (Corollary \ref{cor-ext-loc}) gives us that $\tilde\ma_x^K(y,z)=\ma_{\bar x}^{\lm{K}}(\bar p,\bar q)> \ma_{\bar x}^{\lm{K}}(\bar p',\bar q')=\tilde\ma_x^K(p,q)$, in contradiction to $K$-monotonicity comparison from below (note $\sigma=-1$).

In case $p$ and $q$ are not causally related ($\sigma=-1$) we cannot form timelike triangles, so the main argument does not work. For the timelike curvature bounded below case, we need $0=\tau(p,q)\leq \bar\tau(\bar p,\bar q)$ which is automatically satisfied. For the timelike curvature bounded above case, we need $0=\tau(p,q)\geq \bar\tau(\bar p,\bar q)$, i.e.\ $\bar p\not\ll\bar q$. %which is exactly what we assumed in the definition of $K$-monotonicity comparison. 
First assume $U$ is a causally closed neighbourhood and we indirectly assume the weaker $\bar p\leq\bar q$. We get that there exists a point $\hat{p}\in[x,y]$ with $\hat{p}\leq q$ null related. We compare the triangles $\Delta(x,\hat{p},q)$ and $\Delta(\bar{x},\bar{p},\bar{q})$: we have $\tau(x,\hat{p})<\tau(x,p)$, $\tau(\hat{p},q)=0\leq\bar{\tau}(\bar{p},\bar{q})$ and the longest side-length is equal. Thus the strict law of cosines monotonicity (Remark \ref{rem-loc-mon}) gives that $\tilde{\ma}_x^K(\hat{p},q)> \ma_{\bar{x}}^{\lm{K}}(\bar{p},\bar{q})$, which contradicts the inequality given by $K$-monotonicity comparison from above, i.e., $\tilde{\ma}_x^K(\hat{p},q)\leq\tilde{\ma}^K_x(p,q) =\ma_{\bar{x}}^{\lm{K}}(\bar{p},\bar{q})$. Without $U$ being a causally closed neighbourhood, one does not necessarily get a null related point $\hat{p}$ but a sequence $\hat{p}_n$ with $\tau(\hat{p}_n,q)\searrow 0$. Indirectly assuming $\bar p\ll\bar q$, a limiting argument gives the claim.

In total, $X$ has timelike curvature bounded below (above) by $K$.
\end{pr}

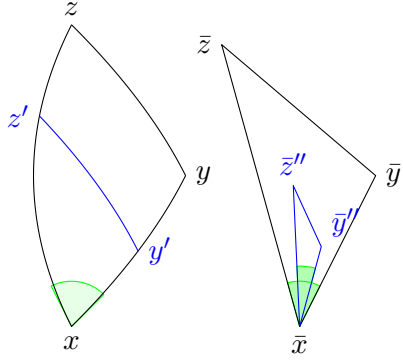
\begin{figure}[h!]
\begin{center}
\begin{tikzpicture}[line cap=round,line join=round,x=2.0cm,y=2.0cm]
\draw [shift={(0.,-1.)},color=green,fill=green,fill opacity=0.1] (0,0) -- (45.:0.3) arc (45.:116.56505117707799:0.3) -- cycle;
\draw[ smooth,samples=100,domain=-1.0:1.0] (0,-1) node[anchor=north]{$x$} plot({\x*\x/4.0-1.0/4.0},\x) node[anchor=south]{$z$};
\draw[ smooth,samples=100,domain=0.0:1.0] plot({-\x*\x/4.0-\x/2.0+1.0/2.0+1.0/4.0},\x);
\draw[ smooth,samples=100,domain=-1.0:0.0] plot({-\x*\x/4.0+\x/2.0+1.0/2.0+1.0/4.0},\x) node[anchor=west]{$y$};
\draw[color=blue, smooth,samples=100,domain=0.0:0.89564392373896] (0.4375,-0.5) node[anchor=west]{$y'$} plot({-\x*\x/4.0-\x/2.0+0.4375},{\x-0.5}) node[anchor=east]{$z'$};

\begin{scope}[xshift=4cm]
\fill [shift={(-0.5,-1.)},draw=green,fill=green!30,] (0,0) -- (74.9598086914095:0.4) arc (74.9598086914095:92.56650651538125:0.4) -- cycle;
\fill [shift={(-0.5,-1.)},draw=green,fill=green!30,] (0,0) -- (63.43494882292201:0.3) arc (63.43494882292201:105.34925915763955:0.3) -- cycle;
\draw[] (-0.5,-1.) node[anchor=north]{$\bar x$} -- (0.,0.) node[anchor=west]{$\bar y$} -- (-1.0128640109411684,0.868402399749306) node[anchor=east]{$\bar z$} -- cycle;
\draw[color=blue] (-0.5418540240528695,-0.06625804298215877) node[anchor=south]{$\bar z''$} -- (-0.5,-1.) -- (-0.3578657479246279,-0.4710322504642892) node[anchor=south west]{$\bar y''$} -- cycle;
\end{scope}
\end{tikzpicture}
\caption{When, comparing with $\lm{K}$, the interior side is shorter than expected, the $K$-angle has to behave accordingly.}\label{fig-angle_comparison}
\end{center}
\end{figure}
 \bigskip

We now give a proof of the fact that signed angles are semi-continuous, given a curvature bound by using the equivalence of timelike curvature bounds and $K$-monotonicity comparison established in Theorem \ref{thm-K-ang-com}. In particular, this implies the case of geodesics with the same time orientation, i.e., Proposition \ref{pop-ang-semicont-fudi}. In the metric setting the situation is of course easier as one does not have different time orientations and thus there is no need for signed angles, cf.\ \cite[Thm.\ 4.3.11]{BBI:01} and \cite[Prop.\ II.3.3]{BH:99}.

\begin{pop}[Semi-continuity of angles] \label{pop-ang-semicont}
 Let $\Xll$ be a locally strictly timelike geodesically connected \LpLS with timelike curvature bounded below (above). Then signed angles of timelike geo\-desics are upper (lower) semicontinuous, i.e., $\alpha,\alpha_n,\beta,\beta_n\colon[0,B]\rightarrow X$ future or past directed timelike geodesics at $x$ with $\alpha_n\to\alpha$, $\beta_n\to\beta$ pointwise, then
\begin{equation}
\ma_x^{\mathrm{S}}(\alpha,\beta)\geq \limsup_n \ma_x^{\mathrm{S}}(\alpha_n,\beta_n)\qquad (\ma_x^{\mathrm{S}}(\alpha,\beta)\leq \liminf_n \ma_x^{\mathrm{S}}(\alpha_n,\beta_n))\,.
\end{equation}
\end{pop}
\begin{pr}
First, we use the fact that the timelike curvature bound implies the corresponding $K$-monotonicity comparison by Theorem \ref{thm-K-ang-com}.
We will prove the case of timelike curvature bounded below by $K\in\R$ --- the other case is completely analogous.

For $s,t>0$ such that $\alpha_n(s)$ and $\beta_n(t)$ are causally related set $\theta_n(s,t):=\tilde\ma_x^{K,\mathrm{S}}(\alpha_n(s),\beta_n(t))$ and analogously for $s,t>0$ such that $\alpha(s)$ and $\beta(t)$ are causally related we define $\theta(s,t):=\tilde\ma_x^{K,\mathrm{S}}(\alpha(s),\beta(t))$. As the side lengths of the (comparison) triangles converge we get $\theta_n\to \theta$ pointwise if both are defined.

Without loss of generality we assume that $\beta$ is future directed. To see that the comparison angles of the approximating sequence exist pick sequences $s_k, t_k\searrow 0$ with $\alpha(s_k)\ll\beta(t_k)$. 
By the openness of the timelike relation $\ll$ (cf.\ \cite[Prop.\ 2.13]{KS:18}) we get that $\alpha_n(s_k)\ll \beta_n(t_k)$ for large enough $n\in\N$ (depending on $k$), so both $\theta$ and $\theta_n$ are defined at $(s_k,t_k)$. 

Thus Corollary \ref{cor-K-mon-ang-bou-sig} gives for $n,k$ large enough that
\begin{align}
 \ma_x^{\mathrm{S}}(\alpha_n,\beta_n)\leq \theta_n(s_k,t_k)\,.
\end{align}
Letting $n\to \infty$ gives
\begin{align}
 \limsup_{n\to\infty}\ma_x^{\mathrm{S}}(\alpha_n,\beta_n)\leq \theta(s_k,t_k)\,,
\end{align}
and taking $k\to\infty$ then yields $\limsup_{n\to\infty}\ma_x^{\mathrm{S}}(\alpha_n,\beta_n)\leq \ma_x^{\mathrm{S}}(\alpha,\beta)$, as required.

% 
% By definition of the angle, Corollary \ref{lorSphericalangleNoDepK} and $K$-monotonicity comparison (Theorem \ref{thm-K-ang-com}, noting $\tau(x,\alpha(s_k))=\tau(x,\alpha_n(s_k))$ by parametrization by arclength) we get that $\theta_n(s_k,t_k)\geq \ma_x^{\mathrm{S}}(\alpha_n,\beta_n)$ for all $n\geq n_k$. Taking the limit $n\to\infty$ yields $\theta(s_k,t_k)\geq\limsup_n\ma_x^{\mathrm{S}}(\alpha_n,\beta_n)$ and so taking the limit $k\to\infty$ gives $\ma_x^{\mathrm{S}}(\alpha,\beta)\geq\limsup_n\ma_x^{\mathrm{S}}(\alpha_n,\beta_n)$ as claimed. 
\end{pr}

Under weak additional assumptions we even get continuity of angles.
\begin{pop}[Continuity of angles]\label{pop-ang-cont-geodesics}
Let $\Xll$ be a locally strictly timelike geodesically connected and locally causally closed \LpLSn. Then angles are continuous between geodesics in the following cases. Let $x\in X$ and $\alpha,\alpha_n,\beta,\beta_n\colon[0,B]$ $\rightarrow X$ future or past directed timelike geodesics starting at $x$ with $\alpha_n\to\alpha$, $\beta_n\to\beta$ pointwise (in particular, $\alpha_n$ is future directed if and only if $\alpha$ is and similarly for $\beta$).
\begin{enumerate}
 \item The space $X$ has timelike curvature bounded above.
 \begin{enumerate}
  \item The geodesics $\alpha,\beta$ have the same time orientation or
  \item if $\alpha,\beta$ have different time orientations we assume that $X$ is strongly causal and the angles $\ma_x(\alpha,\beta)$, $\ma_x(\alpha,\beta_n)$ and $\ma_x(\alpha_n,\beta_n)$ are finite for all $n\in\N$.
 \end{enumerate}
 \item The space $X$ has timelike curvature bounded below, is strongly causal and satisfies geodesic prolongation.
\end{enumerate}
 Then
\begin{equation}
\ma_x(\alpha,\beta)= \lim_n\ma_x(\alpha_n,\beta_n)\,.
\end{equation}
\end{pop}
\begin{pr}
In both cases the timelike curvature bound implies $K$-monotonicity comparison by Theorem \ref{thm-K-ang-com} and all geodesic segments are assumed to lie in a $K$-monotonicity comparison neighborhood.

\begin{enumerate}
 \item Let $X$ have timelike curvature bounded above. By Proposition \ref{pop-ang-semicont} we have that $\ma_x(\alpha_n,\alpha)\to 0$ and $\ma_x(\beta_n,\beta)\to 0$ as $\ma_x(\alpha,\alpha)=\ma_x(\beta,\beta)=0$ by Lemma \ref{lem-ang-geo},\ref{lem-ang-geo-same} and the fact that $\alpha_n$ and $\alpha$ have the same time orientation, as well as $\beta_n$ and $\beta$.
 
 At this point we use the triangle inequality of angles. In case (a) we use Theorem \ref{thm-ang-tri-equ} and in case (b) we use Theorem \ref{thm-ang-tri-equ-oth-cas}.\ref{thm-ang-tri-equ-oth-cas-fupafu}. In case (b) we have by assumption that the angles $\ma_x(\alpha,\beta)$, $\ma_x(\alpha,\beta_n)$ and $\ma_x(\alpha_n,\beta_n)$ are finite and by Lemma \ref{lem-K-ang-com-ang-ex} these angles exist for all $n\in\N$. Thus in both cases we estimate
\begin{align}
 \ma_x(\alpha_n,\beta_n) &\leq \ma_x(\alpha_n,\alpha) + \ma_x(\alpha,\beta) + \ma_x(\beta,\beta_n)\,,\\
 \ma_x(\alpha,\beta) &\leq \ma_x(\alpha,\alpha_n) + \ma_x(\alpha_n,\beta_n) + \ma_x(\beta_n,\beta)\,.
\end{align}
Consequently, in both cases we obtain 
 \begin{equation}
 |\ma_x(\alpha_n,\beta_n)-\ma_x(\alpha,\beta)|\leq\ma_x(\alpha,\alpha_n)+\ma_x(\beta,\beta_n)\to0\,,
 \end{equation}
 as claimed.
 
 \item Let $X$ have timelike curvature bounded below. By geodesic prolongation, there is a timelike geodesic $\check{\alpha}\colon[0,\varepsilon)\to X$ having the opposite time orientation as $\alpha$ and such that the concatenation of the time reversed $\check\alpha$ and $\alpha$ is a timelike geodesic. In particular, $\ma_x(\check{\alpha},\alpha)=0$. Similarly, we find $\check\beta\colon(-\varepsilon,0]\to X$ with $\ma_x(\check{\beta},\beta)=0$. By Proposition \ref{pop-ang-semicont} we have that $\ma_x(\alpha_n,\check\alpha)\to 0$ and $\ma_x(\beta_n,\check\beta)\to 0$. 

At this point we use the triangle inequality of angles (Theorem \ref{thm-ang-tri-equ-oth-cas}). Note that angles between geodesics of different time orientation always exist by Lemma \ref{lem-K-ang-com-ang-ex}. Thus we estimate
\begin{align}
 \ma_x(\alpha_n,\beta_n) &\leq \ma_x(\alpha_n,\check\alpha) + \ma_x(\check\alpha,\check\beta) + \ma_x(\check\beta,\beta_n)\,,\\
 \ma_x(\check\alpha,\check\beta) &\leq \ma_x(\check\alpha,\alpha_n) + \ma_x(\alpha_n,\beta_n) + \ma_x(\beta_n,\check\beta)\,.
\end{align}
Consequently, in both cases we obtain
 \begin{equation}
 |\ma_x(\alpha_n,\beta_n)-\ma_x(\check\alpha,\check\beta)|\leq\ma_x(\check\alpha,\alpha_n)+\ma_x(\check\beta,\beta_n)\to0\,.
 \end{equation}
Finally, we apply Corollary \ref{cor-tri-ine-alo-geo} twice to obtain $\ma_x(\check\alpha,\check\beta)=\ma_x(\alpha,\beta)$, which gives the claim.

\end{enumerate}
\end{pr}

\section*{Conclusion}
\addcontentsline{toc}{section}{Conclusion}
In this article we introduced hyperbolic angles into the setting of Lorentzian length spaces and characterized timelike curvature bounds by an angle monotonicity property. We also introduced a diversity of metric tools, and thereby furthering the development of the theory considerably. Moreover, these tools and methods have further applications in Lorentzian geometry and General Relativity: The most striking example is the recently established \emph{splitting theorem for \LLSn s with non-negative curvature} in \cite{BORS:22}. The splitting theorem can be understood as a rigidity statement of the classical Hawking-Penrose singularity theorem \cite{HP:70} and as such is of fundamental importance in Einstein's theory of gravity.
\bigskip

Finally, let us point out several directions of future research. First, hyperbolic angles could be used to define coordinates in Lorentzian (pre-)\linebreak length spaces with timelike curvature bounded below or above. Second, a natural continuation is to establish the variation formula and an angle comparison condition as in \cite{BMS:22} also for curvature bounded from above (and for curvature bounded from below by a non-zero $K\in\R$ for the former). Further, we expect angles to play a significant role in globalization theorems for curvature bounds and in a Lorentzian version of billiards (cf.\ \cite{BFK:98}) with possible applications to particle collisions.

\section*{Acknowledgment}
We would like to thank Christian Ketterer, Michael Kunzinger, Robert McCann and Felix Rott for helpful and stimulating discussions. Moreover, we are grateful to the anonymous referee for many valuable comments and suggestions, which improved the article and enhanced its presentation. TB's research is supported by the Vienna School of Mathematics of the University of Vienna and the Technical University of Vienna. CS's research is supported by research grant J4305 of the Austrian Science Fund FWF. TB and CS additionally acknowledge support of research grant P33594 of the Austrian Science Fund FWF.

\appendix
\addcontentsline{toc}{section}{Appendix}
\section{Law of Cosines}

\begin{proof}\label{PrLorLOC}(of Lemma \ref{lorLawOfCosines} and Remark \ref{rem-loc-mon})
We first establish the case where the triangle is timelike, i.e., $a,b,c>0$. For $K=0$, we are in Minkowski space and the sides of this triangle are just straight lines. Equation (2.1) of \cite{AB:08}, gives us $-c^2=-a^2-b^2-2ab\sigma\cosh(\omega)$ by noting that their $c$ is $-c^2$ here etc., which is the desired formula.

If $K\neq0$, the sides of the triangle are future or past directed timelike geodesics $\gamma_{pq},\gamma_{qr},\gamma_{pr}$, which we assume to be defined on the domain $[0,1]$. The energy $E(\gamma_{p_ip_j})$ is given by $-\tau(p_i,p_j)^2$ if $p_i\ll p_j$. Now we apply \cite[Thm.\ 3.1.3]{Kir:18} to get
\begin{align*}
\cos(\sqrt{KE(\gamma_{pr})}) &= \cos(\sqrt{KE(\gamma_{pq})}) \cos(\sqrt{KE(\gamma_{qr})})\\
&+\left<\gamma_{qp}'(0),\gamma_{qr}'(0)\right> \frac{\sin(\sqrt{KE(\gamma_{pq})})}{\sqrt{E(\gamma_{pq})}} \frac{\sin(\sqrt{KE(\gamma_{qr})})}{\sqrt{E(\gamma_{qr})}}.
\end{align*}
Note that this uses imaginary numbers, and $\cos(ix)=\cosh(x)$, $\sin(ix)=i\sinh(x)$. \\
For $K<0$, this equation reads
\begin{align}
\cos(sc) &= \cos(sa) \cos(sc)+\sigma a b \cosh(\omega) \frac{\sin(sa)}{ia} \frac{\sin(sb)}{ib}\,,\\
\cos(sc) &= \cos(sa) \cos(sb)-\sigma\cosh(\omega) \sin(sa) \sin(sb)\,,
\end{align}
where $\sigma$ is the sign of the inner product (which is as described in the statement) and the minus comes from an $i^2$ in the denominator.\\
For $K>0$, this equation reads
\begin{align}
\cos(isc) &= \cos(isa) \cos(isb)+\sigma a b \cosh(\omega)\frac{\sin(isa)}{ia} \frac{\sin(isb)}{ib}\,,\\
\cosh(sc) &= \cosh(sa) \cosh(sb)+\sigma\cosh(\omega) \sinh(sa) \sinh(sb)\,,
\end{align}
where the $i$s in the denominator get canceled by the $i$s appearing due to $\sin(ix)=i\sinh(x)$.
\medskip

At this point, we establish the case where $c=0$, i.e., $\sigma=-1$ and without loss of generality $q\ll p\leq r$ with $p,r$ null related. We consider the timelike geodesic $\alpha$ from $q$ to $p$. For each $\eps>0$ we consider the point $p_\eps:=\alpha(1-\eps)$, then $\tau(p_\eps,r)>0$ and $\tau(p_\eps,r)\to 0$ as $\eps\to0$. We also have $\ma_q^{\lm{K}}(p_\eps,r)=\ma_q^{\lm{K}}(p,r)$. We now apply the timelike case above to the timelike triangle $q\ll p_\eps\ll r$. We note that the side-lengths converge as follows: $a_\eps=\tau(q,p_\eps)\to a=\tau(q,p)$, $b=\tau(q,r)$ is constant and $c_\eps=\tau(p_\eps,r)\to 0$. As the term containing $\omega$ stays finite, the law of cosines continues to hold in the limit $\eps\to0$, giving the claimed equality.
\bigskip

For the monotonicities, we without loss of generality assume that $K\in\{-1,0,1\}$. We look at each side varying separately. In each case, we say that $\omega$ and one of the side-lengths depend on some additional variable $t$ (i.e., we insert $\omega(t)$ and $a(t)$ or $b(t)$ or $c(t)$, respectively, into the law of cosines equations) and take the derivative of the law of cosines equations. Then $\sgn(a'(t)) = \sgn(\omega'(t))$ (or $\sgn(c'(t))$) makes $\omega$ monotonically increasing in $a$ (or $c$) and $\sgn(a'(t)) = -\sgn(\omega'(t))$ monotonically decreasing.

For $K=0$ and varying $a$, we get the equation
\begin{equation*}
a'(2a+2b\sigma\cosh(\omega))=\omega'(-2ab\sigma\sinh(\omega))\,.
\end{equation*}
If $c$ is the longest side, $\sigma=1$ and the signs of the factors are clear, we get $\sgn(a')=-\sgn(\omega')$. Otherwise $\sigma=-1$. If $b$ is the longest side, it is clear that $2b\cosh(\omega)\geq 2a$ and so $\sgn(a')=\sgn(\omega')$. If $a$ is the longest side, we have to be more careful. We use the law of cosines formula and transform it to
\begin{align*}
a^2+b^2&=c^2+2ab\cosh(\omega)\,,\\
2a^2-2ab\cosh(\omega)&=2a(a-b\cosh(\omega))=c^2-b^2+a^2\geq0\,,
\end{align*}
thus the positive $2a$ term dominates over the negative $2b\sigma\cosh(\omega)$ and so $\sgn(a')=\sgn(\omega')$.

For $K=-1$, we get the equation
\begin{align}
a'(-\sinh(a)\cosh(b)-\sigma\cosh(\omega)\cosh(a)\sinh(b))\\
=\omega'(\sigma\sinh(\omega)\sinh(a)\sinh(b))\,.
\end{align}
If $c$ is the longest side, $\sigma=1$ and the signs of the factors are clear, giving $\sgn(a')=-\sgn(\omega')$. Otherwise $\sigma=-1$. If $b$ is the longest side, $\cosh(a)\sinh(b)-\sinh(a)\cosh(b)=\sinh(b-a)\geq0$, thus $\sgn(a')=\sgn(\omega')$. If $a$ is the longest side, we have to be more careful. We multiply the law of cosines formula by $\cosh(a)$ and transform it into
\begin{align*}
\cosh(c)&=\cosh(a)\cosh(b)+\cosh(\omega)\sinh(a)\sinh(b)\,,\\
\cosh(a)\cosh(c)&=\cosh(a)^2\cosh(b)\\
&\quad+\cosh(\omega)\sinh(a)\cosh(a)\sinh(b)\,,\\
\cosh(a)\cosh(c)&=(1+\sinh(a)^2)\cosh(b)\\
&\quad+\cosh(\omega)\sinh(a)\cosh(a)\sinh(b)\,,\\
\cosh(a)\cosh(c)-\cosh(b)&=\sinh(a)(\sinh(a)\cosh(b)\\
&\quad+\cosh(\omega)\cosh(a)\sinh(b))\,.
\end{align*}
So the sign of the factor of $a'$ is minus the sign of the left-hand-side here (as $\sinh(a)>0$). Now we use the reverse triangle inequality $b\leq a-c$ and monotonicity of $\cosh$ to get $\cosh(a)\cosh(c)-\cosh(b)\geq \cosh(a)\cosh(c)-\cosh(a-c)=\sinh(a)\sinh(c)\geq0$, thus the factor of $a'$ is positive and $\sgn(a')=\sgn(\omega')$.

For $K=1$, we get the equation
\begin{equation*}
a'(\sin(a)\cos(b)+\sigma\cosh(\omega)\cos(a)\sin(b))=\omega'(-\sigma\sinh(\omega)\sin(a)\sin(b))\,.
\end{equation*}
If $c$ is the longest side, $\sigma=1$ and the signs of the factors are clear, $\sgn(a')=-\sgn(\omega')$. Otherwise $\sigma=-1$. If $b$ is the longest side, $\cos(a)\sin(b)-\sin(a)\cos(b)=\sin(b-a)\geq0$, thus $\sgn(a')=\sgn(\omega')$. If $a$ is the longest side, we have to be more careful. We multiply the law of cosines formula by $\cos(a)$ and transform it into
\begin{align*}
\cos(c)&=\cos(a)\cos(b)-\cosh(\omega)\sin(a)\sin(b)\,,\\
\cos(a)\cos(c)&=\cos(a)^2\cos(b)-\cosh(\omega)\sin(a)\cos(a)\sin(b)\,,\\
\cos(a)\cos(c)&=(1-\sin(a)^2)\cos(b)-\cosh(\omega)\sin(a)\cos(a)\sin(b)\,,\\
\cos(a)\cos(c)-\cos(b)&=-\sin(a)(\sin(a)\cos(b)+\cosh(\omega)\cos(a)\sin(b))\,.
\end{align*}
So the sign of the factor of $a'$ is minus the sign of the left-hand-side here ($0<a<\pi$, so $-\sin(a)<0$). Now we use the reverse triangle inequality $b\leq a-c$ and monotonicity of $\cos$ to get $\cos(a)\cos(c)-\cos(b)\leq \cos(a)\cos(c)-\cos(a-c)=-\sin(a)\sin(c)\leq0$, thus the factor of $a'$ is positive and $\sgn(a')=\sgn(\omega')$.
\end{proof}

\begin{proof}(of Lemma \ref{lorOneSidedCalcs})
The analyticity and order of convergence statements directly follow from inserting the power series of $\cos$, $\sin$, $\cosh$ and $\sinh$.

For (1), we consider the triangles $\Delta p_1p_2p_3$ and $\Delta p_1p_2q$: They share one side-length and the angle $\omega=\ma_{p_2}^{\lm{K}}(p_1,p_3)=\ma_{p_2}^{\lm{K}}(p_1,q)$. We only do the $K=0$ case since the other cases work analogously. By the law of cosines \ref{lorLawOfCosines}, we get the two equations
\begin{align*}
a^2+(b+c)^2&=d^2-2a(b+c)\cosh(\omega)\,,\\
a^2+b^2&=x^2-2ab\cosh(\omega)\,.
\end{align*}
Canceling $\omega$ and solving for $x$, we get
%\begin{align*}
%a^2b+b(b+c)^2&=bd^2-2ab(b+c)\cosh(\omega)\\
%a^2(b+c)+b^2(b+c)&=x^2(b+c)-2ab(b+c)\cosh(\omega)
%\end{align*}
%\[
%-a^2c+b(b+c)((b+c)-b)=bd^2-x^2(b+c)
%\]
%\[
%x^2(b+c)=bd^2+a^2c-bc(b+c)
%\]
\[
x^2=\frac{bd^2+a^2c}{b+c}-bc=\frac{a^2+d^2}{2}+\frac{(a^2-d^2)(c-b)}{2(b+c)}-bc\,.
\]

Note (2) is just the time-reversed statement of (1).

For (3), we have two cases: $q\ll p_2$ and $p_2\ll q$. We consider the triangles $\Delta p_1p_2p_3$ and $\Delta p_1qp_2$: They share one side-length and the angle $\omega={\ma}_{p_1}^{\lm{K}}(p_2,p_3)={\ma}_{p_1}^{\lm{K}}(q,p_2)$ agrees by the law of cosines, which does not distinguish the two cases. First, we establish the $K=0$ case. By the law of cosines \ref{lorLawOfCosines}, we get two equations
\begin{align*}
a^2+(c+d)^2&=b^2+2a(c+d)\cosh(\omega)\,,\\
a^2+c^2&=x^2+2ac\cosh(\omega)\,.
\end{align*}
Canceling $\omega$ and solving for $x$, we get
%\begin{align*}
%a^2c+c(c+d)^2&=b^2c+2ac(c+d)\cosh(\omega)\\
%a^2(c+d)+c^2(c+d)&=x^2(c+d)+2ac(c+d)\cosh(\omega)
%\end{align*}
%\[
%-a^2d+cd(c+d)&=b^2c-x^2(c+d)
%\]
\[
x^2=\frac{b^2c+a^2d}{c+d}-cd\,.
\]

Now for the case $K>0$. By the law of cosines \ref{lorLawOfCosines}, we get the two equations (where $s=\sqrt{|K|}$)
\begin{align*}
\cosh(sb)&=\cosh(sa)\cosh(s(c+d))+\sinh(sa)\sinh(s(c+d))\cosh(\omega)\,,\\
\cosh(sx)&=\cosh(sa)\cosh(sc)+\sinh(sa)\sinh(sc)\cosh(\omega)\,.
\end{align*}
Canceling $\omega$ and solving for $x$, we get
%\begin{align*}
%\cosh(sb)\sinh(sc)&=\cosh(sa)\cosh(s(c+d))\sinh(sc)+\sinh(sa)\sinh(sc)\sinh(s(c+d))\cosh(\omega)\\
%\cosh(sx)\sinh(s(c+d))&=\cosh(sa)\cosh(sc)\sinh(s(c+d))+\sinh(sa)\sinh(sc)\sinh(s(c+d))\cosh(\omega)
%\end{align*}
%\[
%\cosh(sb)\sinh(sc)-\cosh(sx)\sinh(s(c+d))=\cosh(sa)\cosh(s(c+d))\sinh(sc)-\cosh(sa)\cosh(sc)\sinh(s(c+d))
%\]
\begin{align}
\cosh(sx)&=\frac{\cosh(sb)\sinh(sc)-\cosh(sa)\cosh(s(c+d))\sinh(sc)}{\sinh(s(c+d))}\\
&\qquad+\cosh(sa)\cosh(sc)\,.
\end{align}

Now for the case $K<0$. By the law of cosines \ref{lorLawOfCosines}, we get the two equations (where $s=\sqrt{|K|}$)
\begin{align*}
\cos(sb)&=\cos(sa)\cos(s(c+d))+\sin(sa)\sin(s(c+d))\cosh(\omega)\,,\\
\cos(sx)&=\cos(sa)\cos(sc)+\sin(sa)\sin(sc)\cosh(\omega)\,.
\end{align*}
Canceling $\omega$ and solving for $x$, we get
%\begin{align*}
%\cos(sb)\sin(sc)&=\cos(sa)\cos(s(c+d))\sin(sc)+\sin(sa)\sin(sc)\sin(s(c+d))\cosh(\omega)\\
%\cos(sx)\sin(s(c+d))&=\cos(sa)\cos(sc)\sin(s(c+d))+\sin(sa)\sin(sc)\sin(s(c+d))\cosh(\omega)
%\end{align*}
%\[
%\cos(sb)\sin(sc)-\cos(sx)\sin(s(c+d))=\cos(sa)\cos(s(c+d))\sin(sc)-\cos(sa)\cos(sc)\sin(s(c+d))
%\]
\[
\cos(sx)=\frac{\cos(sb)\sin(sc)-\cos(sa)\cos(s(c+d))\sin(sc)}{\sin(s(c+d))}+\cos(sa)\cos(sc)\,.
\]
\end{proof}

\phantomsection
\addcontentsline{toc}{section}{References}
\bibliographystyle{alphaabbr}
\bibliography{Master}

\begin{thebibliography}{BMdOS22}

\bibitem[AB08]{AB:08}
S.~B. Alexander and R.~L. Bishop.
\newblock Lorentz and semi-{R}iemannian spaces with {A}lexandrov curvature
  bounds.
\newblock {\em Comm. Anal. Geom.}, 16(2):251--282, 2008.

\bibitem[AB22]{AB:22}
B.~Allen and A.~Burtscher.
\newblock Properties of the null distance and spacetime convergence.
\newblock {\em Int. Math. Res. Not. IMRN}, (10):7729--7808, 2022.

\bibitem[ACS20]{ACS:20}
L.~{Ak{\'{e}} Hau}, A.~J. {Cabrera Pacheco}, and D.~A. Solis.
\newblock On the causal hierarchy of {L}orentzian length spaces.
\newblock {\em Classical and Quantum Gravity}, 37(21):215013, 2020.

\bibitem[AGKS21]{AGKS:21}
S.~B. Alexander, M.~Graf, M.~Kunzinger, and C.~S\"amann.
\newblock Generalized cones as {L}orentzian length spaces: Causality,
  curvature, and singularity theorems.
\newblock {\em Comm.\ Anal.\ Geom., to appear}, 2021.
\newblock arXiv:1909.09575 [math.MG].

\bibitem[AH98]{AH:98}
L.~Andersson and R.~Howard.
\newblock Comparison and rigidity theorems in semi-{R}iemannian geometry.
\newblock {\em Comm. Anal. Geom.}, 6(4):819--877, 1998.

\bibitem[AKP22]{AKP:22}
S.~Alexander, V.~Kapovitch, and A.~Petrunin.
\newblock Alexandrov geometry: foundations.
\newblock 2022.
\newblock arXiv:1903.08539 [math.DG].

\bibitem[BBI01]{BBI:01}
D.~Burago, Y.~Burago, and S.~Ivanov.
\newblock {\em A course in metric geometry}, volume~33 of {\em Graduate Studies
  in Mathematics}.
\newblock American Mathematical Society, Providence, RI, 2001.

\bibitem[Ber20]{Ber:20}
T.~Beran.
\newblock {Lorentzian length spaces}.
\newblock \textit{Master's thesis, University of Vienna. Available at
  \url{http://othes.univie.ac.at/61080/}}, 2020.

\bibitem[BFK98]{BFK:98}
D.~Burago, S.~Ferleger, and A.~Kononenko.
\newblock A geometric approach to semi-dispersing billiards.
\newblock {\em Ergodic Theory Dynam. Systems}, 18(2):303--319, 1998.

\bibitem[BGH21]{BGH:21}
A.~Burtscher and L.~Garc{\'\i}a-Heveling.
\newblock Time functions on {L}orentzian length spaces.
\newblock {\em Preprint, arXiv:2108.02693 [gr-qc]}, 2021.

\bibitem[BH99]{BH:99}
M.~R. Bridson and A.~Haefliger.
\newblock {\em Metric spaces of non-positive curvature}, volume 319 of {\em
  Grundlehren der Mathematischen Wissenschaften [Fundamental Principles of
  Mathematical Sciences]}.
\newblock Springer-Verlag, Berlin, 1999.

\bibitem[BLMS87]{BLMS:87}
L.~Bombelli, J.~Lee, D.~Meyer, and R.~D. Sorkin.
\newblock Space-time as a causal set.
\newblock {\em Phys. Rev. Lett.}, 59(5):521--524, 1987.

\bibitem[BMdOS22]{BMS:22}
W.~Barrera, L.~Montes~de Oca, and D.~A. Solis.
\newblock Comparison theorems for {L}orentzian length spaces with lower
  timelike curvature bounds.
\newblock {\em General Relativity and Gravitation}, 54(9), 2022.

\bibitem[BORS22]{BORS:22}
T.~Beran, A.~Ohanyan, F.~Rott, and D.~Solis.
\newblock The splitting theorem for globally hyperbolic {L}orentzian length
  spaces with non-negative timelike curvature.
\newblock {\em preprint, arXiv:2209.14724 [math.DG]}, 2022.

\bibitem[BR22]{BR:22}
T.~Beran and F.~Rott.
\newblock Gluing constructions for {L}orentzian length spaces.
\newblock {\em preprint, arXiv:2201.09695 [math.DG]}, 2022.

\bibitem[Bus67]{Bus:67}
H.~Busemann.
\newblock Timelike spaces.
\newblock {\em Dissertationes Math. Rozprawy Mat.}, 53:52, 1967.
\newblock ISSN: 0012-3862.

\bibitem[CG12]{CG:12}
P.~T. Chru{\'s}ciel and J.~D.~E. Grant.
\newblock On {L}orentzian causality with continuous metrics.
\newblock {\em Classical Quantum Gravity}, 29(14):145001, 32, 2012.

\bibitem[CM20]{CM:20}
F.~Cavalletti and A.~Mondino.
\newblock Optimal transport in {L}orentzian synthetic spaces, synthetic
  timelike {R}icci curvature lower bounds and applications.
\newblock {\em preprint, arXiv:2004.08934 [math.MG]}, 2020.

\bibitem[Fin18]{Fin:18}
F.~Finster.
\newblock Causal {F}ermion {S}ystems: A {P}rimer for {L}orentzian {G}eometers.
\newblock {\em Journal of Physics: Conference Series}, 968(1):012004, 2018.

\bibitem[FS12]{FS:12}
A.~Fathi and A.~Siconolfi.
\newblock On smooth time functions.
\newblock {\em Math. Proc. Cambridge Philos. Soc.}, 152(2):303--339, 2012.

\bibitem[GGKS18]{GGKS:18}
M.~Graf, J.~D.~E. Grant, M.~Kunzinger, and R.~Steinbauer.
\newblock The {H}awking--{P}enrose {S}ingularity {T}heorem for
  {$C^{1,1}$}-{L}orentzian {M}etrics.
\newblock {\em Comm. Math. Phys.}, 360(3):1009--1042, 2018.

\bibitem[GKS19]{GKS:19}
J.~D.~E. Grant, M.~Kunzinger, and C.~S\"{a}mann.
\newblock Inextendibility of spacetimes and {L}orentzian length spaces.
\newblock {\em Ann. Global Anal. Geom.}, 55(1):133--147, 2019.

\bibitem[GKSS20]{GKSS:20}
J.~D.~E. Grant, M.~Kunzinger, C.~S\"{a}mann, and R.~Steinbauer.
\newblock The future is not always open.
\newblock {\em Lett. Math. Phys.}, 110(1):83--103, 2020.

\bibitem[GL17]{GL:17}
G.~J. Galloway and E.~Ling.
\newblock Some remarks on the {$C^0$}-(in)extendibility of spacetimes.
\newblock {\em Ann. Henri Poincar\'{e}}, 18(10):3427--3447, 2017.

\bibitem[GLS18]{GLS:18}
G.~J. Galloway, E.~Ling, and J.~Sbierski.
\newblock Timelike completeness as an obstruction to {$C^0$}-extensions.
\newblock {\em Comm. Math. Phys.}, 359(3):937--949, 2018.

\bibitem[Gra20]{Gra:20}
M.~Graf.
\newblock Singularity theorems for {$C^1$}-{L}orentzian metrics.
\newblock {\em Comm. Math. Phys.}, 378(2):1417--1450, 2020.

\bibitem[Har82]{Har:82}
S.~G. Harris.
\newblock A triangle comparison theorem for {L}orentz manifolds.
\newblock {\em Indiana Univ. Math. J.}, 31(3):289--308, 1982.

\bibitem[HP70]{HP:70}
S.~W. Hawking and R.~Penrose.
\newblock The singularities of gravitational collapse and cosmology.
\newblock {\em Proc. Roy. Soc. London Ser. A}, 314:529--548, 1970.

\bibitem[Kir18]{Kir:18}
M.~Kirchberger.
\newblock {Lorentzian Comparison Geometry}.
\newblock \textit{Master's thesis, University of Vienna. Available at
  \url{http://othes.univie.ac.at/56285/}}, 2018.

\bibitem[KOSS22]{KOSS:22}
M.~Kunzinger, A.~Ohanyan, B.~Schinnerl, and R.~Steinbauer.
\newblock The {H}awking-{P}enrose singularity theorem for {$C^1$}-{L}orentzian
  metrics.
\newblock {\em Comm. Math. Phys.}, 391(3):1143--1179, 2022.

\bibitem[KP67]{KP:67}
E.~H. Kronheimer and R.~Penrose.
\newblock On the structure of causal spaces.
\newblock {\em Proc. Cambridge Philos. Soc.}, 63:481--501, 1967.

\bibitem[KS18]{KS:18}
M.~Kunzinger and C.~S\"amann.
\newblock Lorentzian length spaces.
\newblock {\em Ann.\ Glob.\ Anal.\ Geom.}, 54(3):399--447, 2018.

\bibitem[KS22]{KS:22}
M.~Kunzinger and R.~Steinbauer.
\newblock Null distance and convergence of {L}orentzian length spaces.
\newblock {\em Annales Henri Poincar{\'e}}, 23(12):4319--4342, 2022.

\bibitem[LLS21]{LLS:21}
C.~Lange, A.~Lytchak, and C.~S\"{a}mann.
\newblock Lorentz meets {L}ipschitz.
\newblock {\em Adv. Theor. Math. Phys.}, 25(8):2141--2170, 2021.

\bibitem[LV09]{LV:09}
J.~Lott and C.~Villani.
\newblock Ricci curvature for metric-measure spaces via optimal transport.
\newblock {\em Ann. of Math. (2)}, 169(3):903--991, 2009.

\bibitem[McC20]{McC:20}
R.~McCann.
\newblock Displacement concavity of {B}oltzmann's entropy characterizes
  positive energy in general relativity.
\newblock {\em Camb. J. Math.}, 8(3):609--681, 2020.

\bibitem[Min19]{Min:19a}
E.~Minguzzi.
\newblock Causality theory for closed cone structures with applications.
\newblock {\em Rev. Math. Phys.}, 31(5):1930001, 139, 2019.

\bibitem[MS22a]{McCS:22}
R.~J. McCann and C.~S\"amann.
\newblock A {L}orentzian analog for {H}ausdorff dimension and measure.
\newblock {\em Pure and Applied Analysis}, 4(2):367--400, 2022.

\bibitem[MS22b]{MS:22}
A.~Mondino and S.~Suhr.
\newblock An optimal transport formulation of the {E}instein equations of
  general relativity.
\newblock {\em Journal of the European Mathematical Society (JEMS), to appear},
  2022.

\bibitem[O'N83]{ONe:83}
B.~O'Neill.
\newblock {\em Semi-{R}iemannian geometry with applications to relativity},
  volume 103 of {\em Pure and Applied Mathematics}.
\newblock Academic Press, Inc. [Harcourt Brace Jovanovich, Publishers], New
  York, 1983.

\bibitem[S{\"a}m16]{Sae:16}
C.~S{\"a}mann.
\newblock Global hyperbolicity for spacetimes with continuous metrics.
\newblock {\em Ann. Henri Poincar\'e}, 17(6):1429--1455, 2016.

\bibitem[Sbi18]{Sbi:18}
J.~Sbierski.
\newblock The {$C^0$}-inextendibility of the {S}chwarzschild spacetime and the
  spacelike diameter in {L}orentzian geometry.
\newblock {\em J. Differential Geom.}, 108(2):319--378, 2018.

\bibitem[Shi93]{Shi:93}
K.~Shiohama.
\newblock {\em An introduction to the geometry of {A}lexandrov spaces},
  volume~8 of {\em Lecture Notes Series}.
\newblock Seoul National University, Research Institute of Mathematics, Global
  Analysis Research Center, Seoul, 1993.

\bibitem[Stu06a]{Stu:06a}
K.-T. Sturm.
\newblock On the geometry of metric measure spaces. {I}.
\newblock {\em Acta Math.}, 196(1):65--131, 2006.

\bibitem[Stu06b]{Stu:06b}
K.-T. Sturm.
\newblock On the geometry of metric measure spaces. {II}.
\newblock {\em Acta Math.}, 196(1):133--177, 2006.

\bibitem[SV16]{SV:16}
C.~Sormani and C.~Vega.
\newblock Null distance on a spacetime.
\newblock {\em Classical Quantum Gravity}, 33(8):085001, 29, 2016.

\end{thebibliography}

\end{document}